\documentclass[leqno]{article}
\usepackage{palatino} 
\usepackage{verbatim,amsmath,amsfonts,amssymb,array,theorem}
\usepackage[mathbf,mathcal]{euler}
\usepackage{pifont,fancyheadings_ssv,wrapfig,graphicx}
\usepackage{epstopdf}
\usepackage[usenames,dvipsnames]{xcolor}
\usepackage{tikz, tikz-cd}
\usepackage{extarrows}
\usepackage{enumitem}
\usepackage{mathtools}
\usepackage[top=1in, bottom=1in, outer=1in, inner=1in, heightrounded, marginparwidth=1.5cm, marginparsep=1cm]{geometry}
 \usepackage[savepos]{zref}
\usetikzlibrary{matrix, calc, positioning}
\usepackage{thmtools}
\usepackage{thm-restate}
\usepackage{pgfplots}
\usepackage{subcaption}
\usepackage{hyperref}
\usepackage{bm}
\usepackage{dsfont}

\renewcommand{\emptyset}{\mathop{\varnothing}} 

\newcommand{\qed}{\mbox{\hfill$\square$}}

\newcommand{\bigslant}[2]{{\raisebox{.2em}{$#1$}\left/\raisebox{-.2em}{$#2$}\right.}}

\newcommand{\im}{\textnormal{im}}
\newcommand{\ZZ}{\mathbb{Z}}

\newcommand{\RR}{\mathbb{R}}

\newcommand{\CC}{\mathbb{C}}
\newcommand{\NN}{\mathbb{N}}

\newcommand{\dd}{\partial}

\newcommand{\la}{\langle}
\newcommand{\ra}{\rangle}
\newcommand{\cg}[1]{{\cal #1}}
\newcommand{\Int}{\textnormal{Int}}
\newcommand{\Crit}{\textnormal{Crit}}
\newcommand{\coker}{\textnormal{coker}}

\newcounter{acount}

\newcounter{Example}
\newenvironment{example}%
    {\renewcommand{\theExample}{\arabic{Example}}\refstepcounter{Example}%
     \begin{list}{}%
        {\usecounter{Examples}%
         \setlength{\labelsep}{0pt}\setlength{\leftmargin}{0pt}%
         \setlength{\labelwidth}{0pt}\setlength{\listparindent}{0.5in}%
        }%
     \item[\normalfont\textsc{Example} \arabic{Example}) ]%
     \renewcommand{\theExample}{\arabic{Example}}%
    }
    {\end{list}}
\newcounter{Examples}
\newcommand{\firstexample}%
    {\renewcommand{\Lbl}{\textsc{Examples: }}%
     \setcounter{Examples}{\theExample}%
    }

    {\renewcommand{\theExample}{\arabic{Example}}%
     \setcounter{Examples}{\theExample}%
     \refstepcounter{Examples}%
     \begin{list}{\Lbl\arabic{Examples}) }%
        {\usecounter{Examples}%
         \setlength{\labelsep}{0pt}\setlength{\leftmargin}{0pt}%
         \setlength{\labelwidth}{0pt}\setlength{\listparindent}{0.5in}%
        }%
    }%
    {\setcounter{Example}{\theExamples}%
     \renewcommand{\theExample}{\arabic{Example}}\end{list}%
    }   
\newcommand{\Lbl}{}                            

\newenvironment{proof}{%
    \medskip\noindent\textsc{Proof}:
  }{
    \hfill\qed\medskip
  }
\newenvironment{proofprop2}{%
    \medskip\noindent\textsc{Proof of Proposition 2}:
  }{
    \hfill\qed\medskip
  }  
\newenvironment{proofthm1}{%
    \medskip\noindent\textsc{Proof of Theorem 1}:
  }{
    \hfill\qed\medskip
  }  
  
\newenvironment{proofthm8}{%
    \medskip\noindent\textsc{Proof of Theorem 8}:
  }{
    \hfill\qed\medskip
  }  
  
\newenvironment{proofprop4}{%
    \medskip\noindent\textsc{Proof of Proposition 4}:
  }{
    \hfill\qed\medskip
  }  
\newcounter{bbean}

\newenvironment{closedstringms}{%
    \medskip\noindent\textsc{Closed-string mirror symmetry}
  }{
    \medskip
    }
  \newcounter{Remarks}
  \newenvironment{remark}%
    {\renewcommand{\theRemark}{\arabic{Remark}}\refstepcounter{Remark}%
     \begin{list}{}%
        {\usecounter{Remarks}%
         \setlength{\labelsep}{0pt}\setlength{\leftmargin}{0pt}%
         \setlength{\labelwidth}{0pt}\setlength{\listparindent}{0.5in}%
        }%
     \item[\normalfont\textsc{Remark} \arabic{Remark}) ]%
     \renewcommand{\theRemark}{\arabic{Remark}}%
    }
    {\end{list}}
  \newcounter{Remark}
\newcommand{\firstremark}%
    {\renewcommand{\Lbl}{\textsc{Remarks: }}%
     \setcounter{Remarks}{\theRemark}%
    }

    {\renewcommand{\theRemark}{\arabic{Remark}}%
     \setcounter{Remarks}{\theRemark}%
     \refstepcounter{Remarks}%
     \begin{list}{\Lbl\arabic{Remarks}) }%
        {\usecounter{Remarks}%
         \setlength{\labelsep}{0pt}\setlength{\leftmargin}{0pt}%
         \setlength{\labelwidth}{0pt}\setlength{\listparindent}{0.5in}%
        }%
    }%
    {\setcounter{Remark}{\theRemarks}%
     \renewcommand{\theRemark}{\arabic{Remark}}\end{list}%
    }   
\theoremstyle{plain}
\newtheorem{corollary}{Corollary}              
{\theorembodyfont{\rmfamily}
}
{\theorembodyfont{\rmfamily} }
\newtheorem{lemma}{Lemma}                      
\newtheorem{proposition}{Proposition}          
\newtheorem{theorem}{Theorem}                  

{\theorembodyfont{\rmfamily} }
{\theorembodyfont{\rmfamily} }
{\theorembodyfont{\rmfamily} }

\theoremheaderfont{\scshape}
\newtheorem{thm}{Theorem}
 \hypersetup{
    colorlinks,
    citecolor=blue,
    filecolor=blue,
    linkcolor=blue,
    urlcolor=blue
}
\definecolor{frenchblue}{rgb}{0.0, 0.45, 0.73}

\tikzcdset{arrow style=tikz}

\begin{document}
\title{The quantitative nature of reduced Floer theory}
\author{Sara Venkatesh}
\date{}
\maketitle
\abstract{We study the reduced symplectic cohomology of disk subbundles in negative symplectic line bundles.  We show that this cohomology theory ``sees'' the spectrum of a quantum action on quantum cohomology.  Precisely, quantum cohomology decomposes into generalized eigenspaces of the action of the first Chern class by quantum cup product.  The reduced symplectic cohomology of a disk bundle of radius $R$ sees all eigenspaces whose eigenvalues have size less than $R$, up to rescaling by a fixed constant.  Similarly, we show that the reduced symplectic cohomology of an annulus subbundle between radii $R_1$ and $R_2$ captures all eigenspaces whose eigenvalues have size between $R_1$ and $R_2$, up to a rescaling.  We show how local closed-string mirror symmetry statements follow from these computations.}
\tableofcontents
\section{Introduction}
Symplectic cohomology is a tool for delving the geometry of an open symplectic manifold.  It was initially studied by Cieliebak, Floer, Hofer, and Wysocki to probe quantitative aspects of domains in $\RR^{2n}$; applications were focused on embedding problems and capacities.  The focus on quantitative geometry changed when Viterbo introduced a ``qualitative'' version of symplectic cohomology \cite{viterbo}.  This qualitative definition has proved indispensable to the study of global symplectic properties, from attacking classification problems to analyzing Lagrangian embeddings \cite{seidel-biased}.  The open symplectic manifolds considered have been, for the most part, {\it Liouville manifolds}: manifolds that contain the entire symplectization $\RR\times\Sigma$ of a contact manifold $\Sigma$.

Very little has been done to understand symplectic cohomology away from the Liouville setting.  The first attempt is due to Ritter, who computed the symplectic cohomology of symplectic line bundles satisfying a negativity condition \cite{ritter-gromov}.  Rather than containing an entire symplectization, these line bundles only contain the positive ``piece'' of a symplectization: they are compactifications of the space $(0, \infty)\times\Sigma$.  Ritter's result surprisingly tied the symplectic cohomology of a line bundle $E$, denoted by $SH^*(E)$, to its quantum cohomology.  The main theorem in \cite{ritter-gromov} shows that symplectic cohomology sees almost all of quantum cohomology; it misses only the zeroth generalized eigenspace, denoted by $QH^*_0(E)$, of a particular action on quantum cohomology.  Precisely, there is a ring isomorphism
\begin{equation} 
\label{eq:ritter}
SH^*(E) \simeq \bigslant{QH^*(E)}{QH^*_0(E)}.
\end{equation}
The action in question is given by quantum cup product with the pull-back of the first Chern class of the the line bundle $E$.

In \cite{venkatesh} we showed how to refine Viterbo's qualitative symplectic cohomology to produce a quantitative invariant.  A chain complex computing symplectic cohomology has a natural family of non-Archimedean metrics; one can complete the chain complex with respect to a fixed metric and produce a reduced cohomology theory, \`a la $L^2$-cohomology \cite{dai}.  We call the resulting cohomology theory {\it reduced symplectic cohomology}.  Each fixed metric encodes information about the size of the Reeb orbits on a fixed contact hypersurface.  In line bundles, these fixed contact hypersurfaces are circle subbundles, and they come in an $\RR_{>0}$ family, indexed by radius, that induces an $\RR_{>0}$ family of metrics.  We write $\widehat{SH^*}(D_R)$ for the reduced symplectic cohomology theory associated to the radius $R$, where $D_R$ is the disk subbundle of radius $R$.  

In \cite{venkatesh-thesis} we studied reduced symplectic cohomology on {\it monotone toric} line bundles.  These line bundles were previously studied by Ritter-Smith in the context of wrapped Fukaya categories \cite{ritter-s}.  They showed that the wrapped Fukaya category of a monotone, toric line bundle is split-generated by a single Lagrangian torus $L$ lying in a particular circle subbundle.  In \cite{venkatesh-thesis}, we showed that
\begin{equation} 
\label{eq:priorwork}
\widehat{SH^*}(D_R) \neq 0 \iff L\subset D_R.
\end{equation}
Thus, reduced symplectic cohomology contains local information about the disk bundle $D_R$.

It is not known what the Fukaya category looks like away from the monotone toric case.  However, we can still hope to say something about reduced symplectic cohomology.  Take coefficients in the {\it Novikov field over $\CC$}, denoted by $\Lambda$ and defined by formal series
\[
\Lambda = \left\{\sum_{i=0}^{\infty}c_iT^{\alpha_i}\hspace{.1cm}\big|\hspace{.1cm} c_i\in \CC, \RR\ni\alpha_i\rightarrow\infty\right\}.
\]
Quantum cohomology splits into generalized eigenspaces
\[
QH^*(E) = \bigoplus_{\lambda}QH^*_{\lambda}(E),
\]
where $QH^*_{\lambda}(E)$ is the $\lambda$-generalized eigenspace of the pullback of $c_1^E$ acting by quantum cup product.  Each non-zero eigenvalue $\lambda$ has a ``size'', denoted by $ev(\lambda)$, given by the natural valuation on the Novikov field (see (\ref{eq:val})).

The full reduced symplectic cohomology is difficult to compute, and it may have subtle invariance.  We work instead with a modified version, denoted by $\overline{SH^*}(D_R)$, which is formed by completing only the subspace spanned by the kernel of the differential.  The main result of this paper is the following Theorem.

\begin{restatable}{theorem}{thmmain}\label{thmlabel}
\label{thm:overlinesh}
Let $E$ be a weak\textsuperscript{+} monotone, negative, symplectic line bundle with negativity constant $k$.  The reduced symplectic cohomology of a disk subbundle $D_R\subset E$ is isomorphic as a $\Lambda$-algebra to
\[
\overline{SH^*}(D_R) \simeq \bigslant{QH^*(E)}{\bigoplus\limits_{ev(\lambda) > k\pi R^2}QH^*_{\lambda}(E)}.
\]
\end{restatable}
Here, the negativity constant $k$ determines the line bundle $E$ and affects the symplectic structure.  The weak\textsuperscript{+} monotonicity is a standard technical assumption.

\begin{remark}
Theorem 1 generalizes Ritter's result (\ref{eq:ritter}), which may be rephrased as a computation of the reduced symplectic cohomology of the disk bundle of infinite radius.  Note that, in this non-Archimedean setting, $ev(0) = \infty$.
\end{remark}

If $E$ additionally lies over a toric base, there is an isomorphism
\[
\widehat{SH^*}(D_R) \simeq \overline{SH^*}(D_R)
\]
(see Proposition \ref{prop:comequalsred}).

\begin{remark} If $E$ is monotone over a toric base, all non-zero eigenvalues have the same size.  For $\lambda\neq 0$, the split-generating Lagrangian $L$ lies in a circle bundle of radius $R = \frac{1}{\sqrt{k \pi\cdot ev(\lambda)}}$ \cite{ritter-fano}.  Thus, Theorem 1 is also a generalization of the equivalence (\ref{eq:priorwork}).
\end{remark}
The full theory $\widehat{SH^*}(D_R)$ has an extension to the {\it symplectic cohomology of a trivial cobordism}, initially defined and studied in \cite{cieliebak-o} and \cite{c-f-o}.  In the case of a negative line bundle, the trivial cobordisms are the annulus subbundles.  We denote an annulus subbundle between radii $R_1 < R_2$ by $A_{R_1, R_2}$.  Completed symplectic cohomology for a trivial cobordism yields an invariant $\widehat{SH^*}(A_{R_1, R_2})$, defined in \cite{venkatesh}.

\begin{theorem}
If $E$ is furthermore a line bundle over a toric base, there is a vector-space isomorphism
\begin{equation}
\label{eq:shann}
\widehat{SH^*}(A_{R_1, R_2}) \simeq \bigslant{QH^*(E)}{\bigoplus\limits_{\substack{ev(\lambda) > k\pi R_2^2 \\ \text{or} \\ ev(\lambda) \leq k \pi R_1^2}}QH^*_{\lambda}(E)}.
\end{equation}
\end{theorem}

\begin{remark}
The main result of \cite{c-f-o} shows that, in the Liouville setting, the uncompleted symplectic cohomology of a trivial cobordism $(-\epsilon, \epsilon)\times\Sigma$ is isomorphic to the Rabinowitz Floer homology of $\Sigma$.  Expounding upon this result, the reduced theory $\widehat{SH^*}(A_{R_1, R_2})$ is conjecturally related to the Rabinowitz Floer theory developed by Albers-Kang for contact-type hypersurfaces in negative line bundles \cite{albers-k}.
\end{remark}

\begin{remark}
The chain complex computing $\widehat{SH^*}(A_{R_1, R_2})$ considers positively-traversed Reeb orbits at radius $R_2$ and negatively-traversed Reeb orbits at radius $R_1$.  One could just as easily switch the roles of $R_1$ and $R_2$, instead considering the negatively-traversed Reeb orbits at radius $R_2$ and the positively-traversed Reeb orbits at radius $R_1$.  Abusing notation, we write $\widehat{SH^*}(A_{R_2, R_1})$ for the completed cohomology theory that switches these roles.  This theory is dual to $\widehat{SH^*}(A_{R_1, R_2})$ and satisfies 
\[
\widehat{SH^*}(A_{R_2, R_1}) \simeq \bigoplus_{k\pi R_1^2 < ev(\lambda) \leq k\pi R_2^2}QH_*^{\lambda}(E),
\]
where $QH^{\lambda}_*(E)$ is the $\lambda$-generalized eigenspace of $P.D.(c_1^E)$ acting by quantum intersection product on the quantum homology $QH_*(E)$.
\end{remark}

Toric line bundles $E$ have a Landau-Ginzburg mirror $(E^{\vee}, W)$, where $E^{\vee}$ is a rigid analytic space, and $W$ is an analytic function on $E^{\vee}$.  A closed-string mirror symmetry statement is proved for some cases in \cite{ritter-s}, and in the full monotone case in \cite{ritter-fano}.  Namely, if $E$ is monotone, then there is an isomorphism
\begin{equation}
\label{eq:csmsritter}
SH^*(E) \simeq Jac(W).
\end{equation}
A local statement of mirror symmetry for annulus subbundles follows as a corollary of Theorem 1 and the isomorphism (\ref{eq:csmsritter}).
\begin{corollary}
Let $R_1 < R_2$.  An annulus subbundle $A_{R_1, R_2}$ in a monotone, negative line bundle over a toric base has a Landau-Ginzburg mirror $(A_{R_1, R_2}^{\vee}, W)$ such that 
\begin{equation} 
\label{eq:csms-ann}
\widehat{SH^*}(A_{R_1, R_2})\simeq Jac(W).
\end{equation}
\end{corollary}
This proves a conjecture of the author made in \cite{venkatesh}.  We believe that the correspondence (\ref{eq:csms-ann}) holds more generally, but we do not know what the superpotential $W$ looks like away from the monotone case.

\subsection{Outline of paper}
In Section 2 we recall the definition of Hamiltonian Floer theory and symplectic cohomology.  We define and discuss reduced symplectic cohomology, and we discuss extensions to the symplectic cohomology of a cobordism.  In Section 3 we introduce negative line bundles and the specific Floer data that we will use to compute symplectic cohomology.  The entirety of Section 4 is devoted to proving Theorem 1.  In Section 5 we prove that, in some cases, completed and reduced symplectic cohomology coincide.  We then prove Theorem 2, restated as Theorem \ref{thm:rfh}, and discuss closed string mirror symmetry.

\subsection{Acknowledgements}
We thank Paul Seidel for motivating this project and Mohammed Abouzaid for helpful discussions.  We thank the Institute for Advanced Study for the productive environment under which the bulk of this paper was completed.  This work is based upon work supported by the National Science Foundation under Award No. 1902679.

\section{Floer theory}
\label{sec:sh}
We begin this section by recalling the data of Hamiltonian Floer theory and symplectic cohomology, fixing conventions along the way.  We refer the reader to \cite{audin-d} and \cite{mcduff-s} for an introduction to Floer theory, and we refer to \cite{seidel-biased} for details on symplectic cohomology.  We then introduce the notion of {\it reduced symplectic cohomology}, recently defined and studied by the author in \cite{venkatesh}, as well as in the independent works of Groman \cite{groman}, McLean \cite{mclean}, and Varolgunes \cite{varolgunes}.

Throughout this paper, we will consider a symplectic manifold $(E, \Omega)$ of dimension $2m$ that is {\it weak\textsuperscript{+} monotone}.  Abusing notation, denote by $\pi_2(E)$ the image of $\pi_2(E)$ in $H_2(E)$ under the Hurewicz homomorphism.  Practically, $(E, \Omega)$ is weak\textsuperscript{+} monotone if it satisfies at least one of four conditions on $\pi_2(E)$:
\begin{enumerate}
\item there is a positive constant $\kappa > 0$ such that $c_1^{TE} = \kappa[\Omega]$ on $\pi_2(E)$,
\item $c_1^{TE}\big|_{\pi_2(E)} = 0$,
\item $[\Omega]\big|_{\pi_2(E)} = 0$,
\item or the minimal Chern number, defined to be $$\min\limits_{\substack{A\in\pi_2(E)\\ c_1^{TE}(A) \neq 0}}|c_1^{TE}(A)|,$$ is at least $m-1$.
\end{enumerate}
This rather ad hoc collection of conditions ensures that moduli spaces are well-defined in both the definition of Floer theory and in the definition of the P.S.S. maps that relate Floer theory to quantum cohomology \cite{hofer-s}.

\subsection{Hamiltonian Floer theory}

Let $H:E\times S^1\longrightarrow \RR$ be a Hamiltonian.  $H$ induces a Hamiltonian vector field $X_H$, which is the unique vector field satisfying 
\[
dH(-) = \Omega(-, X_H).
\]
The smooth maps $x:S^1\longrightarrow E$ satisfying
\[
\dot{x} = X_H(x(t))
\]
are called the {\it periodic orbits} of $H$.  Assume that the periodic orbits of $H$ satisfy a generic non-degeneracy condition.  Denote the set of orbits by $\cg{P}(H)$.  Let $T$ be a formal variable, and define the {\it universal Novikov ring} over $\CC$
\begin{equation} 
\label{eq:novikov}
\Lambda = \left\{\sum_{i=0}^{\infty}c_iT^{\alpha_i}\hspace{.1cm}\big|\hspace{.1cm} c_i\in \CC, \RR\ni\alpha_i\rightarrow\infty\right\}.
\end{equation}
Define the Hamiltonian cochain complex $CF^*(H; \Lambda)$ to be a cochain complex with underlying vector space over $\Lambda$ generated by the periodic orbits of $X_H$:
\[
CF^*(H; \Lambda) := \bigoplus_{x\in\cg{P}(H)}\Lambda\la x\ra.
\]
Let $\tau$ be the minimal Chern number.  Grading is given by the {\it cohomological Conley-Zehnder index} and takes values in $\ZZ/2\tau\ZZ$ \cite{seidel-biased}.

To define the differential on $CF^*(H; \Lambda)$, fix a capping $\tilde{x}$ of each periodic orbit $x\in\cg{P}(H)$.  The {\it action} of $x$ is defined to be
\begin{equation} 
\label{eq:action}
\cg{A}_H(x) = -\int_{D} \tilde{x}^*\Omega + \int_0^1 H(x(t))dt.
\end{equation}

Let $J$ be an $\Omega$-tame almost-complex structure that is $\Omega$-compatible in neighborhoods of the periodic orbits of $X_H$.  Consider a solution $u:\RR_s\times S^1_t\longrightarrow E$ of Floer's equation
\begin{equation}
\label{eq:floer}
\frac{\dd u}{\dd s} + J\left(\frac{\dd u}{\dd t} - X_H\right) = 0.
\end{equation}
The {\it energy} of $u$ is defined to be
\begin{equation}
\label{eq:energy}
E(u) = \int_{\RR\times S^1} ||\dd_s u||^2 ds\wedge dt.
\end{equation}
A finite-energy Floer solution converges asymptotically in $s\rightarrow\pm\infty$ to orbits in $\cg{P}(H)$.  Fix $x_-, x_+\in\cg{P}(H)$ and denote by
\[
\widehat{\cg{M}}(x_-, x_+)
\]
the set of finite-energy Floer solutions with 
\[
\lim\limits_{s\rightarrow\pm\infty}u(s, t) = x_{\pm}(t).
\]
The moduli space $\widehat{\cg{M}}(x_-, x_+)$ is equipped with a free $\RR$-action whenever $x_-\neq x_+$ that translates the $s$-coordinate of a Floer solution.  Modding out by this $\RR$ action yields a moduli space
\[
\cg{M}(x_-, x_+),
\]
whose components can be indexed by their dimension.  The zero-dimensional component, denoted by $\cg{M}^0(x_-, x_+)$, is compact if $E$ is closed.  In this case, the Floer differential $\dd^{fl}$ is defined on $CF^*(H; \Lambda)$ by
\[
\dd^{fl}(y) = \sum_{x\in\cg{P}(H) }\sum_{u\in\cg{M}^0(x, y)} \pm T^{-[\Omega]([\tilde{y}\#(-u)\#(-\tilde{x})])}\la x\ra,
\]
where the sign is determined by fixing choices of orientations on each moduli space, and the expression $[\tilde{y}\#(-u)\#(-\tilde{x})]$ refers to the homology class of the sphere constructed by gluing $\tilde{y}$, $-u$, and $-\tilde{x}$ along their boundaries.  We refer to \cite{venkatesh-thesis} for a discussion of orientations in the context of this paper.

\subsection{Symplectic cohomology}
\label{subsec:sh}
Assume now that $E$ is an open symplectic manifold, which, outside of a compact set, is modeled on the symplectization
\begin{equation} 
\label{eq:sigma}
(0, \infty)_r\times\Sigma, \Omega = d(\pi r^2\alpha)
\end{equation}
of a contact manifold $\Sigma$ with contact form $\alpha$.
Fix a radius $R > 0$ and a monotone-increasing sequence $\tau_n\in\RR_{>0}$ with $\tau_n\rightarrow\infty$.  Let $\{H_n\}_{n\in\NN}$ be a family of Hamiltonians satisfying the following constraints.
\begin{enumerate}
    \item $|H_n|$ is bounded uniformly by a fixed constant $C>0$ on $E\setminus(R, \infty)\times\Sigma$,
    \item $H_n$ is linear in $\pi r^2$ with slope $\tau_n$  on $(R, \infty)\times\Sigma$, and
    \item the non-constant orbits of $H_n$ lie in some small neighborhood $(R - \frac{1}{n}, R)\times\Sigma$.
\end{enumerate}
The first condition ensures that the action of a periodic orbit is not heavily influenced by $H_n$, the second condition ensures that Floer solutions are well-behaved, and the final condition ensures that the set of periodic orbits $\bigcup\limits_n\cg{P}(H_n)$ cluster near the contact hypersurface $\{R\}\times\Sigma$.  Assume that each $H_n$ is chosen so that all one-periodic orbits of $H_n$ are non-degenerate.  Choose an almost-complex structure that is {\it cylindrical} on $(R, \infty)\times\Sigma$:
\[
J^*dr = -2\pi r\alpha.
\]
Under the assumptions on $E$, $\{H_n\}$, and $J$, the Floer cochain complex defined in Section 2.1 is well-defined.  

By assumption, $H_n \leq H_{n+1}$ outside a compact set of $E$.  Thus, there are chain maps, called {\it continuation maps},
\[
c_n:CF^*(H_n; \Lambda)\rightarrow CF^*(H_{n+1}; \Lambda)
\]
for each $n$.
Let ${\bf q}$ be a formal variable of degree $-1$, satisfying ${\bf q}^2 = 0$.  Define the symplectic cochain complex of $E$ to be
\begin{equation}
\label{eq:sh}
SC^*(H_R) = \bigoplus_{n=0}^{\infty}CF^*(H_n; \Lambda)[{\bf q}],
\end{equation}
with differential $\dd(x+y{\bf q}) = \dd^{fl}(x) + (c - id)(y) + \dd^{fl}(y){\bf q}$, where
\[
c = \bigoplus_{n=0}^{\infty}c_n.
\]  
The cohomology of $SC^*(H_R)$, denoted by $SH^*(H_R)$, is called the {\it symplectic cohomology of $E$}.

A standard result in Floer theory is
\begin{theorem}
The symplectic cohomology $SH^*(H_R)$ is independent of choice of radius $R$, family $\{H_n\}$, and cylindrical almost-complex structure $J$.
\end{theorem}
We will often write $SH^*(E)$ instead of $SH^*(H_R)$.

\subsection{Reduced symplectic cohomology}
\label{subsec:redsh}
Having fixed a definition of action in (\ref{eq:action}), the Floer cochain complex has a natural non-Archimedean metric.  To see this, first recall that the Novikov ring has a valuation
\begin{align}
\label{eq:val}
ev:\Lambda&\rightarrow \RR\cup\{\infty\} \\
\sum_{i=0}^{\infty}c_iT^{\alpha_i} &\mapsto \min_{c_i\neq 0}\alpha_i \hspace{.5cm}\text{if }\exists\hspace{.1cm} c_i\neq 0 \\
0 &\mapsto\infty. 
\end{align}
Define a valuation $\cg{A}$ on $CF^*(H_n)$ by
\[
\cg{A}\left(\sum_{x_i\in\cg{P}(H_n)}^j C_ix_i\right) = \min_i\left(ev(C_i) + \cg{A}_{H_n}(x_i)\right),
\]
where $C_i\in\Lambda$ are $\Lambda$-valued coefficients.  Extend this to a valuation on $SC^*(H)$ by
\[
\cg{A}\left(\sum_{i=0}^{j}X_i{\bf q}^{\ell_i}\right) = \min_i \cg{A}_{H_{n_i}}(X_i),
\]
where each $X_i$ lies in a fixed $CF^*(H_{n_i})$ and $\ell_i\in\{0, 1\}$.  The non-Archimedean metric $||\cdot||$ on $SC^*(H_R)$ is given by 
\[
||X|| = e^{-\cg{A}(X)}.
\]
Following \cite{groman}, denote by
$
\widehat{\ker(\dd)}
$
the completion of $\ker(\dd)\subset SC^*(H)$ with respect to $||\cdot ||$.  Denote by
$
\overline{\im(\dd)}
$
the closure of $\im(\dd)$ in $\widehat{\ker(\dd)}$.  The quotient
\[
\overline{SH^*}(H_R) := \bigslant{\widehat{\ker(\dd)}}{\overline{\im(\dd)}}
\]
is the {\it reduced symplectic cohomology} of the Floer data $\{H_n\}$.

An alternative definition of reduced symplectic cohomology, and one which is prevalent in the literature, is to first complete $SC^*(H)$ with respect to $||\cdot||$, and then define $\widehat{\dd}$ on $\widehat{SC^*}(H)$ to be the natural extension of $\dd$.  This defines a symplectic cohomology theory
\begin{equation}
    \label{eq:completedsh}
    \widehat{SH^*}(H_R) = \bigslant{\ker(\widehat{\dd})}{\overline{\im(\widehat{\dd})}}.
\end{equation}
To avoid confusion, we call $\widehat{SH^*}(H_R)$ the {\it completed symplectic cohomology}.

Completed symplectic cohomology has an equivalent definition as the algebraic completion induced by a filtration.  By choosing $H_n\leq H_{n+1}$ everywhere, the action is increased by both the Floer differential and continuation maps.  It therefore defines a filtration on $SC^*(H)$.  Define the subcomplex
\begin{equation} 
\label{eq:sub}
SC^*_{(a, \infty)}(H_R) = \big\la X\in SC^*(H_R) \hspace{.2cm}\bigg|\hspace{.2cm}\cg{A}(X) > a\big\ra, \dd_{(a, \infty)} 
\end{equation}
and quotient complex
\begin{equation} 
\label{eq:quotient}
SC^*_a(H_R) = \bigslant{SC^*(H_R)}{(SC^*_{(a, \infty)}(H_R)}, \dd_a
\end{equation}
These complexes are modules over the {\it positive Novikov ring}
\[
\Lambda_0 := ev^{-1}([0, \infty]]
\]
and define cohomology theories denoted, respectively, by $SH^*_{(a, \infty)}(H_R)$ and $SH^*_a(H_R)$.  Running over all $a\in\RR$, these theories form a directed, respectively inversely directed, system through the canonical inclusion, respectively projection at the level of cochains.  This gives an alternative description of completed symplectic cohomology, through a Theorem of Groman.
\begin{theorem}[Groman: Theorem 8.4 in \cite{groman}]
\label{thm:groman}
\[
\widehat{SH^*}(H_R) = \lim_{\substack{\leftarrow \\ a}}SH^*_a(H_R).
\]
\end{theorem}

\begin{remark}
\label{rem:countable}
In practice, we will take the inverse limit over a countable set $a_1, a_2, ... \rightarrow\infty$.  This simplifies computations.
\end{remark}

\begin{remark}
$\overline{SH^*}(H_R)$ and $\widehat{SH^*}(H_R)$ now depend upon the families $\{H_n\}$ used to define them.  As we will see, these cohomology theories are quantitative invariants encoding local information about domains contained in $E$.
\end{remark}

\begin{remark}
In this paper we study $\overline{SH^*}(H_R)$, rather than $\widehat{SH^*}(H_R)$, simply because this is the object that we can compute.  We give examples in Section \ref{sec:toric} of manifolds $E$ and families $\{H_n\}$ on $E$ for which
\begin{equation} 
\label{eq:reducedequalscomplete}
\overline{SH^*}(H_R) = \widehat{SH^*}(H_R).
\end{equation}
We do not know in how much generality (\ref{eq:reducedequalscomplete}) holds.  However, we believe $\overline{SH^*}(H_R)$ to be at least as robust an invariant as $\widehat{SH^*}(H_R)$.  While we have no Floer theory examples to support this belief, it is straight-forward to find examples of chain complexes for which these two theories disagree.   Consider the following example from Morse theory.

\begin{figure}[htbp!]
    \centering
    \begin{subfigure}{.45\textwidth}
    \includegraphics[scale=.6]{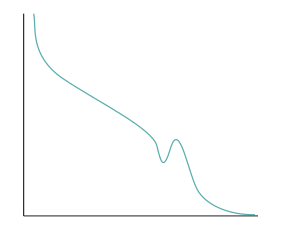}
    \end{subfigure}
    \begin{subfigure}{.45\textwidth}
    \includegraphics[scale=.6]{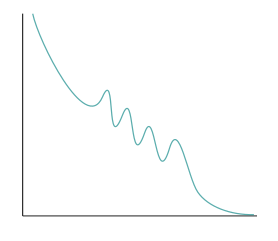}
    \end{subfigure}
    \caption{The functions $g_1$ and $g_4$}
    \label{fig:morse}
\end{figure}
Let $\{g_n:I\subset\RR\longrightarrow\RR\}$ be the family of Morse functions pictured in Figure \ref{fig:morse}.  Denote the Morse cochain complex of $g_n$ over a ring $R$ by $CM^*(g_n; R)$.  Choose continuation maps
\[
c_n: CM^*(g_n; R)\longrightarrow CM^*(g_{n+1}; R)
\]
that correspond to the inclusion of a subcomplex.  Analogously to symplectic cohomology, define a new Morse-type complex
\[
CM^*(g) = \bigoplus_{n=0}^{\infty}CM^*(g_n; R)[{\bf q}],
\]
where ${\bf q}$ is a formal variable of degree $-1$, satisfying ${\bf q}^2 = 0$.  Denoting the Morse differential by $\dd^M$, the differential on $SC^*(M)$ is given by
\[
\dd(x + y{\bf q}) = \dd^M(x) + (c - id)(y) + \dd^M(y){\bf q}.
\]
For $x\in \Crit(g_n)$, define
\[
\cg{A}(x) = g_n(x).
\]
The assignment $\cg{A}$ extends to a valuation on the Morse complex $CM^*(g)$ that, in turn, defines a non-Archimedean metric.  We use these metrics to define reduced and completed Morse cohomology theories, denoted by $\overline{HM^*}(g)$ and $\widehat{HM^*}(g)$.
It is straight-forward to see that
\[
\overline{HM^*}(g) = 0,
\]
but
\[
\widehat{HM^*}(g) = R.
\]
For any single $g_n$,
\[
\overline{HM^*}(g_n) = \widehat{HM^*}(g_n) = 0.
\]
Thus, in this contrived scenario, the reduced cohomology seems to behave stably, while the completed theory does not.
\end{remark}

\subsection{Rabinowitz Floer homology}
\label{subsec:rfh}
Recall that a Liouville cobordism $W$ is an exact symplectic manifold with contact-type boundary.  A boundary component is {\it positive} if the contact orientation agrees with the boundary orientation.  Otherwise, it is {\it negative}.  A {\it filled} Liouville cobordism is an exact symplectic manifold $D$ with positive contact boundary such that 
\begin{enumerate}
    \item $W$ symplectically embeds into $D$ and
    \item under this embedding, the boundary of $D$ is the positive boundary of $W$.
\end{enumerate}
The completed symplectic cohomology of a domain has a natural extension to a theory for filled cobordisms.  This theory was first defined in \cite{c-f-o} for trivial Liouville cobordism with Liouville filling and extended to general Liouville cobordism with Liouville filling in \cite{cieliebak-o}.  The definition was extended by the author to a completed theory in \cite{venkatesh}.  We briefly recall the construction for a trivial cobordism.

Define a ``dual'' symplectic homology theory
\[
SC_*(H_R) = \prod_{n=0}^{\infty}CF^*(-H_n)[\boldsymbol{q}].
\]
As with symplectic cohomology, $SC_*(H_R)$ admits a filtration by action, given by the chain complexes $CF^*_{(a, \infty)}(H_n)$.  Define subcomplexes
\[
SC_*^{(a, \infty)}(H_R) = \prod_{n=0}^{\infty}CF^*_{(a, \infty)}(-H_n)[{\bf q}]
\]
with homology $SH_*^{(a, \infty)}(H_R)$.  Action-completed symplectic homology is defined to be
\[
\widehat{SH_*}(H_R) = H\left(\lim_{\substack{\rightarrow \\ a}}SC_*^{(a, \infty)}(H_R)\right).
\]

Let $\{H_n'\}$ be a family of Hamiltonians corresponding to some radius $R'$, and let $\{H_n\}$, as always, be a family of Hamiltonians corresponding to some radius $R$.
We define a map
\[
\mathfrak{\widehat{c}}:\widehat{SC_*}(H_{R'})\longrightarrow \widehat{SC^*}(H_R),
\]
as follows.  There is a map
\[
\widehat{SC_*}(H_{R'})\longrightarrow SC_*(H_{R'})
\]
that is the direct limit of the inclusion maps
\[
SC_*^{(a, \infty)}(H_{R'}) \hookrightarrow SC_*(H_{R'})
\]
and a map
\[
SC_*(H_{R'})\longrightarrow CF^*(-H_0')
\]
that is projection onto the first component.  Similarly, there is a map
\[
SC^*(H_R)\longrightarrow\widehat{SC^*}(H_R)
\]
that is the inverse limit of the projection maps
\[
SC^*(H_R)\twoheadrightarrow SC^*_a(H_R)
\]
and a map
\[
CF^*(H_0)\hookrightarrow SC^*(H_R)
\]
that is inclusion into the first component.
The map $\mathfrak{\widehat{c}}$ is the composition
\[
\begin{tikzcd}
\widehat{SC_*}(H_{R'})\arrow{d} & \widehat{SC^*}(H_R) \\
SC_*(H_{R'}) \arrow{d}{} \arrow{r}{\mathfrak{c}} & SC^*(H_R) \arrow{u} \\
CF^*(-H_0') \arrow{r}{c} & CF^*(H_0) \arrow{u}{}
\end{tikzcd}
\]
that projects an infinite sum onto the ``$CF^*(-H_0')$'' component, maps this component onto $CF^*(H_0)$ via continuation, and finally includes into the symplectic cochain complex.

Without loss of generality assume that $R' \leq R$.  Define the symplectic cohomology of the cobordism $[R', R]\times\Sigma$ to be the cohomology of the cone of $\mathfrak{c}$, the latter written as
\[
SC^*([R', R]\times\Sigma) = Cone(\mathfrak{c}),
\]
and the completed symplectic cohomology to be the cohomology of the cone of $\widehat{\mathfrak{c}}$,
\[
\widehat{SC^*}([R', R]\times\Sigma) = Cone(\widehat{\mathfrak{c}}).
\]
Denote these cohomology theories by $SH^*([R', R]\times\Sigma)$, respectively $\widehat{SH^*}([R', R]\times\Sigma)$.

\begin{remark}
If $R' > R$, we can still follow the above recipe to define a symplectic cohomology theory.  We continue to write $SH^*([R', R]\times\Sigma)$, respectively $\widehat{SH^*}([R', R]\times\Sigma)$.  These theories are dual to $SH^*([R, R']\times\Sigma)$, respectively $\widehat{SH^*}([R, R']\times\Sigma)$.
\end{remark}

\begin{remark}
\label{rem:comnotred}
Note that, {\it a priori}, 
\[
H\left(\widehat{SC^*}(H_R)\right)\neq \widehat{SH^*}(H_R).
\]
The left-hand side is a quotient by $\im(\widehat{\dd})$, while the right-hand side is a quotient by $\overline{\im(\widehat{\dd})}$.  In the examples we consider, however, the two will coincide.  See Section \ref{sec:toric}.
\end{remark}

\begin{remark}
Suppose that $E$ is exact and $\Sigma\subset E$ is a convex contact hypersurface as in (\ref{eq:sigma}).  $\Sigma$ has associated to it a Floer-type invariant called Rabinowitz Floer homology, denoted by $RFH^*(\Sigma)$.  Cieliebak-Frauenfelder-Oancea showed in \cite{c-f-o} that there is an isomorphism
\[
H\left(Cone(\mathfrak{c})\right) \simeq RFH(\Sigma).
\]
In the non-exact case, there is a completed version of Rabinowitz Floer homology associated to a contact hypersurface $\{R\}\times\Sigma$, which we denote by $\widehat{RFH^*}(\{R\}\times\Sigma)$.  This was first studied by Albers-Kang in \cite{albers-k}.  In Section \ref{sec:toric} we will give examples of scenarios in which, for $R' < R'' < R$, 
\begin{equation} 
\label{eq:rfhiso}
\widehat{SH^*}([R', R]\times\Sigma)\simeq \widehat{RFH^*}(\{R''\}\times\Sigma).
\end{equation}
There are maps 
\[
\widehat{SH^*}([R', R]\times\Sigma)\rightarrow\widehat{SH^*}([S', S]\times\Sigma)
\]
whenever $[S', S]\subset[R', R]$ \cite{cieliebak-o}.  We expect the isomorphism (\ref{eq:rfhiso}) to generalize to an isomorphism
\begin{equation} 
\label{eq:rfhiso2}
\lim_{\substack{\rightarrow \\ R' < R'' < R}}\widehat{SH^*}([R', R]\times\Sigma)\simeq\widehat{RFH^*}(\{R''\}\times\Sigma).
\end{equation}
\end{remark}

\subsection{Morse-Bott Floer theory and cascades}
\label{sub:mb}
Thus far we have assumed that all periodic orbits are non-degenerate.  In the examples considered in this paper, however, all periodic orbits will be transversally non-degenerate, requiring Morse-Bott techniques.  We follow the exposition in \cite{bourgeois-o}.  Let $H$ be a Hamiltonian whose orbits are each either constant and non-degenerate or transversally non-degenerate.  For each non-constant orbit $x\in\cg{P}(H)$ choose a generic perfect Morse function $f_x:S^1\longrightarrow\RR$.  A choice of $\Omega$-tame almost-complex structure $J$ defines {\it cascades}: tuples ${\bf u} = (c_m, u_m, c_{m-1}, u_{m-1}, \dots, u_1, c_0)$ associated to a sequence of orbits $x_m, x_{m-1}, \dots, x_0$ with $x_{m-1}, \dots, x_1$ non-constant, such that
\begin{enumerate}
\item $c_i\in\im(x_i)$
\item $u_i$ is a finite-energy Floer solution corresponding to the Floer data $(H, J)$, 
\item $\lim\limits_{s\rightarrow\infty}u_i(s, 0)$ is in the stable manifold of $c_i$ (or $c_i = \lim\limits_{s\rightarrow\infty}u_i(s, t)$ if $x_i$ is constant), and
\item $c_i$ is in the unstable manifold of $\lim\limits_{s\rightarrow-\infty}u_{i+1}(s, 0)$ (or $c_i = \lim\limits_{s\rightarrow-\infty}u_{i+1}(s, t)$ if $x_i$ is constant).
\end{enumerate} 
Choose capping discs $\tilde{x}_0$ for $x_0$ and $\tilde{x}_m$ for $x_m$.  The union $\tilde{x}_m\#-u_m\#-u_{m-1}\#\dots\#-u_1\#-\tilde{x}_0$ represents a homology class $\beta\in H_2(E)$.  Let $p$ and $q$ be constant orbits of $H$ or critical points of some functions $f_x$ and $f_{x'}$.  The moduli space $\widehat{\cg{M}}_{\beta, m}(q, p)$ is the space of tuples $(c_m, u_m, c_{m-1}, u_{m-1}, \dots, u_1, c_0)$ representing class $\beta$ such that $c_m$ is in the stable manifold of $p$ (or is equal to a constant orbit $p$) and $c_0$ is in the stable manifold of $q$ (or equal to a constant orbit $q$).  Each component of $\widehat{\cg{M}}_{\beta, m}(q, p)$ carries an $\RR^m$ action, induced by the $\RR$-actions on each Floer trajectory.  By Proposition 3.2 in \cite{bourgeois-o}, 
\[
\cg{M}_{\beta}(q, p) := \bigsqcup_{m\geq 1} \bigslant{\widehat{\cg{M}}_{\beta, m}(q, p)}{\RR^m}
\]
is a manifold of the expected dimension.  The Floer differential now counts
\[
\dd^{fl}(p) = \sum_{\substack{|q| - |p| = 1 \\ \beta\in\pi_2(E)}}\sum_{u\in\cg{M}^0_{\beta}(q, p)} \pm T^{-[\Omega](\beta)}q,
\]
where $\cg{M}^0_{\beta}(q, p)$ is the zero-dimensional stratum.  The continuation maps are similarly modified.

Instead of using cascades, one could just generically perturb the Hamiltonian.  However, the $S^1$-symmetry of the unperturbed Hamiltonian will be useful in the computations in this paper.  Our use of cascades is justified by a result by Bourgeois-Oancea, showing the equivalence of the two approaches.
\begin{thm}[Bourgeois-Oancea: Theorem 3.7 in \cite{bourgeois-o}]
If $H$ is a transversally-nondegenerate Hamiltonian, there exists a non-degenerate Hamiltonian $H'$ -- a perturbation of $H$ -- such that
\[
CF^*(H)\simeq CF^*(H')
\]
are chain-isomorphic.
\end{thm}

\section{Negative line bundles}
Let $(M, \omega)$ be a symplectic manifold of dimension $2m-2$.  Denote by $E\xlongrightarrow{\rho} M$ the line bundle satisfying $c_1^{E} = -k[\omega]$ for some fixed $k > 0$.  Such a line bundle is called {\it negative}.  $E$ is a symplectic manifold; assume that $E$ is weak\textsuperscript{+} monotone so that moduli spaces of Floer solutions do not see bubbling.  Following \cite{oancea-leray} and \cite{ritter-gromov}, we construct a canonical symplectic form on $E$ through $\omega$.

Let $J$ be an $\omega$-compatible almost-complex structure on $M$.  Let $|\cdot |$ be a Hermitian metric on $E$ with induced Chern curvature $\cg{F}$.  Define a radial coordinate $r$ by $r(w) = |w|$ and a fiber-wise angular one-form on the complement of the zero-section by
\[
\alpha = \frac{1}{4\pi k}d^c\log(r^2)
\]
so that
\[
d\alpha = \frac{1}{4\pi k} dd^c\log(r^2) = \frac{i}{2\pi k}\dd\bar{\dd}\log(r^2) = -\frac{i}{2\pi k}\rho^*\cg{F}\equiv -\frac{1}{k} \rho^*c_1^E \equiv \rho^*[\omega].
\]
Note that $\alpha$ defines a contact one-form on the unit circle bundle.  Let 
\[
\Omega := (1 + k\pi r^2)d\alpha + 2k\pi r dr\wedge\alpha = d\alpha + d(k\pi r^2\alpha)
\]
be a symplectic form on the complement of the zero section.  Extend $\Omega$ smoothly over the zero-section by 
\[
\Omega\big|_{\text{zero section}} = -\frac{i}{2\pi k}\rho^*\cg{F} + \text{\{area form of fiber\}}.
\]
Then $\Omega$ is a symplectic form on $E$ and $[\Omega] = [\rho^*\omega]$.

We wish to compute the symplectic cohomology of $E$.  There is a simple family of Hamiltonians to take, which lead to an elegant computation of $SH^*(E)$.  This was developed by Ritter in \cite{ritter-gromov} and \cite{ritter-fano} using the $S^1$-action on $E$ that rotates the fibers.  To fit our framework, we modify this construction slightly, and we appeal to \cite{groman} to assert that the two frameworks yield isomorphic homology theories.

Let $\{g_n:\RR\longrightarrow\RR\}_{n\in\ZZ}$ be a family of functions defined as
\[
g_n(r) = \left(\frac{n}{k} + \frac{1}{2k}\right)r,
\]
as depicted in Figure \ref{fig:g}.  Choose a $\cg{C}^2$-small Morse function $f:M\longrightarrow\RR$ on $M$.  Define
\[
G_n = g_n(k\pi r^2) + (1 + k\pi r^2)\rho^*f.
\]

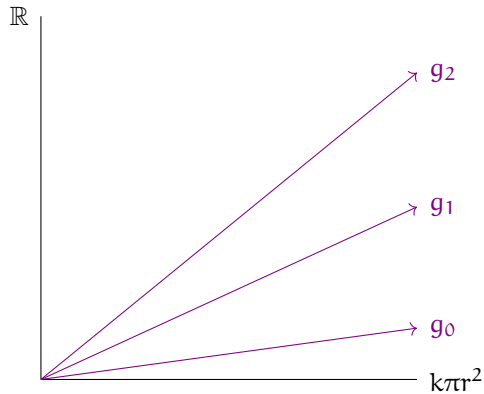
\begin{figure}[htpb!]
\centering
\begin{tikzpicture}[scale=5]
\draw (0, -1.03660231) -- (0, -.07) node [left = 1pt] {$\RR$};
\draw (0, -1.03660231) -- (1, -1.03660231)node [right = 1pt] {$k\pi r^2$};
\draw[->, violet]  (0, -1.03660231)-- (1, -.9)node [right=1pt] {$g_0$};
\draw[->, violet] (0, -1.03660231) -- (1, -.57735)node [right=1pt] {$g_1$};
\draw[->, violet]  (0, -1.03660231) -- (1, -.22)node [right=1pt] {$g_2$};
\end{tikzpicture}
\caption{The Hamiltonians $g_0(k\pi r^2), g_1(k\pi r^2),$ and $g_2(k\pi r^2)$}
\label{fig:g}
\end{figure}
Recall that, away from the zero section,
\[
\Omega = (1 + k\pi r^2)d\alpha + 2k\pi r dr\wedge\alpha
\]
Let $X_f^h$ be the horizontal lift of the Hamiltonian vector field $X_f$ on $M$, uniquely defined through the connection one-form $\alpha$.  As
\begin{align*}
dG_n &= 2k\pi r(g_n'(k\pi r^2) + \rho^*f)dr + (1 + k\pi r^2)\rho^*df \\
&= 2k\pi r\left(\frac{n}{k} + \frac{1}{2k} + \rho^*f\right)dr + (1 + k\pi r^2)\rho^*df,
\end{align*}
the Hamiltonian vector field of $G_n$ away from the zero section is
\[
X_{G_n} = \left(\frac{n}{k} + \frac{1}{2k} + \rho^*f\right)R_{\alpha} + X_f^h.
\]
The Reeb orbits of $\alpha$ have period $\frac{1}{k}$, and so the periodic orbits of $X_{G_n}$ exist only where $\frac{n}{k} + \frac{1}{2k} + \rho^*f$ is an integer multiple of $\frac{1}{k}$.  By construction, $\rho^*f$ is $\cg{C}^2$-small, and so we can assume that $X_{G_n}$ has no periodic orbits away from the zero section.  On the zero section $X_{G_n} = X_f^h$, which can be identified with $X_f$.  Thus, the periodic orbits of $X_{G_n}$ correspond to the periodic orbits of $X_f$.  As $f$ is $\cg{C}^2$ small, these correspond precisely to the critical points of the Morse function $f$.  It follows that there is a vector space isomorphism
\[
CF^*(G_n) \simeq CM^*(f)
\]
for every $n\in\ZZ$.

Suppose that $\sigma:S^1\longrightarrow Ham(E, id)$ is a loop of Hamiltonians based at the identity.  Then $\sigma$ acts on the loopspace $\cg{L}(E)$ by $\sigma^*x(t) = \sigma(t)(x(t))$.  This lifts to an action $\tilde{\sigma}$ on a cover of the loopspace $\widetilde{\cg{L}(E)}$.  We fix $\widetilde{\cg{L}(E)}$ to be the cover defined by the deck transformation group
\[
\Gamma = \ker([\Omega]\big|_{\pi_2(E)})\cap \ker(c_1^{TE}(E)\big|_{\pi_2(E)}).
\]
In other words, $\widetilde{\cg{L}(E)}$ is the group of cappings of each loop, under the equivalence relation $u\simeq v$ if 
\begin{itemize} 
\item $\dd u = \dd v$, 
\item $[\Omega]([u\#(-v)]) = 0$, and 
\item $c_1^{TE}([u\#(-v)]) = 0$.
\end{itemize}

Let $\sigma_t$ be the action on $E$ that rotates each fiber by $e^{2\pi i t}$.  This is a Hamiltonian action generated by the Hamiltonian $\frac{1}{k}(k\pi r^2)$, preserving the radial coordinate $k\pi r^2$, and so
\[
\sigma_t^*G_n := G_n\circ \sigma_t - \left(\frac{1}{k}k\pi r^2\right)\circ \sigma_t = \left(G_n - \frac{1}{k}(k\pi r^2)\right)\circ \sigma_t = G_{n-1}.
\]
Define the lift $\tilde{\sigma}$ to preserve cappings of constant loops.

Let $J_t$ be any one-parameter family of $\Omega$-tame almost-complex structures.  The action of $\sigma_t$ on $J_t$ defined by
\[
\sigma_t^*J_t = d\sigma_t^{-1}\circ J_t\circ d\sigma_t,
\]
produces another one-parameter family of $\Omega$-tame almost complex structures.  Recall that Floer data $(H, J)$ is {\it generic} if the Floer cochain complex $CF^*(H, J)$ is well-defined.  The following two theorems, due to Ritter, yield the promised computation of symplectic cohomology (see Section 7 of \cite{ritter-gromov}).
\begin{theorem}[Ritter \cite{ritter-gromov}]
\label{thm:ritterrotate}
If $(G_n, J)$ is generic, then so is $(\sigma_t^*G_n, \sigma_t^*J)$.  The action $\tilde{\sigma}$ induces a chain isomorphism
\[
\cg{S}: CF^*(G_n, J) \longrightarrow CF^{*+2}(\sigma^*G_n, \sigma^*J) = CF^{*+2}(G_{n-1}, \sigma^*J).
\]
\end{theorem}
Choose continuation maps
\[
c^G_n:CF^*(G_n,J) \longrightarrow CF^*(G_{n+1}, J')
\]
between Floer data $(G_n, J)$ and $(G_{n+1}, J')$.
\begin{theorem}[Ritter \cite{ritter-gromov}] 
\label{theorem:rittercommutes}
The following diagram induces a commutative diagram on the level of cohomology.
\[
\begin{tikzcd}
CF^*(G_{n-1}, J) \arrow{r}{c_{n-1}^G} \arrow{d}{\cg{S}^n}& CF^*(G_n, J') \arrow{d}{\cg{S}^n} \\
CF^{*+2n}(G_{-1}, (\sigma^*)^{n}J)  \arrow{r}{c_{-1}^G} & CF^{*+2n}(G_0, (\sigma^*)^nJ')
\end{tikzcd}
\]
\end{theorem}
\begin{corollary}[Ritter \cite{ritter-gromov}]
\label{cor:sh=G0}
The symplectic cochain complex
\[
SC^*(G) := \bigoplus_{n=0}^{\infty} CF^*(G_n)[{\bf q}]
\]
is quasi-isomorphic to a complex
\[
SC^*(G_0) := \bigoplus_{n=0}^{\infty}CF^{*}(G_0)[-2n][{\bf q}]
\]
with differential
\[
\dd(x + y{\bf q}) = \dd^{fl}(x) + y - c_{-1}^G\circ\cg{S}(y) + \dd^{fl}(y){\bf q}.
\]
\end{corollary}
Note that the grading shift $[-2n]$ {\it increases} the Conley-Zehnder index by $2n$.

Denote the $0$\textsuperscript{th} generalized eigenspace of the map 
\[
\boldsymbol{c_{-1}^G}\circ\boldsymbol{\cg{S}}:HF^*(H_0)\longrightarrow HF^(H_0).
\] 
by $HF^*_0(H_0)$.  A linear algebra consequence of Corollary \ref{cor:sh=G0} is
\begin{corollary}[Ritter \cite{ritter-gromov}]
\label{cor:ritter1}
The uncompleted symplectic cohomology of $E$ is
\[
SH^*(E) = SH^*(G) \cong SH^*(G_0) \cong\bigslant{HF^*(H_0)}{HF^*_0(H_0)}.
\]
Under this isomorphism, the map
\[
HF^*(H_0)\longrightarrow SH^*(E)
\]
induced by inclusion on cochains is the quotient map.
\end{corollary}
Finally, symplectic cohomology can be rephrased completely in terms of topological information.  There is a map, termed the P.S.S. map after its creators Piukhin-Salamon-Schwartz, that identifies the quantum cohomology $QH^*(E)$ with $HF^*(H_0)$.
\begin{theorem}[Ritter \cite{ritter-gromov}]
\label{cor:hftoqh}
The P.S.S. isomorphism 
\[
QH^*(E)\cong HF^*(H_0)
\]
identifies up to non-zero scalar $\eta$ the map $\boldsymbol{c_{-1}^G}\circ\boldsymbol{\cg{S}}$ with the action of quantum cup product $\cup_*$ with $\rho^*c_1^{E}$.  If $QH^*_0(E)$ is the $0$th-generalized eigenspace of the map $x\mapsto \rho^*c_1^{E}\cup_*x$, there is an isomorphism of $\Lambda$-algebras
\[
SH^*(E)\cong\bigslant{QH^*(E)}{QH^*_0(E)}.
\]
\end{theorem}
Indeed, Ritter shows that the scalar $\eta$ has valuation $ev(\eta) = 0$ (see Theorem 67 of \cite{ritter-gromov}).
\begin{corollary}
\label{cor:specsmatch}
The P.S.S. isomorphism identifies a $\lambda$-generalized eigenspace of $\boldsymbol{ c_{-1}^G}\circ\boldsymbol{\cg{S}^{-1}}$ with a $\eta\lambda$-generalized eigenspace of $\rho^*c_1^E\cup_*-$, and
\[
ev(\lambda) = ev(\eta\lambda).
\]
\end{corollary}

\subsection{Disk subbundles}
We now use reduced symplectic cohomology to define an invariant of a disk subbundle contained in $E$.  Fix a radius $R\in(0, \infty)$, and let $D_R$ be the disk bundle of radius $R$.  Following Subsection \ref{subsec:sh}, we will construct a chain complex generated by one-periodic orbits that cluster near the boundary of $D_R$.  With an eye towards computing reduced symplectic cohomology, we use a specific family of Hamiltonians and a specific class of almost-complex structures.

\subsubsection{The almost-complex structure}
\label{subsec:acs}
Let $\cg{J}(\omega_{std})$ be the space of $S^1$-families of almost-complex structures on $\CC$ compatible with the standard symplectic form, and such that each almost-complex structure is cylindrical in a neighborhood of periodic orbits and at infinity.  Let $\cg{J}(\omega)$ be the space of $S^1$-families of almost-complex structures on $M$ compatible with $\omega$.  The one-form $\alpha$ determines a splitting of $TE$ into a vertical component $V\cong\CC$ and horizontal component $H$.  Let $L_t(H, V)$ be the space of $S^1$-families of linear maps from $H$ to $V$.  Let $\cg{U}$ be an open set comprised of small neighborhoods of the circle bundles on which live non-constant periodic orbits of each $H_n$, as well as small neighborhoods of the constant orbits.  Let $\mathfrak{B}(i_t, j_t)$ be the elements $B_t\in L_t(H, V)$ with compact support in the complement of $\cg{U}$, and satisfying $i_tB_t + B_t\rho^*j_t = 0$ for all $t$.  Define
\[
\cg{J}(\Omega) = \left\{J_t = \left[\begin{array}{cc}i_t & B_t \\ 0 & \rho^*j_t\end{array}\right] \in End(TE) \hspace{.1cm}\bigg|\hspace{.1cm} i_t\in\cg{J}(\omega_{std}), j_t\in\cg{J}(\omega), B_t\in\mathfrak{B}(i_t, j_t), J_t\text{ is }\Omega\text{--tame}\right\}.
\]
The conditions on $i_t, j_t$ and $B_t$ ensure that Floer trajectories ``flow outward'', that Floer trajectories converge to periodic orbits, and that $J_t$ is almost-complex.  

Transversality and regularity for $J\in\cg{J}(\Omega)$ was proven by Albers-Kang in \cite{albers-k} for the completely analogous case of Floer solutions in Rabinowitz Floer homology.  We therefore omit these proofs, and simply refer to \cite{albers-k}.

\begin{lemma}[Albers-Kang: Proposition 2.11 in \cite{albers-k})]
There exists a comeager subset of $\cg{J}(\Omega)$ for which the finite-energy cascades of $H_n$ and $H_n^s$, for any $n$, are cut out transversally.
\end{lemma}

\begin{lemma}[Albers-Kang: Lemma 2.15 in \cite{albers-k}]
All simple $J$-holomorphic spheres are regular.
\end{lemma}

\subsubsection{The Hamiltonians}
\label{subsec:ham}
Fix a constant $C > 0$ and let $\left\{h_n\colon  \RR \longrightarrow\RR\right\}_{n\in\NN}$ be a family of functions where each $h_n$ is
\begin{enumerate}
\item convex and monotone increasing on $\RR_{\geq 0}$, 
\item bounded in absolute value by $C$ on $[0, k\pi R^2]$, and
\item of slope $\frac{n}{k} + \frac{1}{2k}$ on $(k\pi R_n^2, \infty)$, for some $R_n < R$.
\end{enumerate}
Further assume that the sequence $\{R_n\}$ tends to $R$ as $n$ tends to $\infty$.  To simplify later proofs, we also choose $h_n$ such that
\begin{enumerate}
\item $h_n$ and $h_n'$ are monotone increasing on $\RR_{\geq 0}$ and
\item $h_n = h_{n+1}$ on $[0, k\pi R_n^2]$.
\end{enumerate}
For example, we can take $\{h_n\}$ to be smoothings of the functions shown in Figure \ref{fig:unbddh}.

Choose a Morse function function $f\colon M\longrightarrow\RR$ that is $\cg{C}^2$-small.  Define a family of Hamiltonians $\left\{H_n\colon E\longrightarrow\RR\right\}_{n\in\NN}$ by 
\[
H_n = h_n(k\pi r^2) + (1+k\pi r^2)\rho^*f.
\]  
We assume that the one-periodic orbits of $H_n$ are transversally nondegenerate.  For example, we can take each $h_n$ to be a smoothing of a piecewise linear function, as in Figure \ref{fig:unbddh}.

\begin{figure}[htpb!]
\centering
\begin{tikzpicture}[scale=5]
\draw (0, -1.03660231) -- (0, -.07) node [left = 1pt] {$\RR$};
\draw (0, -1.03660231) -- (1, -1.03660231)node [right = 1pt] {$k\pi r^2$};
\draw[violet] (0, -1.03660231) -- (.267948, -1);
\draw[->, violet]  (.267948, -1) -- (1, -.9)node [right=1pt] {$h_0$};
\draw[->, violet] (.267948, -1) -- (1, -.57735)node [right=1pt] {$h_1$};
\draw[->, violet]  (.641751, -.784185) -- (1, -.37797)node [right=1pt] {$h_2$};
\draw[->, violet] (.822168, -.5796118) -- (1, -.2582)node [right=1pt] {$h_3$};
\draw[->, violet] (.911356, -.418414) -- (1, -.1796)node [right=1pt] {.};
\draw[->, violet] (.955751, -.2988102) -- (1, -.12599) node [right=1pt] {.};
\draw[->, violet] (.977884, -.212366) -- (1, -.08873) node [right=1pt] {.};
\end{tikzpicture}
\caption{A family of Hamiltonians $\{h^n(k\pi r^2)\}$}
\label{fig:unbddh}
\end{figure}
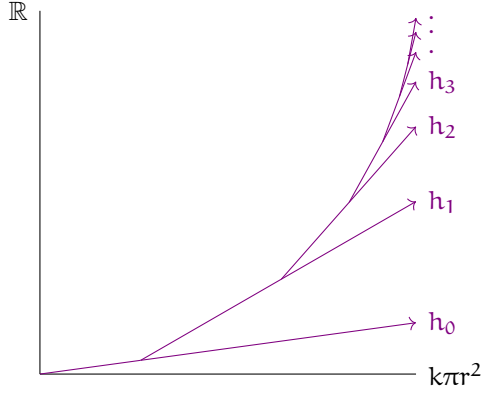

Let $R_{\alpha}$ be the unique vector field on $E\setminus\{\text{zero section}\}$ satisfying 
\[
\left\{\begin{array}{c}\alpha(R_{\alpha}) = 1 \\ dr(R_{\alpha}) = 0 \\ i_{R_{\alpha}}d\alpha = 0\end{array}\right..
\]
Note that $R_{\alpha}\big|_{r=\mathfrak{r}}$ is the Reeb vector field of the contact form $\alpha\big|_{\{r=\mathfrak{r}\}}$ and the simply-covered orbits of $R_{\alpha}$ have period $\frac{1}{k}$.  As 
\[
dH_n = 2k\pi r (h_n'(k\pi r^2) + \rho^*f)dr + (1 + k\pi r^2)\rho^* df,
\] 
the Hamiltonian vector field of $H_n$ is 
\[
X_{H_n} = (h_n'(k\pi r^2) + \rho^*f)R_{\alpha} + X_f^h,
\]
Thus, the one-periodic orbits of $H_n$ correspond bijectively to 
\begin{enumerate} 
\item the $S^1$-families of orbits of $R_{\alpha}\big|_{D_R}$ with period between $\frac{1}{k}$ and $\frac{1}{n}$, lying in fibers above the critical points of $f$, and
\item the critical points of $f$ itself.  
\end{enumerate}
Impose a perfect Morse function on each $S^1$-family of orbits, so that each family gives rise to two distinguished orbits: the minimum and the maximum of the perfect Morse function.

Denote by $\cg{P}(H_n)$ the union of all minimum and maximum distinguished orbits of $H_n$, in addition to the constant orbits, and by $\cg{P}(H)$ the union $\bigcup\limits_n\cg{P}(H_n)$.

Choose generic almost-complex structures in $\cg{J}(\Omega)$ to define the Floer complexes $CF^*(H_n)$.  For each $n$ choose a generic $\RR$-family of functions $h_n^s\colon \RR\rightarrow\RR$, monotonely decreasing in $s$, with $h_n^s = h_n$ when $s>>0$ and $h_n^s = h_{n+1}$ when $s<<0$.  Set $H_n^s = h_n^s(k\pi r^2) + (1 + k\pi r^2)\rho^*f$.  Define continuation maps
\[
c_n:CF^*(H_n) \rightarrow CF^*(H_{n+1})
\]
through $H_n^s$.  In the notation of Section \ref{sec:sh}, this collection of data defines cochain complexes $SC^*(H), \overline{SC^*}(H),$ and $\widehat{SC^*}(H)$.

\subsubsection{Some technical lemmas for computation}
\label{subsec:chaincomplex}
In order to compute $\overline{SH^*}(H)$, we need three standard results on the behavior of the Floer differential and continuation maps.

Define $\mathfrak{w}(x)$ to be the winding number of a non-constant periodic orbit $x$, viewed as a map from the circle into $\CC^*$.  Define $\mathfrak{w}(x)$ of a constant orbit $x$ to be zero.
\begin{lemma}
\label{lem:intmax}
Let $u$ be a solution of Floer's equation (\ref{eq:floer}) with $\lim\limits_{s\rightarrow\pm\infty}u(s, t) = x_{\pm}(t)$.  Then $\mathfrak{w}(x_+)\geq\mathfrak{w}(x_-)$.
\end{lemma}

\begin{proof}
We use the integrated maximum principal of \cite{abouzaid}.  Assume for contradiction that $\mathfrak{w}(x_-) > \mathfrak{w}(x_+)$.  Say that $x_{\pm}$ lives in the sphere bundle of radius $\sigma_{\pm}$.  If $x_{\pm}$ are non-constant then the $R_{\alpha}$-orbit underlying $x_{\pm}$ has period $\frac{1}{k}\mathfrak{w}(x_{\pm})$.  From the earlier computation $X_{H_n} = (h_n'(k\pi r^2) + \rho^*f)R_{\alpha} + X_f^h$ it follows that $h_n'(k\pi\sigma_{\pm}^2) + \rho^*f(x_{\pm}) = \frac{1}{k}\mathfrak{w}(x_{\pm})$.  The winding numbers are integers, and so
\[
h_n'(k\pi\sigma_-^2) + \rho^*f(x_-) \geq h_n'(k\pi\sigma_+^2) + \rho^*f(x_+) + \frac{1}{k}.
\]
By the smallness of $f$ we can assume that
\[
h_n'(k\pi\sigma_-^2) - h_n'(k\pi\sigma_+^2)  \geq \rho^*f(x_+) - \rho^*f(x_-) + \frac{1}{k} > 0.
\]
As $h_n$ is convex, we deduce that $\sigma_- > \sigma_+$.  If $x_+$ is constant and $x_-$ is non-constant it follows immediately that $\sigma_- > \sigma_+ = 0$.

Choose a generic circle subbundle $\mathfrak{S}$ of radius $\sigma$, with $\sigma_+ < \sigma < \sigma_-$, and on which $B = 0$ and $i_t$ is cylindrical.  For example, if $\sigma$ is close to $\sigma_+$ or $\sigma_-$ these conditions will, by construction, be met.  Let $\cg{D}$ be the region bounded by $\mathfrak{S}$, and denote $\Sigma = u^{-1}(E\setminus\cg{D})$.  Let $v\colon \Sigma\longrightarrow E\setminus\cg{D}$ be the restriction of $u$.  We will equate $\Sigma$ with its image under the inclusion into $\RR\times S^1$ and use the coordinates $(s, t)$ induced on $\Int(\Sigma)$.

Let $c_x = h_n(k\pi x^2) - h'_n(k\pi x^2)k\pi x^2$ be the $y$-intercept of the tangent line to $h_n(k\pi r^2)$ at $k\pi x^2$.  Then on $\mathfrak{S}$, $H_n = h_n'(k\pi \sigma^2)k\pi \sigma^2 + (1 + k\pi \sigma^2)\rho^*f + c_{\sigma}$ (and $X_{H_n} = (h_n'(k\pi \sigma^2) + \rho^*f)R_{\alpha} + X_f^h$).  In particular, 
\[
H_n\big|_{\mathfrak{S}} = (1 + k\pi \sigma^2)\alpha(X_{H_n}) - h_n'(k\pi \sigma^2) + c_{\sigma}.
\]

Let $\dd\Sigma_+$ be the union of the boundary components of $\Sigma$ mapping into $\mathfrak{S}$.  Note that we have chosen $J$ to be $\Omega$-tame, so that
\[
E_J(v) := \frac{1}{2}\int_{\Sigma}\left(\Omega(\dd_sv, J\dd_sv\right) + \Omega\left(\dd_tv - X_{H_n}(v), J(\dd_tv - X_{H_n}(v))\right)ds\wedge dt \geq 0
\]
Shuffling terms, we have
\begin{equation*}
\begin{split}
E_J(v) &= \int_{\Sigma} v^*\Omega - v^*dH_n\otimes dt \\
&= \int_{\dd\Sigma_+} (1 + k\pi r^2)v^*\alpha - H_n(v(s, t))dt - \int_{S^1} (1 + k\pi r^2)x_-^*\alpha - H_n(x_-(t))dt 
 \\&= \int_{\dd\Sigma_+} (1 + k\pi \sigma^2)v^*\alpha - (1 + k\pi \sigma^2)\alpha(X_{H_n})\otimes dt + \left(h'_n(k\pi \sigma^2)- c_{\sigma}\right)dt 
 \\ &\hspace{2cm} - \int_{S^1} (1 + k\pi \sigma_-^2)x_-^*\alpha - H_n(x_-(t))dt 
\end{split}
\end{equation*}
We consider this final equation in pieces.  A solution $v(s, t)$ of Floer's equation satisfies 
\[
(dv - X_{H_n}\otimes dt)^{(0, 1)} = 0.
\] 
We have that, on $\mathfrak{S}$, $dr\circ i_t = -2k\pi r\alpha$ and $B_t = 0$.  The former implies that $JR_{\alpha}$ is proportional to $\dd_r$ and the latter implies that $JX_f^h$ lives in the horizontal distribution.  So altogether, $\alpha(JX_{H_n}) = 0$.  Thus, the ``$\dd\Sigma_+$'' terms become
\begin{align*}
\int_{\dd\Sigma_+}(1 + k\pi \sigma^2)\alpha(dv - X_{H_n}\otimes dt) + &\left(h'_n(k\pi \sigma^2)- c_{\sigma}\right)dt  \\
&= \int_{\dd\Sigma_+} -(1 + k\pi \sigma^2)\alpha\circ J(dv - X_{H_n}\otimes dt)\circ j + \left(h'_n(k\pi \sigma^2)- c_{\sigma}\right)dt  \\
&= \int_{\dd\Sigma_+} -\frac{1 + k\pi \sigma^2}{2k\pi\sigma}dr \circ dv \circ j + \left(h'_n(k\pi \sigma^2)- c_{\sigma}\right)dt  \\
&\leq \int_{\dd\Sigma_+}\left(h'_n(k\pi \sigma^2)- c_{\sigma}\right)dt, 
\end{align*}
where the last inequality follows because a vector $\zeta$ that is positively-oriented with respect to the boundary orientation satisfies $dr\circ j(\zeta) \geq 0$.  We also have
\[
H_n(x_-(t)) = h_n'(k\pi\sigma_-^2)k\pi\sigma_-^2 + c_{\sigma_{-}} + (1+k\pi\sigma_-^2)\rho^*f(x_-(t)) = (1 + k\pi\sigma_-^2)\frac{\mathfrak{w}(x_-)}{k} - h_n'(k\pi\sigma^2) + c_{\sigma_{-}}
\]
and so the ``$x_-$'' terms become
\begin{align*}
-\int_{S^1} (1 + k\pi r^2)x_-^*\alpha + \int_0^1 H_n(x_-(t))dt &= -\frac{\mathfrak{w}(x_-)}{k}(1 + k\pi \sigma_-^2) +\frac{\mathfrak{w}(x_-)}{k}(1 + k\pi \sigma_-^2) + c_{\sigma_{-}} - h_n'(k\pi\sigma^2_-) \\
&= \int_{S^1}\left(- h_n'(k\pi\sigma_-^2) + c_{\sigma_{-}}\right)dt
\end{align*}
Altogether,
\begin{align}
\label{eq:energyest}
E_J(v) &\leq \int_{\dd\Sigma_+}\left(h'_n(k\pi \sigma^2)- c_{\sigma}\right)dt  + \int_{S^1}\left(- h_n'(k\pi\sigma_-^2) + c_{\sigma_{-}}\right)dt.
\end{align}
$\Sigma$ is a collection of bounded regions in $\CC^*$ and one unbounded region enclosing the origin.  As $dt$ is exact on $\CC^*$, the bounded regions contribute nothing to the right-hand side of (\ref{eq:energyest}).  Let $\tilde{\Sigma}$ be the unbounded component.  Near zero, $\Sigma$ is contained in a neighborhood of $x_-$, and so all boundary components of $\tilde{\Sigma}$ occur within the intersection of $\Sigma$ with some annulus $[R, \infty)\times S^1$.  Let $\mathfrak{h}$ be a function on $\tilde{\Sigma}$ that is equal to $h_n'(k\pi\sigma^2) - c_{\sigma}$ on $\tilde{\Sigma}\cap\left([R, \infty)\times S^1\right)$ and equal to $h_n'(k\pi\sigma_-^2) - c_{\sigma_{-}}$ for all radii $s < R$.  Then
\begin{align*}
E_J(v) &\leq \int_{\dd\Sigma}\mathfrak{h}dt = \int_{\tilde{\Sigma}} d(\mathfrak{h}dt) \\
&= \int_{\tilde{\Sigma}\cap\left((0, R]\times S^1\right)}d(\mathfrak{h}dt) + \int_{\tilde{\Sigma}\cap\left([R, \infty)\times S^1\right)}d(\mathfrak{h}dt) \\
&= \int_{\tilde{\Sigma}\cap\left((0, R]\times S^1\right)}d(\mathfrak{h}dt) \\
&= (h_n'(k\pi\sigma^2) - h_n'(k\pi\sigma_-^2)) + (c_{\sigma_{-}}-c_{\sigma}).
\end{align*}

As $h_n$ is convex by assumption and $\sigma < \sigma_-$, both $\left(h_n'(k\pi\sigma^2) - h_n'(k\pi\sigma_-^2)\right) < 0$ and $\left(c_{\sigma_-} - c_{\sigma}\right) < 0$.  It follows that
\begin{align}
\label{eq:energyineq}
E_J(v) &\leq (h_n'(k\pi\sigma^2) - h_n'(k\pi\sigma_-^2)) + (c_{\sigma_{-}}-c_{\sigma}) < 0.
\end{align}
which yields the desired contradiction.

\end{proof}

Lemma \ref{lem:intmax} says that the winding number is decreased by the Floer differential.  Lemma \ref{cor:contleaves}, below, says that the winding number is decreased by a continuation map.  Thus, the winding numbers provide an auxiliary filtration on $SC^*(H)$.
\begin{lemma}
\label{cor:contleaves}
Let $u$ be a solution of Floer's equation with respect to a Hamiltonian $(H_n)_s$ such that $\dd_s(H_n)_s \leq 0$ everywhere.  Suppose $\lim\limits_{s\rightarrow\pm\infty}u(s, t) = x_{\pm}(t)$.  Then $\mathfrak{w}(x_+)\geq\mathfrak{w}(x_-)$.
\end{lemma}

\begin{proof}
The proof of Lemma \ref{lem:intmax} applies almost verbatim, except the energy $E_J(v)$ will have an additional integral
\[
\int_{\Sigma}\dd_s(H_n)_s ds\wedge dt 
\]
which is non-positive, by assumption.  This does not affect the final inequality (\ref{eq:energyineq}), from which a contradiction was derived.

\end{proof}

Finally, the additional structure imposed on the continuation maps yields the following.
\begin{lemma}
Continuation maps act as the canonical inclusions, sending a periodic orbit of $X_{H_n}$ to the periodic orbit of $X_{H_{n+1}}$ represented by the same map $S^1\rightarrow E$.
\end{lemma}

\begin{proof}
Let $u$ be a Floer solution of the data $((H_n)_s, J)$, where $(H_n)_s$ induces an action-increasing continuation map $c_n:CF^*(H_n)\rightarrow CF^*(H_{n+1})$.  By assumption, $h_n = h_{n+1}$ on $[0, k\pi R_n^2]$.  Thus, $(H_n)_s = H_n = H_{n+1}$ within the disk bundle of radius $R_n$.  By index considerations, any Floer solution is either constant or leaves the disk bundle of radius $R_n$.  The latter cannot happen by Lemma \ref{cor:contleaves} and its easier analogue: $u$ remains inside the smallest disk bundle containing both of its asymptotes.  Counting the constant Floer solutions precisely describes the canonical inclusion.

\end{proof}

\section{Proof of Theorem \ref{thm:overlinesh}}
Recall the main theorem of this paper.
\thmmain*
We prove Theorem \ref{thm:overlinesh} in two steps, using the relationships between $\overline{SH^*}(H)$ and the closely-related theories $SH^*(H)$ and $\widehat{SH^*}(H)$.  Namely, there is an inclusion
\[
\phi:\overline{SH^*}(H)\hookrightarrow \widehat{SH^*}(H).
\]
induced by the inclusion
\[
\widehat{\ker(\dd)}\hookrightarrow \widehat{SC^*}(H)
\]
such that the following diagram commutes.
\begin{equation} 
\label{eq:commutediag}
\begin{tikzcd}
& \overline{SH^*}(H) \arrow{d}{\phi} \\
SH^*(H) \arrow{ur}{\eta} \arrow{r}{\pi} & \widehat{SH^*}(H)
\end{tikzcd}
\end{equation}
Here, $\eta$ and $\pi$ come from the inclusion of $\ker(\dd)$ into its completion.  Viewing $\overline{SH^*}(H)$ as a subspace of $\widehat{SH^*}(H)$ through $\phi$, we will prove
\begin{proposition}
\label{prop:imphi}
The image of $\pi$ is precisely the subspace $\overline{SH^*}(H)$.
\end{proposition}
In particular, 
\[
\overline{SH^*}(H) \simeq\bigslant{SH^*(H)}{\ker(\pi)}.
\]

Recall from Corollary \ref{cor:ritter1} that the P.S.S. map
\[
\iota:HF^*(H_0)\longrightarrow SH^*(E)
\]
induced by the inclusion
\[
CF^*(H_0)\hookrightarrow SC^*(E)
\]
is surjective.  The composition
\begin{equation} 
\label{eq:hattedpss}
\Phi := \pi\circ\iota:HF^*(H_0)\longrightarrow \widehat{SH^*}(H)
\end{equation}
is therefore surjective.  By Proposition \ref{prop:imphi}, its image is equal to $\overline{SH^*}(H)$.  Thus, 
\[
\overline{SH^*}(H) \simeq \bigslant{HF^*(H_0)}{\ker(\Phi)}.
\]

The remainder of this section is devoted to computing the kernel of $\Phi$ and proving Proposition \ref{prop:imphi}.  Recall from Theorem \ref{thm:groman} the isomorphism
\[
\widehat{SH^*}(H) = \lim_{\substack{\leftarrow \\ a}}SH^*_a(H).
\]
Under this isomorphism, the map $\Phi$ is the inverse limit of maps
\[
\Phi_a: HF^*(H_0)\longrightarrow SH^*_a(H)
\]
induced by the chain-level quotient maps
\[
SC^*(H)\longrightarrow SC^*_a(H) = \bigslant{SC^*(H)}{SC^*_{(a, \infty)}(H)}.
\]
We will compute the kernel of each $\Phi_a$ and show that 
\[
\ker(\Phi) = \lim_{\substack{\leftarrow \\ a}}\ker(\Phi_a).
\]
In Step I (Subsection \ref{subsec:step1}) we identify eigenvectors of the action $\boldsymbol{c_{-1}^G\circ\cg{S}}$ lying in $\ker(\Phi_a)$.  In Step II (Subsection \ref{subsec:step2}) we show that, in the limit as $a$ becomes large, these are essentially the only elements in $\ker(\Phi_a)$. 

\subsection{Step I}
\label{subsec:step1}
Consider an auxiliary family of Hamiltonians, defined as follows.  Recall the radii $\{R_n\}$ that are part of the data of the Hamiltonians $\{H_n\}$.  Define
\[
K_n = G_n - nk\pi R_n^2.
\]
The subcomplex
\[
\bigoplus_{n=0}^{\infty}CF^*(K_n) \subset SC^*(K)
\]
has a valuation derived from the usual action on each Floer chain complex.  We extend this non-trivially to all of $SC^*(K)$ by defining
\[
\cg{A}({\bf q}) = -k\pi R^2.
\]
Geometrically, this corresponds to the shift
\[
CF^*(K_n){\bf q} \cong CF^*(K_n - k\pi R^2)
\]
and ensures that action is increased by choices of continuation maps 
\[
c_n^K: CF^*(K_n){\bf q}\longrightarrow CF^*(K_{n+1}).
\]
See Figure \ref{fig:unbdd}.

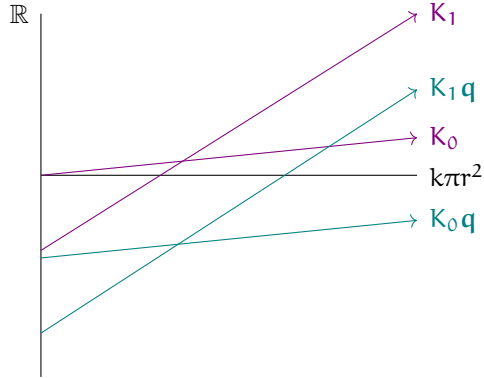
\begin{figure}[htpb!]
\centering
\begin{tikzpicture}[scale=5]
\draw (0, -1.03660231) -- (0, -.07) node [left = 1pt] {$\RR$};
\draw (0, -.5) -- (1, -.5)node [right = 1pt] {$k\pi r^2$};
\draw[->, violet]  (0, -.5)-- (1, -.4)node [right=1pt] {$K_0$};
\draw[->, teal] (0, -.72) -- (1, -.62)node [right=1pt] {$K_0{\bf q}$};
\draw[->, violet]  (0, -.7) -- (1, -.07)node [right=1pt] {$K_1$};
\draw[->, teal]  (0, -.92) -- (1, -.272)node [right=1pt] {$K_1{\bf q}$};
\end{tikzpicture}
\caption{The family of Hamiltonians $K_n$}
\label{fig:unbdd}
\end{figure}

Because action is increased by continuation, $SC^*(K)$ has subcomplexes
$
SC^*_{(a, \infty)}(K)
$
and quotient complexes
$
SC^*_a(K)
$
analogous to the subcomplex (\ref{eq:sub}) and quotient complex (\ref{eq:quotient}) of $SC^*(H)$.

Choices of action-increasing continuation maps
\[
CF^*(K_n)\longrightarrow CF^*(H_n)
\]
induce a continuation map
\[
\mathfrak{c}^{KH}:SH^*(K)\longrightarrow SH^*(H)
\]
that descends to a map
\[
\mathfrak{c}_a^{KH}:SH^*_a(K)\longrightarrow SH^*_a(H).
\]
This induces maps on the long-exact sequences
\[
\begin{tikzcd}
\dots \arrow{r} & SH^*_{(a, \infty)}(H) \arrow{r} & SH^*(H) \arrow{r}{\pi_a} & SH^*_a(H)\arrow{r} &  \dots \\
\dots\arrow{r} & SH^*_{(a, \infty)}(K)\arrow{r} \arrow{u} &  SH^*(K)\arrow{r}{\pi_a^K}  \arrow{u}{\mathfrak{c}^{KH}} &SH^*_a(K)\arrow{r} \arrow{u}{\mathfrak{c}^{KH}_a} &  \dots
\end{tikzcd}
\]
such that each square commutes.  The commutativity of the right-hand square gives
\begin{equation}
\label{diag:pichicommute}
\pi_a\circ\mathfrak{c}^{KH} = \mathfrak{c}_a^{KH}\circ\pi_a^K.
\end{equation}
By construction,
\[
\Phi = \mathfrak{c}^{KH}\circ\Phi^K,
\]
where $\Phi^K:HF^*(K_0)\longrightarrow{SH^*}(K)$ is the PSS map, defined similarly to (\ref{eq:hattedpss}). To $\Phi^K$ is associated the action-filtered map
\[
\Phi_a^K:HF^*(K_0)\longrightarrow SH^*_a(K).
\]
Note that
\begin{equation}
\label{eq:piphinonsense}
\Phi_a = \mathfrak{c}_a^{KH}\circ\Phi_a^K.
\end{equation}
Recall the following fact.
\begin{lemma}[Lemma A.1 in \cite{fooo}]
$\Lambda$ is algebraically closed.
\end{lemma}
Denote by $\{\lambda_1, .., \lambda_m\}$ the eigenvalues of the map
\[
\boldsymbol{c_{-1}^G}\circ\boldsymbol{\cg{S}}: HF^*(H_0)\longrightarrow HF^*(H_0)
\]
(including geometric multiplicities) and fix a Jordan basis 
\[
{\cg B} = \left\{{\bf v_1^{\lambda_1}}, ..., {\bf v_{k_{\lambda_1}}^{\lambda_1}}, {\bf v_1^{\lambda_2}}, ..., {\bf v_{k_{\lambda_m}}^{\lambda_m}}\right\},
\]
so that $\{{\bf v_1^{\lambda_i}}, ..., {\bf v_{k_{\lambda_i}}^{\lambda_i}}\}$ spans the invariant subspace associated with the eigenvalue $\lambda_i$.  Define a valuation
\begin{align*} 
\boldsymbol{\cg{A}}:HF^*(K_n)&\longrightarrow\RR\cup\{\infty\} \\
0 &\mapsto\infty \\
{\bf X}\neq 0&\mapsto \min_{\substack{X\in CF^*(K_n) \\ [X] = {\bf X}}}\cg{A}(X).
\end{align*}
\begin{lemma}
\label{lem:seminorm}
The valuation $\boldsymbol{\cg{A}}$ satisfies
\begin{enumerate} 
\item $\boldsymbol{\cg{A}}({\bf X + Y}) \geq \min(\boldsymbol{\cg{A}}({\bf X}), \boldsymbol{\cg{A}}({\bf Y}))$ and
\item There exists a constant $C_n$, such that, for any $X\in CF^*(K_n)$ with $[X] = {\bf X}$, \[\cg{A}(X) - \boldsymbol{\cg{A}}({\bf X}) \leq C_n.\]
\end{enumerate}
\end{lemma}

\begin{proof}
Let $Z\in CF^*(K_n)$ be a representative of ${\bf X} + {\bf Y}$ such that
\[
\cg{A}(Z) = \boldsymbol{\cg{A}}({\bf X} + {\bf Y}).
\]
Let $X$ be a representative of ${\bf X}$ and $Y$ a representative of ${\bf Y}$ such that 
\[
Z = X + Y.
\]
As $\cg{A}$ is a non-Archimedean valuation,
\[
\cg{A}(Z) \geq \min(\cg{A}(X), \cg{A}(Y)).
\]
Thus,
\[
\boldsymbol{\cg{A}}({\bf X} + {\bf Y}) = \cg{A}(Z) \geq \min(\cg{A}(X), \cg{A}(Y)) \geq \min(\boldsymbol{\cg{A}}({\bf X}), \boldsymbol{\cg{A}}({\bf Y})).
\]
This shows item (1).  Item (2) follows immediately from the fact that $CF^*(K_n)$ is finite-dimensional over $\Lambda$.

\end{proof}

We begin the computation of the kernel of $\Phi_a:HF^*(H_0)\longrightarrow SH^*_a(H)$ with the following Lemma.
\begin{lemma}
\label{lem:kernelincludes}
Let $\lambda$ be an eigenvalue of $\boldsymbol{c_{-1}^G}\circ\boldsymbol{\cg{S}}$ with $ev(\lambda) > k\pi R^2$.  For any constant $\kappa \in\Lambda$ and any basis vector ${\bf v_i^{\lambda}}$ in $\cg{B}$, 
\[
\Phi_a(\kappa{\bf v_i}^{\lambda}) =0.
\]
\end{lemma}

\begin{proof}
From Equation (\ref{eq:piphinonsense}), it suffices to show that
\[
\mathfrak{c}_a^{KH}\circ \Phi_a^K(\kappa {\bf v_i^{\lambda}}) = 0.
\]
Indeed, we will show that
\begin{equation}
    \label{eq:zero}
\Phi_a^K(\kappa {\bf v_i^{\lambda}}) := \pi_a^K\circ\Phi^K(\kappa {\bf v_i^{\lambda}}) = 0.
\end{equation}
If $ev(\lambda) = \infty$, that is, if $\lambda = 0$, this follows immediately from Corollary \ref{cor:ritter1}.  So suppose $ev(\lambda) < \infty$.  We will construct a cocycle $X\in SC^*(K)$ such that
\[
\Phi^K(\kappa {\bf v_i^{\lambda}}) = [X]
\]
and
\[
\pi_a^K([X]) = 0.
\]
Let $X_0 \in CF^*(H_0) = CF^*(K_0)$ be a cocycle representing $\kappa {\bf v_i^{\lambda}}$.  Define a sequence $\left\{X_n{\bf q}\in CF^*(K_n){\bf q}\right\}_{n\in\NN}$ by defining the base case
\[
\dd_{\bf q}(X_0{\bf q}) = X_0
\]
and inductively defining
\[
X_n{\bf q} = c_{n-1}^{K}(X_{n-1}){\bf q}.
\]
Because the Floer differential commutes with continuation maps, 
\[
\dd^{fl}(X_n) = c_{n-1}^K\circ...\circ c_0^K\circ \dd^{fl}(X_0) = 0.
\]
Thus,
\begin{align*}
\dd\left(\sum_{n=0}^N X_n{\bf q}\right) &= X_0 - c_N^K(X_N) + \sum_{n=1}^{N} X_n - c_{n-1}^K(X_{n-1}) \\
&= X_0 - X_{N+1} + \sum_{n=1}^{N} X_n - X_n \\
&= X_0 - X_{N+1},
\end{align*}
and so
\[
[X_0] = [X_n]
\]
for any $n$.  We will show that, for $n$ large, $\cg{A}(X_n) > a$.  This implies that $\pi_a^K([X_n]) = 0$.  

Denote by $HF^*_{\lambda}(H_0)$ the invariant subspace corresponding to $\lambda$.  In the fixed Jordan basis, the matrix of $\boldsymbol{c_{-1}^G\circ\cg{S}}\big|_{HF^*_{\lambda}(H_0)}$ is
\[
A =
\left[
\begin{array}{cccccc}
\lambda & 1 &  &  & & \\
 & \lambda & 1 & & 0 & \\
& & \ddots &\ddots &  & \\
&0 & &\ddots &\ddots & \\
 &  & & & \lambda & 1\\
& & & & & \lambda 
\end{array}
\right].
\]
Ignoring non-zero $\CC$-scalar factors (these do not affect action),
\begin{equation}
\label{eq:actionofA}
A^{k_{\lambda}} = \left[
\begin{array}{cccccc}
\lambda^{k_{\lambda}} & \lambda^{k_{\lambda} - 1} &  & \dots & & \lambda\\
 & \lambda^{k_{\lambda}} & \lambda^{k_{\lambda} - 1} &  & \dots & \lambda^2\\
& & \ddots &\ddots &  &\vdots \\
&0 & &\ddots &\ddots & \\
 &  & & & \lambda^{k_{\lambda}} & \lambda^{k_{\lambda} - 1}\\
& & & & & \lambda^{k_{\lambda}}
\end{array}
\right]
\end{equation}
and
\[
A^N = \lambda^{N-k_{\lambda}}A^{k_{\lambda}}
\]
for all $N \geq k_{\lambda}$.
View $X_0\in CF^*(G_0)$ under the identification $K_0 = G_0$.  Again ignoring non-zero $\CC$-scalar factors,
\[
A^N\kappa\boldsymbol{v_i^{\lambda}} = \kappa \lambda^{N-k_{\lambda}}\sum_{j=1}^i \lambda^{k_{\lambda} - i + j}\boldsymbol{v_j^{\lambda}}
\]
Using Property (1) of Lemma \ref{lem:seminorm}, we compute
\begin{align}
\boldsymbol{\cg{A}}((\boldsymbol{c_{-1}^G\circ\cg{S}})^N(\kappa\boldsymbol{v_i^{\lambda}})) &= \boldsymbol{\cg{A}}\left(\kappa \lambda^{N-k_{\lambda}}\sum_{j=1}^i \lambda^{k_{\lambda} - i + j}\boldsymbol{v_j^{\lambda}}\right) \\
&\geq ev(\kappa \lambda^{N-k_{\lambda}}) + \min_{j}\boldsymbol{\cg{A}}(\lambda^{k_{\lambda} -i + j}\boldsymbol{v_j^{\lambda}}) \\ 
&\geq (N-k_{\lambda})\cdot ev(\lambda) + ev(\kappa) + (k_\lambda - i)\cdot ev(\lambda) + \min_j\cg{A}(\boldsymbol{v_j^{\lambda}}) \\
&\geq N\cdot ev(\lambda) + ev(\kappa) + \boldsymbol{\cg{A}}(\boldsymbol{v_i^{\lambda}}) - \boldsymbol{\cg{A}}(\boldsymbol{v_i^{\lambda}}) + \min_j\boldsymbol{\cg{A}}(\boldsymbol{v_j^{\lambda}}) - i\cdot ev(\lambda)\\
&= N\cdot ev(\lambda) + \boldsymbol{\cg{A}}(\kappa\boldsymbol{v_i^{\lambda}}) - \boldsymbol{\cg{A}}(\boldsymbol{v_i^{\lambda}}) + \min_j\boldsymbol{\cg{A}}(\boldsymbol{v_j^{\lambda}}) - i\cdot ev(\lambda) \\
&= \label{eq:actionbound2} N\cdot ev(\lambda) + \boldsymbol{\cg{A}}(\kappa\boldsymbol{v_i^{\lambda}}) + C,
\end{align}
where
\[
C = -\boldsymbol{\cg{A}}(\boldsymbol{v_i^{\lambda}}) + \min_j\boldsymbol{\cg{A}}(\boldsymbol{v_j^{\lambda}}) - i\cdot ev(\lambda).
\]
Let $X_N' = \cg{S}^{-N}\circ(c_{-1}^G\circ\cg{S})^N(X_0)\in CF^*(G_N).$  By Theorem \ref{theorem:rittercommutes}, 
\[
{X_N'} = c_{N-1}^G\circ c_{N-2}^G\circ ...\circ c_0^G(X_0).
\]
The map $\cg{S}$ preserves action, and so
\[
\cg{A}(X_N') = \cg{A}((c_{-1}^G\circ\cg{S})^N(X_0)).
\]
The representatives of $[(c_{-1}^G\circ\cg{S})^N(X_0)]$ and $[X_N']$ are in bijective correspondence through the action-preserving map $\cg{S}^{-N}$.  Thus,
\begin{equation} 
\label{eq:actionbound1}
\boldsymbol{\cg{A}}([X_N']) = \boldsymbol{\cg{A}}([(c_{-1}^G\circ\cg{S})^N(X_0)]).
\end{equation}
Recall that $[X_0] = \kappa\boldsymbol{v_i^{\lambda_j}}$.  As $(c_{-1}^G\circ\cg{S})^N$ is a chain map,
\[
[(c_{-1}^G\circ\cg{S})^N(X_0)] = \boldsymbol{(c_{-1}^G\circ\cg{S})}^N(\kappa\boldsymbol{v_i^{\lambda_j}}).
\]
Combining (\ref{eq:actionbound2}) and (\ref{eq:actionbound1}),
\[
\boldsymbol{\cg{A}}([X_N']) \geq \boldsymbol{\cg{A}}(\kappa\boldsymbol{v_i^{\lambda}}) + N\cdot ev(\lambda) + C.
\]
The canonical identification
\[
CF^*(G_N)\xrightarrow{\simeq}CF^*(K_N)
\]
descends to a map on cohomology that sends $[X_N']$ to $[X_N]$.  This isomorphism decreases action on-the-nose by $-Nk\pi R_N^2$.  The action of $[X_N]$ is bounded by
\[
\boldsymbol{\cg{A}}([X_N]) \geq \boldsymbol{\cg{A}}([X_N']) - Nk\pi R_N^2,
\]
and so
\[
\boldsymbol{\cg{A}}([X_N]) \geq N\cdot ev(\lambda) - Nk\pi R_N^2 + \boldsymbol{\cg{A}}(\kappa\boldsymbol{v_i^{\lambda}}) + C.
\]
The difference $ev(\lambda) - k\pi R_N^2 > 0$ is bounded below by the positive number $ev(\lambda) - k\pi R^2 > 0$, and $\boldsymbol{\cg{A}}(\kappa\boldsymbol{v_i^{\lambda}}) + C$ is a constant.  There therefore exists $n$ satisfying
\[
n(ev(\lambda) - k\pi R_N^2) + \boldsymbol{\cg{A}}(\kappa\boldsymbol{v_i^{\lambda}}) + C > a.
\]
Thus,
\[
\cg{A}(X_n) \geq \boldsymbol{\cg{A}}({\bf X_n}) > a.
\]
\end{proof}

\begin{lemma}
\label{lem:kercontains}
Denote by $HF^*_{\lambda}(H_0)$ the invariant subspace of $HF^*(H_0)$ corresponding to an eigenvalue $\lambda$ of ${\bf c_{-1}^G}\circ\boldsymbol{\cg{S}}$.
\[
\bigoplus_{ev(\lambda) > k\pi R^2} HF^*_{\lambda}(H_0)\subset \ker(\Phi_a).
\]
\end{lemma}

\begin{proof}
The map $\Phi_a$ is a $\Lambda_0$-module homomorphism, where $\Lambda_0 = ev^{-1}([0, \infty])$.  Any element \[
X\in \bigoplus\limits_{ev(\lambda) > k\pi R^2} HF^*_{\lambda}(H_0)
\]
can be written as the sum of generalized eigenvectors $\kappa{\bf v_i}^{\lambda}$ with $ev(\lambda) > k\pi R^2$ and $\kappa\in\Lambda$.  Lemma \ref{lem:kernelincludes} shows that 
\[
\Phi_a(\kappa{\bf v_i}^{\lambda}) = 0.
\]
By linearity, $\Phi_a(X) = 0$ as well.

\end{proof}

\subsection{Step II}
\label{subsec:step2}
We want to show that, as $a\rightarrow\infty$, the projection of $\ker(\Phi_a)$ onto
\[
\bigoplus_{ev(\lambda) \leq k\pi R^2}HF^*_{\lambda}(H_0)
\]
is zero.  Denote by $HF^*_{\lambda, (a, \infty)}(H_0)$ the intersection of $HF^*_{\lambda}(H_0)$ with the image of the map
\[
HF^*_{(a, \infty)}(H_0)\longrightarrow HF^*(H_0).
\]

\begin{proposition}
\label{prop:kernel}
The kernel of $\Phi_a:HF^*(H_0)\longrightarrow SH^*_a(H)$ satisfies the following inclusions:
\[
\bigoplus_{ev(\lambda) > k\pi R^2}HF^*_{\lambda}(H_0) \oplus\bigoplus_{ev(\lambda) \leq k\pi R^2} HF^*_{\lambda, (a, \infty)}(H_0) \subset \ker(\Phi_a)
\]
and
\[
\ker(\Phi_a) \subset \bigoplus_{ev(\lambda) > k\pi R^2}HF^*_{\lambda}(H_0) \oplus\bigoplus_{ev(\lambda) \leq k\pi R^2} HF^*_{\lambda, (a - C, \infty)}(H_0),
\]
where $C > 0$ is a constant depending only upon the fixed Jordan basis $\cg{B}$.
\end{proposition}
The proof of Proposition \ref{prop:kernel} relies on five Lemmas that bound action.

\subsubsection{Lemmas bounding action}
\begin{lemma}
\label{lem:actionbounds}
There exists a constant $C > 0$ such that, for any cocycles $X, Y \in CF^*(H_0)$ satisfying
\[
[X] = [Y],
\]
the difference in actions is no more than $C$, that is,
\[
|\cg{A}(X) - \cg{A}(Y)| \leq C.
\]
\end{lemma}

\begin{proof}
This follows immediately from property (2) of Lemma \ref{lem:seminorm} and the fact that $HF^*(H_0)$ is finitely generated.

\end{proof} 

Let 
\[\mathfrak{c}_n^{HG}:CF^*(H_n)\longrightarrow CF^*(G_n)
\]
be an action-increasing map.
\begin{lemma}
\label{lem:valincreased}
The valuations of the Novikov-valued coefficients are not decreased by $\mathfrak{c}^{HG}_n$.
\end{lemma}

\begin{proof}
Let $G_n^s$ be a monotone homotopy between $H_n$ and $G_n$ of the form
\[
G_n^s = \mathfrak{h}_n^s(k\pi r^2) + (1 + k\pi r^2)\rho^*f,
\]
where $\mathfrak{h}_n^s:\RR\longrightarrow\RR$ is a monotone-decreasing homotopy between $h_n$ and $g_n$.

Let $J_s$ be a family of admissible almost-complex structures.  A solution $u(s, t)$ of the Floer equation
\[
\frac{\dd u}{\dd s} + J_s\left(\frac{\dd u}{\dd t} - X_{H_n^s}\right) = 0
\]
descends to a solution $v = \rho\circ u(s, t)$ of the Floer equation
\begin{equation}
\label{eq:floerf}
\frac{\dd v}{\dd s} + \rho_*J_s\left(\frac{\dd v}{\dd t} - X_f\right) = 0
\end{equation}
on the base manifold $M$.  Let $\tilde{v}$ be the compactification of $v$ by its limit points, and let $\tilde{u}$ be the compactification of $u$ by the limiting fiber disks, so that $\rho_*(\tilde{u}) = \tilde{v}$.  Since 
\[
\Omega([\tilde{u}]) = \rho^*\omega([\tilde{u}]) = \omega([\tilde{v}]),
\]
$u$ contributes a weight 
\[
T^{-\Omega[\tilde{u}]} = T^{-\omega[\tilde{v}]}
\]
to $c^{HG}_n$.  Action is increased by solutions of (\ref{eq:floerf}), and so
\[
0 \leq -\omega([\tilde{v}]) + \lim_{s\rightarrow-\infty}f(v(s, t)) - \lim_{s\rightarrow\infty}f(v(s, t)).
\]
In particular,
\begin{equation}
\label{eq:fbound}
-\max_{x, y\in M}|f(x) - f(y)| \leq -\omega([\tilde{v}]).
\end{equation}
By assumption, $f$ is $\cg{C}^2$ small; assume that $f > 0$ and $f$ is small enough that
\begin{equation}
\label{eq:boundonf}
-\max_{x\in M}f(x) > -\min_{\substack{A\in\pi_2(M) \\ \omega(A)\neq 0}}|\omega(A)|.
\end{equation}
(\ref{eq:fbound}) and (\ref{eq:boundonf}) imply 
\[
-\min_{\substack{A\in\pi_2(M) \\ \omega(A)\neq 0}}|\omega(A)| < -\omega([\tilde{v}]),
\]
or
\[
0 \leq -\omega([\tilde{v}]).
\]
Applying this argument to every $n$, the Lemma follows.

\end{proof}

Let $\mu_1, \mu_2, \dots \mu_{\ell}$ be eigenvalues of the operator
\[
c_{-1}^G\circ\cg{S}:CF^*(G_0)\longrightarrow CF^*(G_0),
\]
indexed with multiplicities.  Let $B = \{u_1^{\mu_1}, u_2^{\mu_1}, ..., u_{j_{\mu_{\ell}}}^{\mu_{\ell}}\}$ be a Jordan eigenbasis for $c_{-1}^G\circ\cg{S}$.  Write the $\mu$th-generalized eigenspace as $CF_{\mu}^*(G_0)$.

Let $X \in CF^*(G_0)$, and write $X$ in the basis $B$.
\[
X = \sum_{i, j}c_{ij}u_i^{\mu_j},
\]
where $c_{ij}\in\Lambda$ and the range of $i$ depends on $j$.  Now write
\[
X = \sum_{x_n\in\cg{P}(H_0)} d_nx_n,
\]
where $d_n\in\Lambda$. 

\begin{lemma}
\label{lem:actionbounds2}
There exists a constant $C\in\RR$, depending only on the choice of bases, such that
\[
\cg{A}(X) \leq \min_{i, j}\cg{A}(c_{ij}u_i^{\mu_j}) + C.
\]
\end{lemma}

\begin{proof}
Let $e_{ijn}$ be constants determining the change-of-basis $B\mapsto\{x_n\}$, so that
\[
x_n = \sum_{i, j}e_{ijn}u_i^{\mu_j}.
\]
Then
\begin{align*}
\sum_{i, j}c_{ij}u_i^{\mu_j} &= X
= \sum_{n} d_nx_n 
= \sum_n d_n\sum_{i, j}e_{ijn}u_i^{\mu_j} 
= \sum_{i, j} \left(\sum_n d_ne_{ijn}\right)u_i^{\mu_j},
\end{align*}
and so
\[
c_{ij} = \sum_n d_ne_{ijn}.
\]
Take
\[
C = \min_{ijn}\cg{A}(e_{ijn}) + \min_{ij}\cg{A}(u_i^{\mu_j}) - \max_{x\in M}f(x).
\]
It follows that
\begin{align*}
\cg{A}(c_{ij}u_i^{\mu_j}) &\geq\min_{n}\cg{A}(d_ne_{ijn}u_i^{\mu_j}) \\
&\geq \left(\min_n\cg{A}(d_n)\right) + \left(\min_n\cg{A}(e_{ijn})\right) + \left(\min_{i, j}\cg{A}(u_i^{\mu_j})\right)\\
&\geq \cg{A}(X) + C.
\end{align*}
The inequality given by the last line follows from the definition of the action of $X$ and the observation that
\[
\min_{ijn}\cg{A}(e_{ijn}) \leq \min_n \cg{A}(e_{ijn})
\]
for any fixed $i, j$.
\end{proof} 

\begin{lemma}
\label{lem:eigeneval}
If $\mu$ is an eigenvalue of the map
\[
c_{-1}^G\circ\cg{S}: CF^*(H_0)\longrightarrow CF^*(H_0)
\]
then
\[
ev(\mu)\geq 0.
\]
\end{lemma}
\begin{proof}
Let $u^{\mu}\in\cg{B}$ be the eigenvector satisfying
\[
c_{-1}^G\circ\cg{S}(u^{\mu}) = \mu u^{\mu}.
\]
Let $G_0'$ be an upward shift of the Hamiltonian $G_0$:
\[
G_0' = G_0 + \delta
\]
for some small $\delta > 0$.  Let $c_{-1}^{G'}$ be an action-increasing continuation map
\[
c_{-1}^{G'}: CF^*(G_{-1})\longrightarrow CF^*(G_0').
\]
Modifying $c_{-1}^G$ if necessary, we assume without loss of generality that the diagram
\[
\begin{tikzcd}
CF^*(G_{-1}) \arrow{rr}{c_{-1}^{G'}} \arrow{dr}{c_{-1}^G} & & CF^*(G_0') \\
& CF^*(G_0) \arrow{ur}{=} &
\end{tikzcd}
\]
commutes.  Clearly,
\begin{align*}
ev(\mu) + \cg{A}_{G_0}(u^{\mu}) &= \cg{A}_{G_0}(\mu u^{\mu}) \\
&= \cg{A}_{G_0'}(\mu u^{\mu}) - \delta \\
&\geq \cg{A}_{G_{0}}(u^{\mu}) - \delta, 
\end{align*}
where the last inequality follows from the fact that action is increased by $c_{-1}^{G'}\circ\cg{S}$.  Thus, 
\begin{equation}
\label{eq:lambdageq0}
ev(\mu) \geq -\delta.
\end{equation}
As (\ref{eq:lambdageq0}) holds for any $\delta>0$, we conclude that
\[
ev(\mu) \geq 0.
\]
\end{proof}
\\
Finally, we omit the proof of the following standard linear algebra Lemma.
\begin{lemma}[Banach open mapping theorem]
\label{lem:bomt}
If $M:CF^*(H_0)\longrightarrow CF^*(H_0)$ is an invertible $\Lambda$-linear map, then there exists a constant $C$ such that, for all $X\in CF^*(H_0)$,
\[
\cg{A}(M(X)) \leq \cg{A}(X) + C.
\]
\end{lemma}




\subsubsection{Proof of Proposition \ref{prop:kernel}}
Recall the action-increasing maps
\[
\mathfrak{c}_n^{HG}:CF^*(H_n)\longrightarrow CF^*(G_n).
\]
The following Lemma simplifies the computation of the ulterior Lemma \ref{lem:smalleigenaction}.
\begin{lemma}
\label{lem:contcommute}
The continuation maps $\mathfrak{c}_n^{HG}$ may be chosen to commute at the level of cochains with the continuation maps $c_n^G$ and $c_n$, that is,
\[
\mathfrak{c}_n^{HG}\circ c_{n-1} = c_n^G\circ\mathfrak{c}_{n-1}^{HG}.
\]
\end{lemma}

\begin{proof}
We sketch the proof.  Let $g_n^s:\RR\rightarrow\RR$ be the map interpolating between the functions $g_n$ and $g_{n+1}$ in the definition of $G_n$ and $G_{n+1}$.  Define the continuation maps inductively as follows.  Denote by $\mathfrak{h}_1^s$ a generic map interpolating between $h_1$ and $g_1$ so that
\begin{enumerate}
    \item $\mathfrak{h}_1^s = g_1^s$ on $[0, \frac{k\pi R_1^2}{2}]$ and
    \item $\mathfrak{h}_1^s$ is monotone-decreasing in $s$ elsewhere.
\end{enumerate}
Define $\mathfrak{c}_1^{HG}$ through $\mathfrak{h}_1^s$.  Now let $\mathfrak{k}_n$ be the function that is
\begin{enumerate}
    \item equal to $g_n$ on $[0, k\pi R_n^2]$ and
    \item a translation of $h_n$ on $[k\pi R_n^2, \infty)$.
\end{enumerate}
See the teal curve in Figure \ref{fig:cc=cc}.  Let $\mathfrak{h}_n^s$ be a generic function interpolating between $h_n$ and $\mathfrak{k}_n$ that is
\begin{enumerate}
    \item equal to $\mathfrak{h}_{n-1}^s$ on $[0, k\pi R_{n-1}^2]$ and
    \item translation by a constant on $(k\pi R_{n-1}^2, \infty)$.
\end{enumerate}
Define a continuation map $\mathfrak{c}_n^{HK}$ through $\mathfrak{h}_n$.

Let $\mathfrak{k}_n^s$ be a generic function interpolating between $\mathfrak{k}_n$ and $g_n$ that is
\begin{enumerate}
    \item equal to $g_{n-1}^s$ on $[0, k\pi R_{n-1}^2]$ and
    \item monotone-decreasing in $s$ on $[k\pi R_{n-1}^2, \infty)$.
\end{enumerate}
Define a continuation map $\mathfrak{c}_n^{KG}$ through $\mathfrak{k}_n^s$.  Set
\[
\mathfrak{c}_n^{HG} = \mathfrak{c}_n^{KG}\circ\mathfrak{c}_n^{HK}.
\]
The Lemma follows from the integrated maximum principal of Lemma \ref{lem:intmax}.
\begin{figure}[htpb!]
\centering
\begin{tikzpicture}[scale=5]
\draw (0, -1.03660231) -- (0, -.07) node [left = 1pt] {$\RR$};
\draw (0, -1.03660231) -- (1, -1.03660231)node [right = 1pt] {$k\pi r^2$};
\draw[violet] (0, -1.03660231) -- (.267948, -1);
\draw[-, violet] (.267948, -1) -- (.641751, -.784185)node [right=1pt]{} ;
\draw[->, violet]  (.641751, -.784185) -- (1, -.37797)node [right=1pt]{$h_2$} ;
\draw[-, teal] (0, -1.03660231) -- (.641751, -0.66608764836)node [right=1pt] {};
\draw[->, teal]  (.641751, -0.66608764836) -- (1, -0.25987264836)node [right=1pt] {$\mathfrak{k}_2$};
\draw[->, frenchblue]  (0, -1.03660231) -- (0.85246559323, -.07)node [right=1pt] {$g_2$};
\end{tikzpicture}
\caption{Defining $\mathfrak{c}_2^{HG}$ through the intermediary function $\mathfrak{k}_2$}
\label{fig:cc=cc}
\end{figure}
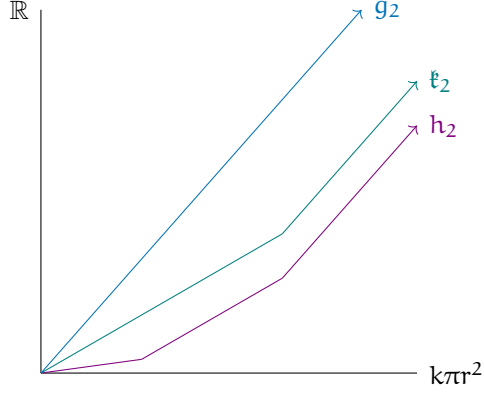

\end{proof}

Denote by 
\[
\pi_{\mu}:CF^*(G_0)\longrightarrow CF^*_{\mu}(G_0)
\]
the projection onto the $\mu$th-generalized eigenspace.  The following Lemma bounds the action of a cochain in an arbitrary Floer complex $CF^*(H_n)$ by the action of a cochain in $CF^*(H_0)$.
\begin{lemma}
\label{lem:smalleigenaction}
Let $\mu$ be an eigenvalue with $ev(\mu)\leq k\pi R^2$.  Let $X\in CF^*(H_n)$ be a cocycle such that 
\[
Y_n := \pi_{\mu}\circ \cg{S}^{n}\circ\mathfrak{c}^{HG}(X) \neq 0.
\]
Let $Y_0$ be the unique element in $CF^*_{\mu}(G_0)$ satisfying $(c_{-1}^G\circ\cg{S})^n(Y_0) = Y_n$.  There exists a constant $\cg{C}$, independent of $\mu$, satisfying
\[
\cg{A}(X) - \cg{A}(Y_0) < \cg{C}.
\]
\end{lemma}

Before embarking on the proof of Lemma \ref{lem:smalleigenaction} in full generality, we illustrate the idea of the proof with the following simplified scenario.  Suppose that $Y_0 = y\in \cg{P}(H_0)$ is a bona fide $\mu$-eigenvector, and $X = x\in\cg{P}(H_n)$ is a cocycle of winding number $\mathfrak{w}(x) = n$ satisfying
\[
\cg{S}^n\circ\mathfrak{c}_n^{HK}(x) = \mu^n\kappa y = (c_{-1}^G\circ\cg{S})^n(\kappa y)
\]
for some $\kappa\in\Lambda$.  By Lemma \ref{lem:eigeneval}, the valuation of the Novikov coefficient is increased by $\mathfrak{c}_n^{HG}$ (in this case the initial Novikov coefficient is $1$ and has valuation $ev(1) = 0$).  Therefore,
\[
0 = ev(1) \leq ev(\mu^n\kappa) = n\cdot ev(\mu) + ev(\kappa).
\]
Assume that the circle bundle on which $x$ lives has radius close to $R$.  The action of $x$, compared to the action of $\kappa y$, is
\begin{align*}
\cg{A}(x) - \cg{A}(\kappa y) &\simeq -k\pi n R^2 - ev(\kappa) \\
&\leq -k\pi n R^2 + n\cdot ev(\mu) \\
&= n(ev(\mu) - k\pi R^2) \\
&< 0
\end{align*}
if $ev(\mu) < k\pi R^2$.  Taking $\cg{C} = 0$, this is precisely the statement of Lemma \ref{lem:smalleigenaction}.  Of course, we can never assume that we are working in such a simplified scenario.  Namely, we must take $X$ to be a weighted sum of periodic orbits and $Y_0$ to be a weighted sum of generalized eigenvectors.

\begin{proof}
    Write 
    \[
    X = \sum_{k} \kappa_kx_k,
    \]
    where $\kappa_k\in\Lambda$ and $x_k\in\cg{P}(H_n)$.
    Because $\mathfrak{c}_n^{HG}$ is a $\Lambda$-linear morphism,
    \[
    \cg{S}^{n}\circ\mathfrak{c}^{HG}_n(X) = \sum_{k=1}^m \cg{S}^{}\circ\mathfrak{c}^{HG}_n(\kappa_kx_k).
    \]
    Write $\pi_{\mu}\circ\cg{S}^{n}\circ\mathfrak{c}^{HG}_n(\kappa_kx_k)$ in the basis $B$:
    \begin{equation} 
    \label{eq:jnfkappax}
    \pi_{\mu}\circ \cg{S}^{n}\circ\mathfrak{c}^{HG}_n(\kappa_kx_k) = \sum_{j=1}^{j_{\mu}}\kappa_{k, j, }u_j^{\mu}
    \end{equation}
    for some $\{\kappa_{k, j}\in\Lambda\}$.
    By assumption,
    \begin{equation} 
    \label{eq:kuequalsyo}
    \sum_{j=1}^{j_{\mu}}\left(\sum_k\kappa_{k, j}\right)u_j^{\mu} = \pi_{\mu}\circ\cg{S}^n\circ\mathfrak{c}_n^{HG}(X) = (c_G^{-1}\circ\cg{S})^n(Y_0).
    \end{equation}
    We want to compare the action of some $\kappa_kx_k$ with the action of some $Y_0$.  We do this through the auxiliary cochain $Y_n$.  Let $A$ be the matrix of $c_{-1}^G\circ\cg{S}\big|_{CF^*_{\mu}(H_0)}$ in the fixed Jordan basis.  Then, up to non-zero $\CC$-scalars,
    \[
    Y_n = \mu^{n-j_{\mu}}A^{j_{\mu}}(Y_0).
    \]
    By Lemma \ref{lem:bomt} there exists a constant $C > 0$ such that
    \[
    \cg{A}(Y_n) = (n-j_{\mu})\cdot ev(\mu) + \cg{A}(A^{j_{\mu}}(Y_0)) \leq (n-j_{\mu})\cdot ev(\mu) + \cg{A}(Y_0) + C.
    \]
    Thus, by (\ref{eq:kuequalsyo}), there exists some $j$, which we denote by $\mathfrak{J}$, and some $k$, denoted by $\mathfrak{k}$, satisfying
    \begin{equation} 
    \label{eq:est1}
    ev(\kappa_{\mathfrak{k}, j}) \leq n\cdot ev(\mu) + \cg{A}(Y_0) + C',
    \end{equation}
    where
    \[
    C' = \max_{j}\left(-j_{\mu}\cdot ev(\mu) + C - \cg{A}(u_j^{\mu})\right).
    \]
    
    Denote by $w$ the winding number $\mathfrak{w}(x_{\mathfrak{k}})$ of $x_{\mathfrak{k}}$.  View $x_{\mathfrak{k}}$ as an element of $CF^*(H_w)$ under the inclusion $CF^*(H_w)\hookrightarrow CF^*(H_n)$.  By Lemma \ref{lem:contcommute},
    \begin{align}
    \label{eq:windingornot}
    \cg{S}^{n}\circ\mathfrak{c}^{HG}_n(\kappa_{\mathfrak{k}}x_{\mathfrak{k}}) &= \cg{S}^{n}\circ\mathfrak{c}^{HG}_n\circ c_{n-1}^H\circ...\circ c_w^H(\kappa_{\mathfrak{k}}x_{\mathfrak{k}}) \\
    &= (c_{-1}^G\circ\cg{S})^{n-w}\circ\cg{S}^{w}\circ \mathfrak{c}^{HG}_w(\kappa_{\mathfrak{k}}x_{\mathfrak{k}}).
    \end{align}
    Again appealing to the Jordan normal form, write
    \begin{equation} 
    \label{eq:jnf}
    \pi_{\mu}\circ \cg{S}^{w}\circ \mathfrak{c}^{HG}_w(\kappa_{\mathfrak{k}}x_{\mathfrak{k}}) = \sum_{j }\gamma_{j, \mathfrak{k}}u_j^{\mu}.
    \end{equation}
    The orbits in $\cg{P}(G_0)$ are constant and appear at energy level $H = f\circ\rho \geq 0$.  Lemma \ref{lem:valincreased} states that the valuation of the Novikov-coefficient is increased by $\mathfrak{c}_w^{HG}$.  The action of ${\cg A}(\mathfrak{c}_w^{HG}(\kappa_{\mathfrak{k}}x_{\mathfrak{k}}))$ is the valuation of the Novikov coefficient, plus the value of $f\circ\rho$. Together, this implies that
    \begin{equation} 
    ev(\kappa_{\mathfrak{k}}) \leq {\cg A}(\mathfrak{c}_w^{HG}(\kappa_{\mathfrak{k}}x_{\mathfrak{k}})).
    \end{equation}

    By Lemma \ref{lem:actionbounds2}, there is a constant $C''$ such that
    \[
    \cg{A}(\mathfrak{c}^{HG}(\kappa_{\mathfrak{k}}x_{\mathfrak{k}})) = \cg{A}(\cg{S}^w\circ\mathfrak{c}^{HG}(\kappa_{\mathfrak{k}}x_{\mathfrak{k}})) \leq\cg{A}(\gamma_{j, \mathfrak{k}}u_j^{\mu}) + C''
    \]
    for any $j$.  We conclude that
    \begin{equation}
        \label{eq:boundkappa}
       ev(\kappa_{\mathfrak{k}}) \leq \cg{A}(\gamma_{j, \mathfrak{k}}u_j^{\mu}) + C''
    \end{equation}
    for any $j$.
    
    Applying the map $(c_{-1}^G\circ\cg{S})^{n-w}$ to (\ref{eq:jnf}), ignoring non-zero $\CC$ scalars as always, yields
    \[
    \pi_{\mu}\circ (c_{-1}^G\circ\cg{S})^{n-w}\circ\cg{S}^{w}\circ \mathfrak{c}^{HG}_w(\kappa_{\mathfrak{k}}x_{\mathfrak{k}}) = \sum_{j=1}^{j_{\mu}}\sum_{q=j}^{\min(n-w+j, j_{\mu})}\mu^{n-w-q+j}\gamma_{q, \mathfrak{k}}u_j^{\mu}.
    \]
    Comparing this with (\ref{eq:jnfkappax}) through the equivalence (\ref{eq:windingornot}), and ignoring non-zero $\CC$-scalars, 
    \[
    \kappa_{\mathfrak{k},{\mathfrak{J}}} = \sum_{q={\mathfrak{J}}}^{\min(n-w+{\mathfrak{J}}, j_{\mu})}\mu^{n-w-(q-{\mathfrak{J}})}\gamma_{q, \mathfrak{k}}.
    \]
    There therefore exists some $q$, denoted by $Q$, with
    \begin{equation}
        \label{eq:est2}
        ev(\mu^{n-w+{\mathfrak{J}}-Q}\gamma_{Q, \mathfrak{k}}) \leq ev(\kappa_{\mathfrak{k}, {\mathfrak{J}}}).
    \end{equation}
Letting $Q=j$ in (\ref{eq:boundkappa}),
\begin{equation}
\label{eq:evcj}
ev(\kappa_{\mathfrak{k}}) \leq \cg{A}(\gamma_{Q, \mathfrak{k}}u_Q^{\mu}) + C''.
\end{equation}
Combining equations (\ref{eq:est1}), (\ref{eq:est2}), and (\ref{eq:evcj}), 
    \begin{align*} 
    ev(\kappa_{\mathfrak{k}}) &\leq \cg{A}(\mu^{w-n+Q-{\mathfrak{J}}}\kappa_{\mathfrak{k}, {\mathfrak{J}}}u_Q^{\mu}) + C''\\
    &= (w-n+Q-{\mathfrak{J}})\cdot ev(\mu) + ev(\kappa_{\mathfrak{k}, {\mathfrak{J}}}) + \cg{A}(u_Q^{\mu}) + C''\\
    &\leq (w-n+Q-{\mathfrak{J}})\cdot ev(\mu) + n\cdot ev(\mu) + \cg{A}(Y_0) + \cg{A}(u_Q^{\mu}) + C' + C''\\
    &= w\cdot ev(\mu) + \cg{A}(Y_0) + Q\cdot ev(\mu) + \cg{A}(u_Q^{\mu}) - {\mathfrak{J}}\cdot ev(\mu) + C' + C''.
    \end{align*}
    Recall that the winding number of $x_{\mathfrak{k}}$ is $\mathfrak{w}(x_{\mathfrak{k}}) = w$.  Let $x_{\mathfrak{k}}$ lie in the circle bundle of radius $r(x_{\mathfrak{k}})$.  Then
    \begin{align*}
        \cg{A}(\kappa_{\mathfrak{k}}x_{\mathfrak{k}}) &= ev(\kappa_{\mathfrak{k}}) + \cg{A}(x_{\mathfrak{k}}) \\
        &= ev(\kappa_{\mathfrak{k}}) - wk\pi r(x_{\mathfrak{k}})^2 + \int_0^1 H(x_{\mathfrak{k}})dt \\
        &\leq \cg{A}(Y_0) + w(ev(\mu) - k\pi r(x_{\mathfrak{k}})^2) + \int_0^1 H(x)dt + Q\cdot ev(\mu) + \cg{A}(u_Q^{\mu})- {\mathfrak{J}}\cdot ev(\mu) + C' + C''.
    \end{align*}
        Because $ev(\mu) < k\pi R^2$, and because $r(x_{\mathfrak{k}})$ approaches $R$ as $\mathfrak{w}$ approaches infinity, 
    \[
    w(ev(\mu) - k\pi r(x_{\mathfrak{k}})^2) \leq C'''
    \]
    for some fixed $C''' > 0$ independent of $x$.  
    Furthermore, there is a constant $C''''$ bounding the expression
    \[
    \int_0^1 H(x)dt + Q\cdot ev(\mu) + \cg{A}(u_Q^{\mu}) - {\mathfrak{J}}\cdot ev(\mu) + C' + C'' \leq C''''.
    \]
    Let 
    \[
    \cg{C} = C''' + C''''.
    \]
    Then
    \[
    \cg{A}(\kappa_{\mathfrak{k}}x_{\mathfrak{k}}) \leq \cg{A}(Y_0) + \cg{C}.
    \]
    By definition,
    \[
    \cg{A}(X) \leq \cg{A}(\kappa_kx_k)
    \]
    for each $k$.  Thus,
    \[
    \cg{A}(X) \leq \cg{A}(Y_0) + \cg{C}.
    \]
    The result follows.
    
\end{proof}

Lemma \ref{lem:smalleigenaction} shows us how cochains in a fixed eigensummand behave.  However, we would like a result for general cochains.  To this end, define $\pi_{\leq}$ to be the projection 
\[
\pi_{\leq}:CF^*(H_0)\rightarrow\bigoplus_{ev(\mu)\leq k\pi R^2}CF^*_{\mu}(H_0).
\]
Lemma \ref{lem:smalleigenaction} has the following easy corollary.

\begin{corollary}
\label{cor:allsummands}
Choose any $X\in CF^*(H_n)$ such that
\[
\pi_{\leq}\circ\cg{S}^n\circ\mathfrak{c}^{HG}(X) \neq 0.
\]
Let $Y\in\bigoplus\limits_{ev(\mu)\leq k\pi R^2}CF^*_{\mu}(H_0)$ satisfy
\[
(c_{-1}^G\circ\cg{S})^n(Y) = \pi_{\leq}\circ \cg{S}^n\circ\mathfrak{c}^{HG}(X).
\]
Then
\[
\cg{A}(X) - \cg{A}(Y) < C.
\]
\end{corollary}

\begin{proof}
First note that, if
\[
(c_{-1}^G\circ\cg{S})^n(Y) = \pi_{\leq}\circ\cg{S}^n\circ\mathfrak{c}^{HG}(X),
\]
then
\[
(c_{-1}^G\circ\cg{S})^n(\pi_{\mu}(Y)) = \pi_{\mu}\circ\cg{S}^n\circ\mathfrak{c}^{HG}(X)
\]
for each $\mu$ with valuation $ev(\mu)\leq k\pi R^2$.  By Lemma \ref{lem:smalleigenaction}, there exists a constant $C > 0$ such that
\[
\cg{A}(X) - \cg{A}(\pi_{\mu}(Y)) < C
\]
for all $\mu$ with $ev(\mu)\leq k\pi R^2$.  By the non-Archimedean property of $\cg{A}$,
\[
\cg{A}(Y) \geq \min_{\mu}\cg{A}(\pi_{\mu}(Y)).
\]
Combining these two inequalities,
\[
\cg{A}(X) - \cg{A}(Y) \leq \cg{A}(X) - \min_{\mu}\cg{A}(\pi_{\mu}(Y)) < C.
\]

\end{proof}

Finally, we would like a cohomological analogue of Corollary \ref{cor:allsummands}.  Denote by 
\[
\pi_{\lambda}:HF^*(G_0)\longrightarrow HF^*_{\lambda}(G_0)
\]
projection onto the $\lambda$-generalized eigenspace, and, abusing notation, denote by $\pi_{\leq}$ the projection
\[
\pi_{\leq}: HF^*(H_0)\rightarrow\bigoplus_{ev(\lambda)\leq k\pi R^2}HF^*(H_0).
\]
Let us recall some linear algebra facts about maps acting on finite-dimensional vector spaces.  
\begin{enumerate} 
\item Any eigenvalue of $\boldsymbol{c_{-1}^G\circ\cg{S}}$ is an eigenvalue of $c_{-1}^G\circ\cg{S}$. 
\item Restricting to the invariant subspace $\ker(\dd)$, we can abuse notation and assume that ${B}$ is an eigenbasis for $\ker(\dd)$.
\item Without loss of generality, we can assume that the collection of generalized eigenvectors $\{v_1^{\lambda}, ..., v_{k_{\lambda}}^{\lambda}\}$ are chosen to equal
\[
\left\{u_{j_{\lambda}-k_{\lambda}}^{\lambda} + \im(\dd), ..., u_{j_{\lambda}}^{\lambda} + \im(\dd)\right\}
\]
\item If $[X]\in HF^*_{\lambda}(H)$, then $\dd(\pi_{\lambda}(X)) = 0$.
\end{enumerate}
\begin{corollary}
\label{cor:almostdone!}
Let $X\in HF^*(H_n)$ such that
\[
\pi_{\leq}\circ\boldsymbol{\cg{S}}^n\circ\mathfrak{c}^{HG}(X) \neq0.
\]
Let $Y\in \bigoplus\limits_{ev(\lambda)\leq k\pi R^2}HF^*_{\lambda}(G_0)$ satisfy
\[
(\boldsymbol{c_{-1}^G\circ\cg{S}})^n(Y) = \pi_{\leq}\circ\boldsymbol{\cg{S}}^n\circ\mathfrak{c}^{HG}(X).
\]
Let $X_0$ be any cochain-level representative of $X$.  Then
\[
\cg{A}(X_0) - \boldsymbol{\cg{A}}(Y) < C.
\]
\end{corollary}

\begin{proof}
Let $Y_0$ be a minimum-action cochain-level representative of $Y$.  Note that
\[
Y_0 \in\left( \bigoplus_{ev(\lambda)\leq k\pi R^2}CF^*_{\lambda}(H_0)\right)\cup\im(\dd).
\]
Decompose $Y_0$ into components $Z_0$ and $W_0$, where 
\[
W_0\in \bigoplus_{ev(\lambda)\leq k\pi R^2}CF^*_{\lambda}(H_0)\hspace{1cm}\text{and}\hspace{1cm}Z_0\in\im(\dd).
\]
Clearly $W_0\in\ker(\dd)$ and
\[
[(c_{-1}^G\circ\cg{S})^n(W_0)] = \pi_{\leq}\circ[\cg{S}^n\circ\mathfrak{c}^{HG}(X_0)] = [\pi_{\leq}\circ \cg{S}^n\circ\mathfrak{c}^{HG}(X_0)].
\]
Indeed, there exists
\[
\beta\in \im(\dd)\cap \bigoplus_{ev(\lambda)\leq k\pi R^2}CF^*_{\lambda}(H_0)
\]
such that
\[
(c_{-1}^G\circ\cg{S})^n(W_0) + \beta = \pi_{\leq}\circ\cg{S}^n\circ\mathfrak{c}^{HG}(X_0).
\]
As $(c_{-1}^G\circ\cg{S})^n$ is an isomorphism on $CF^*_{\lambda\neq 0}(H_0)$ and $\im(\dd)$ is an invariant subspace, $\beta$ has a well-defined inverse $\beta^{-1}\in\im(\dd)$, and
\[
(c_{-1}^G\circ\cg{S})^n(W_0 + \beta^{-1}) = \pi_{\leq}\circ\cg{S}^n\circ\mathfrak{c}^{HG}(X_0).
\]
By Lemma \ref{lem:smalleigenaction}, there exists $C' > 0$ with
\[
\cg{A}(X_0) - \cg{A}(W_0 + \beta^{-1}) < C'.
\]
But $[Y_0] = [W_0 + \beta^{-1}]$, and so, by Lemma \ref{lem:actionbounds}, there exists $C'' > 0$ with
\[
|\cg{A}(Y_0) - \cg{A}(W_0 + \beta^{-1})| < C''.
\]
Setting $C = C' + C''$,
\[
\cg{A}(X_0) - \cg{A}(Y_0) < C.
\]

\end{proof}
\\
We are finally ready to prove Proposition \ref{prop:kernel}.

\begin{proofprop2}
By Lemma \ref{lem:kernelincludes}, 
\[
\bigoplus_{ev(\lambda) > k\pi R^2}HF^*_{\lambda}(H_0) \subset \ker(\Phi_a).
\]
Clearly
\[
\bigoplus_{\lambda \leq k\pi R^2}HF^*_{\lambda, (a, \infty)}(H_0) \subset \ker(\Phi_a).
\]
By the linearity of $\Phi_a$, it therefore suffices to show the inclusion
\[
\im\left(\pi_{\leq}\big|_{\ker(\Phi_a)}\right) \subset \bigoplus_{ev(\lambda) \leq k\pi R^2}HF^*_{\lambda, (a-C, \infty)}(H_0).
\]
Let 
\[
{\bf Y}\in \bigoplus_{ev(\lambda)\leq k\pi R^2}HF^*_{\lambda}(H_0)
\]  
Suppose that ${\bf Y}\in\ker(\Phi_a)$.  Let $X\in SC^*(H)$ be a cocycle descending in cohomology to
\[
[X] = \Phi({\bf Y}).
\]
As
\[
\pi_a([X]) = \Phi_a({\bf Y}) = 0,
\]
we can assume without loss of generality that
\begin{equation} 
\label{eq:actionygreatera}
\cg{A}(X) > a.
\end{equation}
Let $Y\in CF^*(H_0)$ be any representative of ${\bf Y}$, and let $n\in\NN$ any integer with $\mathfrak{w}(X) \leq n$.  Under the PSS map, 
\[
[X] = [Y]
\]
in $SH^*(H)$.  By homotopy-invariance of continuation maps,
\[
(\boldsymbol{c_{-1}^G\circ\cg{S}})^n([Y]) = \boldsymbol{\cg{S}}^n\circ\mathfrak{c}^{HG}([X]).
\]
By Corollary \ref{cor:almostdone!},
\[
\cg{A}(X) - \cg{A}(Y) < C,
\]
or
\[
\cg{A}(Y) > \cg{A}(X) - C > a - C. 
\]
It follows that
\[
{\bf Y} \in \im(HF^*_{(a-C, \infty)}(H_0)\rightarrow HF^*(H_0)).
\]
\end{proofprop2} 

\subsubsection{The full PSS map}
We have characterized $\ker(\Phi_a)$; we now want to characterize $\ker(\Phi)$.
\begin{proposition}
\label{prop:inverselimiso}
The isomorphism
\[
\im(\Phi_a)\cong\bigslant{HF^*(H_0)}{\ker(\Phi_a)}
\]
induces an isomorphism
\begin{equation}
\label{eq:inverselimiso}
\im(\Phi)\cong\bigslant{HF^*(H_0)}{\lim\limits_{\substack{\leftarrow \\ a}}\ker({\Phi_a})}
\end{equation}
where the connecting maps for
\[
\lim\limits_{\substack{\leftarrow \\ a}}\ker({\Phi_a})
\]
are inclusions.
\end{proposition}

\begin{proof}
Let
\begin{equation} 
\label{eq:connectingmap}
\pi_{a, a'}:SH^*_a(H)\longrightarrow SH^*_{a'}(H)
\end{equation}
be the connecting map in the inverse limit.  By definition,
\[
\Phi_{a'} = \pi_{a, a'}\circ \Phi_a.
\]
Thus,
\[
\ker(\Phi_{a'}) = \ker(\pi_{a, a'}\circ \Phi_a) \supset\ker(\Phi_a).
\]
Similarly,
\[
\pi_{a, a'}(\im(\Phi_{a}))=\im(\Phi_{a'}).
\]
Let
\[
{\iota_{a, a'}}:\ker({\Phi_a})\longrightarrow\ker({ \Phi_{a'}})
\]
be the inclusion.  Abusing notation, let
\[
\pi_{a, a'}:\im(\Phi_a)\longrightarrow \im(\Phi_{a'})
\]
be the surjective map induced by restricting the connecting map (\ref{eq:connectingmap}).
The diagram
\[
\begin{tikzcd}
0 \arrow{r} & \ker({ \Phi_a}) \arrow{r} \arrow{d}{{ \iota_{a, a'}}}& HF^*(H_0) \arrow{r} \arrow{d}{=} & \im(\Phi_a) \arrow{r} \arrow{d}{{ \pi_{a, a'}}} & 0 \\
0 \arrow{r} & \ker({ \Phi_{a'}}) \arrow{r} & HF^*(H_0) \arrow{r} & \im(\Phi_{a'}) \arrow{r} & 0 
\end{tikzcd}
\]
commutes.  As the inverse limit is left-exact, there is an exact sequence
\begin{equation}
\label{eq:exactkerphi}
0\longrightarrow \lim_{\substack{\leftarrow \\ a}}\ker({ \Phi_a}) \longrightarrow HF^*(H_0) \longrightarrow \lim_{\substack{\leftarrow \\ a}}\im(\Phi_a) \longrightarrow  \lim_{\substack{\leftarrow \\ a}}{}^1\ker({ \Phi_a}).
\end{equation}
We will show that $\lim\limits_{\substack{\leftarrow \\ a}}{}^1\ker({ \Phi_a}) = 0$.  

By Proposition \ref{prop:kernel} there exists some set of constants $\left\{C_{\lambda, a}\in\RR\right\}_{\lambda\leq k\pi R^2}$, such that
\[
\ker({\Phi_a}) = \bigoplus_{\lambda>k\pi R^2} HF^*_{\lambda}(H_0)\oplus\bigoplus_{\lambda\leq k\pi R^2} HF^*_{\lambda, (C_{\lambda, a}, \infty)}(H_0)
\]
and $\lim\limits_{a\rightarrow\infty} C_{\lambda, a} = \infty$.
Inverse limits commute with finite sums, so that
\begin{align*}
\lim\limits_{\substack{\leftarrow \\ a}}{}^1\ker({ \Phi_a}) &= \lim\limits_{\substack{\leftarrow \\ a}}{}^1  \bigoplus_{\lambda>k\pi R^2} HF^*_{\lambda}(H_0)\oplus \lim\limits_{\substack{\leftarrow \\ a}}{}^1\bigoplus_{\lambda\leq k\pi R^2} HF^*_{\lambda, (C_{\lambda, a}, \infty)}(H_0) \\
&= \lim\limits_{\substack{\leftarrow \\ a}}{}^1\bigoplus_{\lambda\leq k\pi R^2} HF^*_{\lambda, (C_{\lambda, a}, \infty)}(H_0).
\end{align*}
The second equality holds because the connecting maps of the system
\[
\lim\limits_{\substack{\leftarrow \\ a}}\bigoplus_{\lambda>k\pi R^2} HF^*_{\lambda}(H_0)
\]
are isomorphisms, and the system therefore satisfies the Mittag-Leffler condition.

Recall from Remark \ref{rem:countable} that the index $a$ belongs to a fixed countable sequence $a_1, a_2, a_3, ...$.  By definition,
\begin{align}
\label{eq:lim1}
\lim\limits_{\substack{\leftarrow \\ a}}{}^1\bigoplus_{\lambda\leq k\pi R^2} H&F^*_{\lambda, (C_{\lambda, a}, \infty)}(H_0) \\
&= \coker\left(\prod_{i=1}^{\infty} \bigoplus_{\lambda\leq k\pi R^2} HF^*_{\lambda, (C_{\lambda, a_i}, \infty)}\xrightarrow{\prod_i  {\iota_{a_i, a_{i+1}}}- {id}} \prod_{i=1}^{\infty}\bigoplus_{\lambda\leq k\pi R^2} HF^*_{\lambda, (C_{\lambda, a_i}, \infty)}(H_0) \right). \nonumber
\end{align}
Choose any element
\[
 \sum_{i=1}^{\infty} {\bf X_i}\in\prod_{i=1}^{\infty}\bigoplus_{\lambda\leq k\pi R^2} HF^*_{\lambda, (C_{\lambda, a_i}, \infty)}(H_0).
\]
For each $i\in\NN_+$, consider the formal sum
\[
{\bf Y_i} =  \sum_{j=i}^{\infty}(-1)^{j-i+1}{\bf X_j}.
\]
As ${\bf X_i}\in\bigoplus\limits_{\lambda\leq k\pi R^2} HF^*_{\lambda, (C_{\lambda, a_i}, \infty)}(H_0)$, it has a cochain representative $X_i$ such that $\cg{A}(X_i) >  \min\limits_{\lambda} C_{\lambda, a_i}$.  This implies that
\[
\lim_{i\rightarrow\infty}\cg{A}(X_i) = \infty.
\]
Thus,
\[
\sum_{j=i}^{\infty} (-1)^{j-i}X_j \in \bigoplus_{\lambda\leq k\pi R^2} CF^*_{\lambda}(H_0).
\]
More particularly,
\[
\sum_{j=i}^{\infty} (-1)^{j-i}X_j \in \bigoplus_{\lambda\leq k\pi R^2} CF^*_{\lambda, (C_{\lambda, a_i}, \infty)}(H_0),
\]
is a cochain representative whose cohomology class is given by ${\bf Y_i}$.  We conclude that 
\[
{\bf Y_i}\in \bigoplus\limits_{\lambda\leq k\pi R^2} HF^*_{\lambda, (C_{\lambda, a_i}, \infty)}(H_0)
\]
There is a telescoping sum
\begin{align*}
\prod_{i=1}^{\infty}\left(\iota_{a_i, a_{i+1}} - id\right)\left(\sum{\bf Y_i}\right) &= \sum_{i=1}^{\infty} {\bf Y_{i+1}} - {\bf Y_i} \\
&= \sum_{i=1}^{\infty} {\bf X_i} + \sum_{j=i+1}^{\infty} (-1)^{j-i}{\bf X_j} + (-1)^{j-i + 1 }{\bf X_j} \\
&= \sum_{i=1}^{\infty} {\bf X_i}.
\end{align*}
We have shown that the map $\prod_i\iota_{a_i, a_{i+1}} - id$ is surjective, and so, from the definition (\ref{eq:lim1}),
\[
\lim_{\substack{\leftarrow \\ a}}{}^1\ker( { \Phi_a}) = 0.
\]

\end{proof}

\begin{corollary}
\label{cor:imisquotient}
There is an isomorphism
\[
\im(\Phi) \simeq \bigslant{HF^*(H_0)}{\bigoplus\limits_{ev(\lambda) > k\pi R^2}HF^*_{\lambda}(H_0)}.
\]
\end{corollary}

\begin{proof}
It suffices to show the equality
\[
\lim_{\substack{\leftarrow \\ a}}\ker(\Phi_a) = \bigoplus_{\lambda > k\pi R^2}HF^*_{\lambda}(H_0).
\]
Proposition \ref{prop:inverselimiso} shows that the connecting maps of $\lim\limits_{\substack{\leftarrow \\ a}}\ker(\boldsymbol{ \Phi_a})$ are inclusions; and so the inverse limit is the intersection:
\[
\lim\limits_{\substack{\leftarrow \\ a}}\ker(\boldsymbol{ \Phi_a}) = \bigcap\limits_a \ker(\boldsymbol{ \Phi_a}).
\]
By Proposition \ref{prop:kernel},
\[
\ker(\boldsymbol{ \Phi_a}) = \bigoplus_{\lambda>k\pi R^2}HF^*_{\lambda}(H_0)\oplus\bigoplus_{\lambda\leq k\pi R^2}HF^*_{\lambda, (C_{a, \lambda}, \infty)}(H_0)
\]
for some constants $C_{a, \lambda}$ with limiting behavior $\lim\limits_{\substack{\leftarrow \\ a}}C_{a, \lambda}\rightarrow\infty$.  The only element of $CF^*(H_0)$ with arbitrarily large action is $0$, and so
\[
\bigcap_a  \bigoplus_{\lambda\leq k\pi R^2}HF^*_{\lambda, (C_{a, \lambda}, \infty)}(H_0) = 0.
\]
By definition, 
\[
HF^*_{\lambda}(H_0)\cap HF^*_{\lambda'}(H_0) = 0
\]
whenever $\lambda\neq\lambda'$.  The Corollary follows.

\end{proof}

\subsection{The product structure}
Recall that $SH^*(H)$ has a product structure, called the {\it pair of pants} product.  The product structure defined on an inverse limit of action-filtered Floer groups was first considered in \cite{cieliebak-o}.

Let $\Sigma$ be a Riemann surface with two positive punctures denoted by $p_1, p_2$ and one negative puncture denoted by $q$.   Choose collar neighborhoods $(0, \epsilon]\times S^1$ of both $p_1$ and $p_2$ and $[-\epsilon, 0)\times S^1$ of $q$.  Equip $\Sigma$ with a one-form $\beta$ such that $\beta = w_1dt$ on the collar neighborhood of $p_1$, $\beta = w_2dt$ on the collar neighborhood of $p_2$, and $\beta = w_0dt$ on the collar neighborhood of $q$, for fixed positive integers $w_0, w_1, w_2$ satisfying $w_0 \geq w_1 + w_2$.  As shown in \cite{ritter-fano}, we may assume that $\beta$ also satisfies
\[
d\beta \leq 0.
\]
For a fixed Hamiltonian $H_i$, let $u:\Sigma\longrightarrow E$ satisfy
\[
(du - \beta\otimes X_{H_i})^{(0, 1)} = 0,
\]
with respect to a generic $\Omega$-tame, upper-triangular almost-complex structure.  The energy of $u$ is defined as
\[
E(u) = \frac{1}{2}\int_{\Sigma} ||du - \beta\otimes X_{H_i}||^2 vol_{\Sigma}.
\]
We call $u$ a {\it pair-of-pants}.  The following Lemma is a standard application of the integrated maximum principal.

\begin{lemma}
\label{lem:intmaxprincproduct}
Suppose $u$ has finite energy.  The image of $u$ remains inside $D_R$. 
\end{lemma}
We omit the proof, which is a mash-up of the proof of Lemma \ref{lem:intmax} and Lemma 9.7 in \cite{ritter-fano}.

Fix a Hamiltonian $H_{\ell}$.  For a fixed index $i\in\NN$, define a {\it branch} to be a cascade 
\[
{\bf u^i} = (c_{m^i}^i, u_{m^i}^i, c_{m^i-1}^i, u_{m^i-1}, ..., c_0^i)
\]
associated to a sequence of periodic orbits $x_{m^i}^i, ..., x_0^i$ and a non-decreasing sequence of integer weights $w_{m^i}^i, ..., w_0^i\in\ZZ$, where
\begin{enumerate}
\item $c_j^i \in \im(x_j^i)$,
\item $u_j^i$ is a finite-energy Floer solution corresponding to an $s$-family of Hamiltonians $H_j^i$, where $H_j^i = w_j^iH_{\ell}$ when $s >> 0$ and $H_j^i = w_{j+1}^iH_{\ell}$ when $s << 0$,
\item $\lim\limits_{s\rightarrow\infty}u_j^i(s, 0)$ is in the stable manifold of $c_j^i$ (or $c_j^i = \lim\limits_{s\rightarrow\infty}u_j^i(s, 0)$ if $x_j^i$ is constant),
\item $c_j^i$ is in the unstable manifold of $\lim\limits_{s\rightarrow-\infty}u_{j+1}^i(s, 0)$ (or $c_j^i = \lim\limits_{s\rightarrow-\infty}u_{j+1}^i(s, t)$ if $x_j^i$ is constant).
\end{enumerate}
Define a {\it tree} ${\bf t}$ to be a sequence $({\bf u^0}, {\bf u^1}, {\bf u^2}, v)$ where $v$ is a pair-of pants corresponding to weights $w_0 = w_{m^0}^0$, $w_1 = w_0^{1}$, and $w_2 = w_0^{2}$; and where
\begin{enumerate}
\item taken in the collar neighborhood of $p_1$, $\lim\limits_{s\rightarrow\infty}v(s, 0)$ lies in the stable manifold of $c_0^1$,
\item taken in the collar neighborhood of $p_2$, $\lim\limits_{s\rightarrow\infty}v(s, 0)$ lies in the stable manifold of $c_0^2$,
\item and taken in the collar neighborhood of $q$, $\lim\limits_{s\rightarrow-\infty}v(s, 0)$ lies in the unstable manifold of $c_{m^0}^0$.
\end{enumerate}

Let $\cg{M}^n_0(x, y, z)$ be the moduli space of rigid trees ${\bf t} = ({\bf u^0}, {\bf u^1}, {\bf u^2}, v)$, where $c_{m^1}^1$ is in the stable manifold of $x$, $x_{m^2}^2$ is in the stable manifold of $y$, $x_0^0$ is in the stable manifold of $z$, and $w_{m^1}^1 = w_{m^2}^2 = 1$ and $w_0^0 = n$.  See Figure \ref{fig:tree}.

\begin{figure}[htbp!]
\centering
\includegraphics[scale=.5]{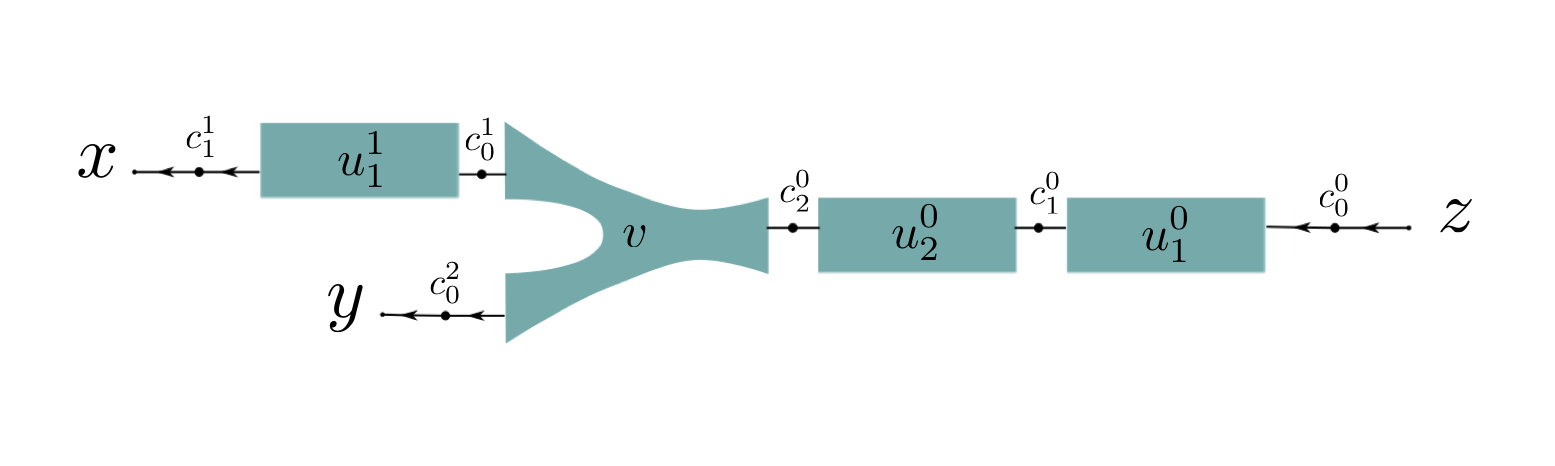}
\caption{A tree in $\cg{M}^n_0(x, y, z)$.}
\label{fig:tree}
\end{figure}

A standard Corollary to Lemma \ref{lem:intmaxprincproduct} is that 
\begin{corollary}
The moduli space $\cg{M}^n_0(x, y, z)$ is compact.
\end{corollary}

Define a map
\begin{align*}
\mu^{2, n}_i:CF^*(H_i)\otimes CF^*(H_i)&\longrightarrow CF^*(nH_i) \\
(x, y)&\mapsto\sum_{{\bf u}\in\cg{M}_0^n(x, y, z)}\pm T^{-\Omega((-\tilde{x}\#{\bf u}\#-\tilde{y})\#\tilde{z})}z,
\end{align*}
for each $n$.  Extend $\mu^{2, n}_i$ to a map
\[
\mu^{2, n}_{i, j}: CF^*(H_i)\otimes CF^*(H_j)\longrightarrow CF^*(nH_{i+j})
\]
where $\mu^{2, n}_{i, j}(x, y) = \mu^{2, n}_{i+j}(c^{j}(x), c^i(y))$.  Recall that the continuation maps were chosen to be inclusions, so, $\mu^{2, n}_{i, j}$ is just the composition of the inclusions
\[
CF^*(H_i)\hookrightarrow CF^*(H_{i+j})\hspace{1cm}\text{and}\hspace{1cm}CF^*(H_j)\hookrightarrow CF^*(H_{i+j})
\]
with the pair-of-pants product.

As $H_{i+j} \geq 0$, there are action-increasing continuation maps
\[
\eta_{i+j}:CF^*(nH_{i+j})\longrightarrow CF^*((n+1)H_{i+j})
\]
Omitting signs, define a map
\[
\mu^2: CF^*(H_i)[{\bf q}]\otimes CF^*(H_j)[{\bf q}]\longrightarrow CF^*(4H_{i+j})[{\bf q}]
\]
by
\begin{align*}
\mu(x, y) &=  \eta_{i+j}^2\circ\mu^{2, 2}_{i, j}(x, y) \\
\mu(x, y{\bf q}) = \mu(x{\bf q}, y) &= \eta_{i+j}^2\circ\mu^{2, 2}_{i, j}(x, y){\bf q} + \eta_{i+j}\circ \mu^{2, 3}_{i, j}(x, y) \\
\mu(x{\bf q}, y{\bf q}) &= \mu^{2, 4}_{i, j}(x, y).
\end{align*}
See \cite{abouzaid-s} for a careful treatment of signs.  We omit the proof of the following Lemma, which is a standard cylinder-breaking analysis.  See, for example, \cite{ritter-s} in the case of the usual non-degenerate product structure in symplectic cohomology and Appendix A in \cite{auroux} in the case of cascades.
\begin{lemma}
$\mu^2$ descends to a map
\[
{\mu^2}:SH^*(H)\otimes SH^*(H)\longrightarrow SH^*(4H)\cong SH^*(H).
\]
\end{lemma}
The energy of a map $u: \Sigma\longrightarrow E$ with positive punctures mapping to $x$ and $y$ and negative puncture mapping to $z$ is
\[
0 \leq E(u) = \cg{A}(T^{-\Omega((-\tilde{x}\#u\#-\tilde{y})\#\tilde{z})}z) - \cg{A}(x) - \cg{A}(y).
\]
Thus,
\[
\cg{A}(T^{-\Omega((-\tilde{x}\#u\#-\tilde{y})\#\tilde{z})}z) \geq \cg{A}(x) + \cg{A}(y).
\]
Because $4H$ is bounded, there is an isomorphism 
\[
\widehat{SC^*}(H)\simeq\widehat{SC^*}(4H).
\]
It follows that $\mu^2$ induces a well-defined map $\hat{\mu}_2$ on $\widehat{SC^*}(H)$, and, in particular, on $\widehat{\ker(\dd)}$.

\begin{lemma}
The map $\hat{\mu}_2$ induces a well-defined product on $\overline{SH^*}(H)$.
\end{lemma}

\begin{proof}
The ordinary product satisfies
\[
\dd\circ\mu_2(X, Y) = \mu_2\circ(\dd X, Y) + \mu_2\circ(X, \dd Y)
\]
for cochains $X, Y\in SC^*(H)$.  Suppose that $X, Y\in\widehat{\ker(\dd)}$.  Write
\[
X = \sum_i X_i\hspace{1cm}\text{and}\hspace{1cm} Y = \sum_j Y_j,
\]
where $X_i, Y_j$ are cocycles satisfying $\cg{A}(X_i) \rightarrow\infty$ and $\cg{A}(Y_j)\rightarrow\infty$.  Then
\begin{align*}
\dd\circ\hat{\mu}_2(X, Y) &= \dd\circ\sum_{i, j}\mu_2(X_i, Y_j) \\
&= \sum_{i, j}\dd\circ\mu_2(X_i, Y_j) \\
&= \sum_{i, j} \mu_2(\dd X_i, Y_j) + \mu_2(X_i, \dd Y_j) \\
&= 0.
\end{align*}
So $\hat{\mu}_2$ is well-defined on $\widehat{\ker(\dd)}$.  Now suppose $X\in\overline{\im(\dd)}$.  Write
\[
X = \sum_i \dd(Z_i)
\]
where $\cg{A}(\dd(Z_i))\rightarrow\infty$.  Let $Y\in\widehat{\ker(\dd)}$.  Then
\begin{align*}
\hat{\mu}_2(\dd(Z_i), Y)
&= \dd\circ\hat{\mu}_2(Z_i, Y) - \hat{\mu}_2(Z_i, \dd(Y)) \\
&= \dd\circ\hat{\mu}_2(Z_i, Y).
\end{align*}
Thus,
\[
\cg{A}(\dd\circ\hat{\mu}_2(Z_i, Y)) = \cg{A}(\hat{\mu}_2(\dd(Z_i), Y)) \geq \cg{A}(\dd(Z_i)) + \cg{A}(Y).
\]
As $\cg{A}(\dd(Z_i))\rightarrow\infty$, we conclude that
\[
\cg{A}(\dd\circ\hat{\mu}_2(Z_i, Y)) \rightarrow\infty 
\]
as well.  Summing over all $i$,
\[
\hat{\mu}_2(X, Y) = \sum_i \dd(\hat{\mu}_2(Z_i, Y)) \in \overline{\im(\dd)}.
\]

\end{proof}

The map
\[
SC^*(H)\rightarrow\bigslant{\widehat{\ker(\dd})}{\overline{\im(\dd)}}
\]
clearly intertwines $\mu_2$ and $\hat{\mu}_2$.  Thus, the cohomology-level map
\[
\eta:SH^*(H)\longrightarrow\overline{SH^*}(H)
\]
is a ring isomorphism.

\subsubsection{Finishing the proof}

\begin{proofthm1}
Recall the injective map
\[
\overline{SH^*}(H) \xrightarrow{\phi}\widehat{SH^*}(H)\simeq\lim_{\substack{\leftarrow \\ a}}SH^*_a(H)
\]
and the commutative diagram
\begin{equation} 
\begin{tikzcd}
& \overline{SH^*}(H) \arrow{d}{\phi} \\
SH^*(H) \arrow{ur}{\eta} \arrow{r}{\pi} & \widehat{SH^*}(H)
\end{tikzcd}
\end{equation}
Given Corollary \ref{cor:imisquotient}, it suffices to show that 
\[
\im(\pi) = \im(\phi).
\]
Clearly
\[
\im(\pi)\subset\im(\phi).
\]
We will show the reverse inclusion.  Fix $[X]\in \overline{SH^*}(H)$ and choose any cochain representative $X$ of $[X]$.  By definition, there are cochains $X_1, X_2, X_3, ...$ in $SC^*(H)$ with $\cg{A}(X_i)\rightarrow\infty$, such that
\[
X = \sum_{i=1}^{\infty}X_i.
\]
Let $Y_i\in HF^*(H_0)$ be the descent of $\cg{S}^{n}\circ\mathfrak{c}^{HG}(X_i)$ to cohomology.  Then
\[
[X_i] = \eta\circ\iota(Y_i).
\]
Write $Y_i$ in the fixed eigenspace decomposition:
\[
Y_i = Y_i^{\leq}+ Y_i^{>},
\]
where
\[
Y_i^{\leq}\in \bigoplus_{ev(\lambda) \leq k\pi R^2}HF^*(H_0)\hspace{1cm}\text{and}\hspace{1cm}Y_i^{>}\in \bigoplus_{ev(\lambda) >k\pi R^2}HF^*(H_0)
\]
Choose $Z_i\in \bigoplus\limits_{ev(\lambda) \leq k\pi R^2}HF^*(H_0)$ such that
\[
(\boldsymbol{c_{-1}^G\circ\cg{S}})^n(Z_i) = Y_i^{\leq}.
\]
By Corollary \ref{cor:almostdone!}, there exists a constant $C$ such that
\[
\cg{A}(X_i) - \boldsymbol{\cg{A}}(Z_i) < C,
\]
or
\begin{equation} 
\label{eq:boundzi}
\boldsymbol{\cg{A}}(Z_i) > \cg{A}(X_i) - C.
\end{equation}
Let
\[
Z = \sum_{i=1}^{\infty}Z_i.
\]
Equation (\ref{eq:boundzi}), together with the assumption that $\cg{A}(X_i)\rightarrow\infty$, shows that $Z$ is a well-defined element of $HF^*(H_0)$.

Recall from Lemma \ref{lem:kercontains} that
\[
\Phi(Y^{>}) = 0.
\]
Thus,
\[
\phi(X_i) = \phi\circ\eta\circ\iota(Y_i) = \Phi(Y_i) = \Phi(Y_i^{\leq}) + \Phi(Y^{>}) = \Phi(Y_i^{\leq}) = \Phi(Z_i).
\]
By linearity of $\Phi$ and $\eta$,
\[
\Phi(Z) = \phi(X).
\]
Thus,
\[
\im(\phi) = \im(\Phi) = \im(\pi).
\]
As $\phi$ is injective and $\pi = \phi\circ\eta$, we conclude that $\eta$ is surjective.  It follows that there is a ring isomorphism
\[
\overline{SH^*}(H)\cong\bigslant{SH^*(H)}{\ker(\eta)} = \bigslant{SH^*(H)}{\ker(\pi)}.
\]
By Corollary \ref{cor:imisquotient},
\[
\bigslant{SH^*(H)}{\ker(\pi)}\simeq\bigslant{HF^*(H_0)}{\bigoplus\limits_{ev(\lambda)> k\pi R^2}HF^*_{\lambda}(H_0)}.
\]
Finally, by Corollary \ref{cor:specsmatch} and the PSS identification $HF^*(H_0)\simeq QH^*(E)$,
\[
\bigslant{HF^*(H_0)}{\bigoplus\limits_{ev(\lambda)> k\pi R^2}HF^*_{\lambda}(H_0)}\simeq\bigslant{QH^*(E)}{\bigoplus\limits_{ev(\lambda)> k\pi R^2}QH^*_{\lambda}(E)}.
\]

\end{proofthm1}

\section{Toric line bundles}
\label{sec:toric}
We assume in this section that $M$ is a toric symplectic manifold.  In particular, we will use the fact that $M$ has a perfect Morse function $f:M\rightarrow\RR$, and the critical points of $f$ all have even index.  We denote the Morse index of a critical point $z$ of $f$ by $\mu_f(z)$.  Let $\{H_n:E\rightarrow\RR\}$ be a family of functions defined as in Subsection \ref{subsec:ham}, using the perfect Morse function $f$.  Let $J$ be an almost-complex structure of the form found in Subsection \ref{subsec:acs}.
\begin{proposition}
\label{prop:comequalsred}
There is an isomorphism
\[
\widehat{SH^*}(H)\simeq \overline{SH^*}(H).
\]
\end{proposition}

\begin{remark}
In the author's thesis \cite{venkatesh}, we computed $\widehat{SH^*}(H)$ on monotone toric line bundles.  Proposition \ref{prop:comequalsred} shows that the results of this paper encompass those of \cite{venkatesh}.
\end{remark}

To prepare for the proof of Proposition \ref{prop:comequalsred} we recall results in \cite{albers-k} about the structure of each Floer complex $CF^*(H_n)$.  It is useful to first consider a $\ZZ$-graded Floer theory.  Let $S$ be a formal variable of degree $\tau$, where $\tau$ is the minimal Chern number of $E$.  The following results were proved by Albers-Kang for monotone and negative-monotone symplectic line bundles, but the proofs carry over verbatim to the more general case.  The following Lemma appears throughout the literature in various guises.  See, for example, \cite{albers-k} or \cite{nelson}.

\begin{lemma}
\label{lem:zzgrading}
The chain complex $SC^*(H; \Lambda[S])$ is $\ZZ$ graded.  A generator $x\in\cg{P}(H_n)$ with capping $\tilde{x}$ has even degree, respectively odd degree, if $x$ is the minimum, respectively maximum, of a perfect Morse function on a transversally non-degenerate family of orbits.  If $\tilde{x}$ is chosen to be the fiber disk, then
\[
|x| = -2\mathfrak{w}(x) + \mu_f(\rho\circ x) + \{-1, 0\},
\]
where the final term is $-1$ if $x$ is a minimum and $0$ if $x$ is a maximum.
\end{lemma}

\begin{remark}
The grading in Lemma \ref{lem:zzgrading} is the cohomological Conley-Zehnder, shifted up by $m$.  We shift the grading for ease of notation; the grading on $CF^*(H_0)$ now coincides with the grading on singular cochains.
\end{remark}

\begin{remark}
Lemma \ref{lem:zzgrading} separates symplectic cohomology into an odd-graded component, denoted by $SC^{odd}(H)$, which consists of all ``minimums'', and and even-graded component, denoted by $SC^{even}(H)$, which consists of all ``maximums'', plus the constant orbits.
\end{remark}

\begin{lemma}[Albers-Kang \cite{albers-k}]
\label{lem:proj}
An element $cS^{\alpha}T^{\beta}x$, with $c\in\CC$ and $x\in\cg{P}(H_n)$ descends to an element
\[
\rho_*(cS^{\alpha}T^{\beta}x) := cS^{\alpha+\frac{k\beta}{\tau}}T^{\beta}x
\]
of the Floer cochain complex associated with the Floer data $\{f, j\}$.  This cochain complex is $\RR$-graded.
The projected generator has degree
\[
|\rho_*(cS^{\alpha}T^{\beta} x)| = 2\tau\alpha + 2 k\beta + \mu_f(\rho\circ x).
\]
\end{lemma}

Let $p\in\RR$ and define
\begin{equation*}
\cg{F}_p = \left\la \left(\sum_{j=0}^{J}\left(\sum_{i=0}^{\infty}c_iT^{\beta_i}\right)S^{\alpha_j}\right)x \hspace{.1cm}\bigg|\hspace{.1cm} x{\bf q}^{\ell}\in\cg{P}(H), \ell\in\{0, 1\}, |\rho_*(S^{\alpha_i}T^{\beta_j}x)| \geq p\right\ra
\end{equation*}

\begin{lemma}[Albers-Kang \cite{albers-k}]
\label{lem:projsol}
A solution $u(s, t)$ of the Floer equation associated with the Floer data $\{H_n, J\}$ on $E$ projects to a solution $\rho_* u$ of the Floer equation associated with the Floer data $\{f, j\}$ on $M$.
\end{lemma}

\begin{corollary}
The complexes $\cg{F}_p$ define a filtration on $SC^*(H)$.
\end{corollary}
Note that this filtration induces a filtration on the subcomplex $SC^*_{(a, \infty)}(H)$ and the quotient complex $SC^*_a(H)$.  The following two Lemmas, stated for the full complex $SC^*(H)$, also apply to these smaller complexes.

\begin{lemma}[Albers-Kang \cite{albers-k}] 
\label{lem:ddecompose}
The differential $\dd_{\Lambda[S]}$ on $SC^*(H; \Lambda[S])$ decomposes as
\[
\dd_{\Lambda[S]} = \dd_0 + \dd_1 + \dd_2 + ...
\]
where $\dd_i$ increases the grading of the projected generator by $i$.
\end{lemma}

\begin{lemma}[Albers-Kang \cite{albers-k}]
\label{lem:dd0}
Floer solutions contributing to the zeroth part of the differential, $\dd_0$, remain in a single fiber.  If $x_M$ is the maximum and $x_m$ is the minimum of a perfect Morse function on an $S^1$-family of orbits, then $\dd_0(x_m) = x_M$ and $\dd_0(x_M) = 0$.
\end{lemma}

Consider Lemma \ref{lem:zzgrading}.  If $x\in\cg{P}(H_n)$ and 
\[
\left(\sum_{j=0}^{J}\left(\sum_{i=0}^{\infty}c_iT^{\beta_i}\right)S^{\alpha_j}\right)x  \in SC^{\ell}(H_n)
\]
lies in a fixed degree ${\ell}$, then
\[
\alpha_j = \frac{1}{2}\left({\ell} + 2\mathfrak{w}(x) - \mu_f(\rho\circ x) + \{-1, 0\}\right)
\]
where $\mathfrak{w}(x)$ is the winding number of $x$ in the fiber and the contribution in $\{-1, 0\}$ depends only on $x$.  So $\alpha_j$ is a fixed constant independent of $j$, which we denote by $\alpha_x^k$.
\begin{lemma}
\label{lem:morselim}
Let $\sum_{i=0}^{\infty}c_iS^{\alpha_{x_i}^{\ell}}T^{\beta_i}x_i{\bf q}^{n_i}$ be an element in the cochain complex $\lim\limits_{\substack{\leftarrow \\ a}}SC^{\ell}_a(H)$ or $SC^*(H)$, which consists of formal sums whose action limits to infinity (in the latter case, the sum includes only finitely many distinct ``$x_i$'').  Assume $c_i\neq 0$ for any $i$.  Then
\[
\lim_{i\rightarrow\infty} |\rho_*(c_iS^{\alpha_{x_i}^{\ell}}T^{\beta_i}x_i)| = \infty.
\]
\end{lemma}
\begin{proof}
By the definition of $\lim\limits_{\substack{\leftarrow \\ a}}SC^*_a(H)$ and $SC^*(H)$,
\[
\lim_{i\rightarrow\infty}\cg{A}(T^{\beta_i}x_i) = \infty.
\]
Because each Hamiltonian $H_n$ is bounded, this implies that
\begin{equation}
\label{eq:actionlim}
\lim_{i\rightarrow\infty} \beta_i - \int_{D^2} (\widetilde{x_i})^*\Omega = \infty.
\end{equation}
By construction, the winding number $\mathfrak{w}(x_i)$ satisfies
\begin{equation}
\label{eq:wnlim}
\mathfrak{w}(x_i)\geq 0.
\end{equation}
Thus,
\begin{equation}
\label{eq:disklim}
\int_{D^2} -\widetilde{x_i}^*\Omega \simeq -\pi R^2\mathfrak{w}(x_i) \leq 0.
\end{equation}
Equations \ref{eq:actionlim} and \ref{eq:disklim} imply that
\[
\lim_{i\rightarrow\infty}\beta_i = \infty,
\]
and so 
\begin{equation} 
\label{eq:b_i}
\lim_{i\rightarrow\infty}2 k\beta_i = \infty,
\end{equation}
where $k > 0$ is the constant satisfying $c_1^E = -k[\omega]$.
The choice of grading implies
\[
\ell = -2\mathfrak{w}(x_i) + 2\tau\alpha_{x_i}^{\ell} + C,
\]
where $C$ is a bounded constant in the interval $[-1, \dim_{\RR}(M)]$.  Equation \ref{eq:wnlim} now implies that
\begin{equation} 
\label{eq:a_i}
2\tau\alpha_{x_i}^{\ell}\geq \ell - C.
\end{equation}
Finally, by Lemma \ref{lem:proj},
\[
|c_iS^{\alpha_{x_i}^{\ell}}T^{\beta_i}\rho\circ x_i| = 2\tau\alpha_{x_i}^{\ell} + 2k\beta_i +\mu_f(\rho\circ x_i).
\]
This grows like $2\tau\alpha_{x_i}^{\ell} + 2 k\beta_i$, which, by observations (\ref{eq:a_i}) and (\ref{eq:b_i}), limits to infinity.

\end{proof}

Let $\dd_{\Lambda[S]}$ be the differential on $SC^*(H; \Lambda[S])$, and let $\dd_{\Lambda}$ be the differential on $SC^*(H; \Lambda)$.  Let 
\begin{equation}
\label{eq:qtau}
q_{\tau}:SC^*(H; \Lambda[S])\longrightarrow SC^*(H; \Lambda)
\end{equation}
be the covering map that sends $S$ to $1$.  Let ${\bf q_{\tau}}$ be the induced map on homology.

Fix a degree $\ell\in\ZZ$ and consider the $\Lambda$-linear map $p_{\tau}^{\ell}$ given on generating orbits by
\begin{align*}
p_{\tau}^{\ell}:SC^{\ell\mod 2\tau}(H; \Lambda)&\longrightarrow SC^{\ell}(H; \Lambda[S]) \\
x &\mapsto S^{\alpha_x}x,
\end{align*}
where $\alpha_x$ is uniquely determined by the degree formula
\[
2\alpha_x + |x| = \ell.
\]
On the level of vector spaces, $p_{\tau}^{\ell}$ is inverse to $q_{\tau}\big|_{SC^{\ell}(H; \Lambda[S])}$.  On the level of cochains,
\[
p_{\tau}^{\ell+1}\circ\dd_{\Lambda} = \dd_{\Lambda[S]}\circ p_{\tau}^{\ell},
\]
and $p_{\tau}^{\ell}$ descends to a map on $SH^{\ell\mod2\tau}(H; \Lambda)$, inverse to 
\[
\boldsymbol{q_{\tau}}\big|_{SH^{\ell}(H; \Lambda[S])}: SH^{\ell}(H; \Lambda[S])\longrightarrow SH^{\ell\mod2\tau}(H; \Lambda).
\]
Thus, $\boldsymbol{q_{\tau}}\big|_{SH^{\ell}(H; \Lambda[S])}$ is an isomorphism, and so the full map $\boldsymbol{q_{\tau}}$ is surjective on cohomology.

$q_{\tau}$ respects the action filtration, inducing a surjective map
\begin{equation}
\label{eq:actionsurj}
SH^*_a(H; \Lambda[S])\longrightarrow SH^*_a(H; \Lambda).
\end{equation}
This map restricts to an isomorphism
\[
SH^{\ell}_a(H; \Lambda[S])\longrightarrow SH^{\ell\text{ mod } 2\tau}_a(H; \Lambda),
\]
which induces an isomorphism
\[
\widehat{SH^{\ell}}(H; \Lambda[S]) \longrightarrow \widehat{SH^{\ell\text{ mod } 2\tau}}(H; \Lambda).
\]
But inverse limits commute with finite direct sums, and so
\[
\widehat{SH^*}(H; \Lambda)\cong \bigoplus\limits_{\ell=1}^{2\tau}\widehat{SH^{\ell\text{ mod } 2\tau}}(H; \Lambda).
\]
We conclude that the induced map
\begin{equation}
\label{eq:surjcoeff}
\widehat{SH^*}(H; \Lambda[S]) \longrightarrow\widehat{SH^*}(H; \Lambda)
\end{equation}
is surjective as well.

\begin{lemma}
\label{lem:pieven}
If $x\in\cg{P}(H_n)$ is an even-graded periodic orbit then $\dd_{\Lambda}(x) = 0$.
\end{lemma}

\begin{proof}
First take coefficients in $\Lambda[S]$.  By index conventions, $x$ is a maximum of a perfect Morse function on an $S^1$-family of orbits.  Suppose for contradiction that $\dd_{\Lambda[S]}(x) = Y\neq 0$.  By Lemma \ref{lem:morselim}, we can write $Y = \sum_{i=p_0}^{\infty} Y_i$, where $|\rho_*Y_i| = i$ and $Y_{p_0}\neq 0$.  By definition, $\dd_{\Lambda[S]}^2(x) = 0$, implying that $\dd_{\Lambda[S]}(Y) = 0$.  Decompose $\dd_{\Lambda[S]}(Y)$ with respect to the filtration: 
\[
\dd_{\Lambda[S]}(Y) = \sum_{i = p_0}^{\infty}\dd_{\Lambda[S]}(Y)_i,
\]
where $|\rho_*\dd_{\Lambda[S]}(Y)_i| = i$.  Then $\dd_{\Lambda[S]}(Y) = 0$ if and only if $\dd_{\Lambda[S]}(Y)_i= 0$ for every $i$.  In particular,
\[
\dd_0(Y_{p_0}) = \dd_{\Lambda[S]}(Y)_{p_0} = 0.
\]
However, by index considerations, each generator appearing in $Y_{p_0}$ is a minimum of a perfect Morse function on an $S^1$-family of orbits.  By Lemma \ref{lem:dd0},
\[
\dd_0(Y_{p_0}) \neq 0.
\]
A contradiction is reached.

Now take coefficients in $\Lambda$.  Recall the map $q_{\tau}$ from Equation (\ref{eq:qtau}), which takes $x(t)$ to itself.  As $q_{\tau}$ is a chain map,
\[
\dd_{\Lambda}(x) = \dd_{\Lambda}\circ q_{\tau}(x) = q_{\tau}\circ \dd_{\Lambda[S]}(x) = 0.
\]

\end{proof}

\begin{corollary}
\label{cor:surjconnecting}
The connecting maps
\[
\pi_{a, a'}:SH^{even}_a(H)\rightarrow SH^{even}_{a'}(H)
\]
are surjective.
\end{corollary}

\begin{lemma}
\label{lem:piodd}
The kernel of $\dd$ is $0$ on the odd-component:
\[
\ker(\dd)\big|_{SC^{odd}(H)} = \ker(\widehat{\dd})\big|_{\widehat{SC^{odd}}(H; \Lambda)} = 0.
\]
As a consequence,
\[
SH^{odd}(H) = \widehat{SH^{odd}}(H) = 0.
\]
\end{lemma}

\begin{proof}
First consider the uncompleted theory $SH^{odd}(H; \Lambda)$.  The choices made ensure that 
\[
CF^{odd}(H_0; \Lambda) = 0,
\]
from which it is clear that $HF^{odd}(H_0; \Lambda) = 0$.  Corollary \ref{cor:ritter1} expresses an isomorphism
\[
SH^*(H; \Lambda)\cong\bigslant{HF^*(H_0; \Lambda)}{HF^*_0(H_0; \Lambda)}.
\]
This isomorphism preserves the parity of the grading, and it follows that
\[
SH^{odd}(H; \Lambda) = \bigslant{HF^{odd}(H_0; \Lambda)}{HF^{odd}_0(H_0; \Lambda)} = 0.
\]
It remains to show that
\[
\widehat{SH^{odd}}(H; \Lambda) = 0.
\]

First consider coefficients in $\Lambda[S]$.  Choose a non-zero sequence $\{X_a \in SC^k_a(H; \Lambda[S])\}$, and assume for contradiction that each $X_a$ is a cocycle and the connecting morphisms $\pi_{a, a'}$ send $[X_a]$ to $[X_{a'}]$, so that ${\pi_{a, a'}}([X_a]) = [X_{a'}]$.  By definition, there exists a coboundary $\dd(Z)$ such that, at the chain level, 
\[
\pi_{a, a'}(X_a) = X_{a'} + \dd(Z).
\]
The proof of Lemma \ref{lem:pieven} implies that $\dd(Z) = 0$, so 
\begin{equation}
\label{eq:piinclusive}
\pi_{a, a'}(X_a) = X_{a'}.
\end{equation}

The action of the summands of the nested sequence $\{X_a\}$ approaches infinity.  By Lemma \ref{lem:morselim}, there exists some $p_0$ such that $X_a\in\cg{F}_{p_0}$ for all $a$.  Assume we have chosen the maximum such $p_0$, so that, for large enough $a$, some non-zero summand $X_{p_0, a}$ of $X_a$ achieves $|\rho_*X_{p_0, a}| = p_0$.  Without loss of generality, assume that we have fixed a large enough $a$.

Let $Y_a = \dd_{\Lambda[S]}(X_a)$.  Write $X_a = \sum_{p=p_0}^{\infty} X_{p, a}$, where $|\rho_*X_{p, a}| = p$, and similarly write $Y_a = \sum_{i=p_0}^{\infty} Y_{p, a}$.  We want to show that $Y_a \neq 0$.  It suffices to show that there exists a single non-zero summand $Y_{p, a}$.

The decomposition of the differential described in Lemma \ref{lem:ddecompose} yields
\[
Y_{p, a} = \dd_0(X_{p, a}) + \dd_1(X_{p-1, a}) + ... + \dd_{p - p_0}(X_{p_0, a}).
\]
In particular,
\[
(\dd_0)_{\Lambda[S]}(X_{p_0, a}) = Y_{p_0, a}.
\]
Assume $a > \cg{A}((\dd_0)_{\Lambda[S]}(X_{p_0, a}))$. 

Due to the  grading conventions, $X_{p_0, a}$ is the minimum of a perfect Morse function on the underlying trajectory.  Therefore, by Lemma \ref{lem:dd0}, $\dd_0(X_{p_0}, a)\neq 0$.  It follows that
\[
Y_{p_0, a} \neq 0,
\]
and so
\[
Y_a = \dd_{\Lambda[S]}(X_a) \neq 0.
\]
Thus, $\dd_{\Lambda[S]}(X_a)\neq 0$.  As this holds for all sufficiently large $a$, we conclude that $\ker(\hat{\dd}) = 0$ and 
\[
\widehat{SH^{odd}}(H; \Lambda[S]) = 0.
\]
Combining the surjectivity of the map (\ref{eq:surjcoeff}) with Lemma \ref{lem:pieven}, it follows that
\[
\ker(\hat{\dd}\big|_{\widehat{SC^{odd}}(H; \Lambda)}) = 0
\]
as well.

\end{proof}
\\
We end this subsection by proving Proposition \ref{prop:comequalsred}.

\begin{proofprop4}
We have seen that there is an injective map
\[
\phi:\overline{SH^*}(H)\rightarrow\widehat{SH^*}(H)
\]
induced by the inclusion
\[
\widehat{\ker(\dd)}\hookrightarrow \widehat{SC^*}(H).
\]
Lemma \ref{lem:piodd} shows that $\widehat{SH^{odd}}(H) = 0$, which implies that
\[
\overline{SH^{odd}}(H) = 0
\]
as well.  By Lemma \ref{lem:pieven},
\[
\ker(\dd) = SC^{even}(H),
\]
and so
\[
\widehat{\ker(\dd}) = \widehat{SC^{even}}(H).
\]
It follows that the map $\phi$ is also surjective.

\end{proofprop4}

\subsection{Annulus subbundles}
Recall that we defined the completed symplectic cohomology theory for trivial cobordisms in Subsection \ref{subsec:rfh}.  In this section, we consider trivial cobordisms $A_{R', R}$ that lie between disk bundles of radii $R'$ and $R$.  We would like to know for which radii completed symplectic cohomology is non-vanishing.

Let $QH_*(E)$ be quantum homology, the dual of quantum cohomology.  The dual of the map $\rho^*c_1^E\cup_*-$ is the quantum intersection product with the Poincar\'e dual of $\rho^*c_1^E$, denoted by $PD(\rho^*c_1^E)\cap_* -$.  These maps have the same eigenvalues, and we denote by $QH_*^{\lambda}(E)$ the $\lambda$-invariant subspace of $QH_*(E)$ under $PD(\rho^*c_1^E)\cap_*-$.

\begin{theorem}
\label{thm:rfh}
If $R'\leq R$, there is a vector-space isomorphism
\[
\widehat{SH^*}(A_{R', R})\cong \bigslant{QH^*(E)}{\bigoplus\limits_{\substack{k\pi(R')^2 \geq ev(\lambda) \\ \text{or} \\ k\pi R^2 < ev(\lambda)} } QH^*_{\lambda}(E)}.
\]
If $R'\geq R$, there is a vector-space isomorphism
\[
\widehat{SH^*}(A_{R', R}) \cong \bigoplus_{k\pi R^2 < ev(\lambda) \leq k\pi(R')^2}QH_{*-1}^{\lambda}(E).
\]
\end{theorem}

\begin{remark}
In particular, 
\[
\widehat{SH^*}(A_{R', R})\neq 0
\]
precisely when the size of some eigenvalue lies between $R'$ and $R$.
\end{remark}

By construction, there is a long exact sequence
\begin{equation} 
\label{eq:les}
\begin{tikzcd}
H\left(\widehat{SC_*}(H')\right) \arrow{rr}{\widehat{\boldsymbol{\mathfrak{c}}}} && H\left(\widehat{SC^*}(H)\right) \arrow{dl} \\
& \widehat{SH^*}(A_{R', R}) \arrow{ul}{[1]} &
\end{tikzcd}
\end{equation}
To compute $\widehat{SH^*}(A_{R', R})$ it therefore suffices to compute $H\left(\widehat{SC_*}(H')\right)$, $H\left(\widehat{SC^*}(H)\right)$, and the connecting map $\widehat{\mathfrak{c}}$.  As noted in Remark \ref{rem:comnotred}, these new homology theories are not {\it a priori} the completed homology theories introduced thus far.  However, in this simplified scenario, the definitions coincide.
\begin{lemma}
\label{lem:whentotakeH}
There are natural isomorphisms
\[
H(\widehat{SC^*}(H))\cong \widehat{SH^*}(H)
\]
and
\[
H(\widehat{SC_*}(H))\cong \widehat{SH_*}(H).
\]
\end{lemma}

\begin{proof}
Applying the proof of Lemma \ref{lem:piodd} to a prospective cocycle in $\widehat{SC^{odd}}(H)$ yields
\[
H^{odd}\left(\widehat{SC^*}(H)\right) = 0.
\]
There is a Milnor exact sequence
\[
0\longrightarrow \lim\limits_{\substack{\leftarrow \\ a}}{}^1 SH^{even}_a(H) \longrightarrow H^{even}\left(\widehat{SC^*}(H)\right)\longrightarrow\widehat{SH^{even}}(H)\longrightarrow 0.
\]
Corollary \ref{cor:surjconnecting} says that the connecting maps
\[
{\pi_{a, a'}}: SH^{even}_a(H)\longrightarrow SH^{even}_{a'}(H), \hspace{1cm} a > a',
\]
are surjective.  Thus, the Mittag-Leffler condition is satisfied and
\[
\lim\limits_{\substack{\leftarrow \\ a}}{}^1 SH^{even}_a(H) = 0.
\]
The first isomorphism follows.

The second isomorphism follows from the exactness of direct limits.

\end{proof}

We begin by computing $SH_*(H)$ and $\widehat{SH_*}(H)$.  The following Lemma is standard fare, but sets the stage for subsequent elaboration.
\begin{lemma}
Up to grading, $SH_*(H)$ is the dual homology theory to $SH^*(H)$.
\end{lemma}

\begin{proof}
A standard Morse theory argument shows that the Floer chain complex $(CF_*(H_i), \dd_{fl})$ is naturally isomorphic to the dual of the Floer cochain complex $(CF^*(H_i), \dd^{fl})$, where $\dd_{fl}$ counts rigid Floer cylinders with negative input and positive output.  By definition, a generator $x$ of $CF^*(H_0)$ corresponds to a generator $\check{x}$ of $CF_*(H_0)$.  $\check{x}$ is the morphism $\mathbf{1}_x$ which on generators $y$ evaluates to
\begin{equation}
\label{eq:explicitdual}
\mathbf{1}_x(y) = \left\{\begin{array}{cc} 1 & x = y \\ 0 & \text{else}\end{array}\right.
\end{equation}
As direct products are dual to direct sums,
\[
Hom(\bigoplus_i CF^*(H_i)[{\bf q}], \Lambda)\cong \prod_i Hom(CF^*(H_i), \Lambda)[{\bf q}] \cong \prod_i CF_*(H_i)[{\bf q}],
\]
with differential
\[
\dd(X + Y{\bf q}) = \dd_{fl}(X) + c_*(Y) + \dd_{fl}(Y).
\]
Here, $c_*$ is defined similarly to $c$, but, as with $\dd_{fl}$, it counts the rigid Floer cylinders with negative input and positive output.

Poincar\'e duality in Floer theory yields natural isomorphisms
\[
(CF_*(H_i), \dd_{fl})\cong (CF^{2m-*}(-H_i), \dd^{fl}).
\]
See \cite{c-f-o} for an elaboration on the grading, recalling that we have shifted their grading scheme by $m$.  These isomorphisms intertwine the continuation maps $c_*$ with the original continuation maps $c$.  Thus, there is a natural chain isomorphism
\[
\prod_i CF_*(H_i)[{\bf q}] \cong \prod_i CF^{2m-*}(-H_i)[{\bf q}].
\] 

\end{proof}

Just as $\boldsymbol{c\circ\cg{S}}$ acts on $HF^*(H_0)$, the composition $\boldsymbol{\cg{S}\circ c}$ acts on $HF^*(-H_0)$.  The bijection
\[
\boldsymbol{\cg{S}}:HF^*(-H_0)\longrightarrow HF^*(H_0)
\]
intertwines $\boldsymbol{c\circ\cg{S}}$ and $\boldsymbol{\cg{S}\circ c}$, and it is easy to see that these two compositions have the same eigenvalues.  We denote by $HF^*_{\lambda}(-H_0)$ the $\lambda$-generalized eigenspace of the action of $\boldsymbol{\cg{S}\circ c}$ on $HF^*_{\lambda}(-H_0)$. 

\begin{lemma}
\label{lem:shisdual}
The isomorphism 
\[
SH^*(H)\cong\bigslant{HF^*(H)}{HF^*_0(H)}
\]
induces an isomorphism
\[
SH_*(H)\cong \bigoplus_{ev(\lambda) < \infty}HF^*_{\lambda}(-H_0)
\]
such that the map $\Psi: SH_*(H)\longrightarrow HF^*(-H_0)$ is the canonical inclusion.
\end{lemma}

\begin{proof}
The dual of the inclusion
\[
CF^*(H_0)\hookrightarrow SC^*(H)
\]
is the projection
\[
\prod_i CF_*(H_i)[{\bf q}] \twoheadrightarrow CF_*(H_0).
\]
Under Poincar\'e duality, this projection is a map 
\[
SC_{2m-*}(H)\longrightarrow CF^{2m-*}(-H_0),
\]
which descends to the map
\begin{equation}
\label{eq:projectionhomology}
\Psi: SH_{2m-*}(H)\longrightarrow HF^{2m-*}(-H_0).
\end{equation}
As $\Lambda$ is torsion free, $SH^*(H; \Lambda)$ is an injective module, and the Universal Coefficient Theorem yields a natural isomorphism
\[
SH_*(H)\cong Hom(SH^*(H), \Lambda).
\]
In particular, denoting by $\iota^{\vee}$ the dual of the projection
\[
\iota:HF^*(H_0)\longrightarrow SH^*(H),
\]
the following diagram commutes
\[
\begin{tikzcd}
SH_{2m-*}(H) \arrow{r}{\simeq} \arrow{d}{\Psi} & Hom(SH^*(H), \Lambda) \arrow{d}{\iota^{\vee}} \\
HF^{2m-*}(-H_0) \arrow{r}{\simeq} & Hom(HF^*(H_0), \Lambda).
\end{tikzcd}
\]
Identifying $SH^*(H)$ with the quotient complex
\[
\bigslant{HF^*(H_0)}{HF^*_0(H_0)} \cong SH^*(H),
\]
the group $SH_{2m-*}(H)$ is precisely the subset of $HF^{2m-*}(-H_0)$ on which $HF^*_0(H_0)$ vanishes, and $\Psi$ is the canonical inclusion.

Let $\cg{B}^{\vee} = \left\{w_1^{\lambda_1}, w_2^{\lambda_1}, ..., w_{k_{\lambda}}^{\lambda_k}\right\}$ be the basis of $HF^{2m-*}(-H_0)$ dual to the fixed Jordan basis $\cg{B}$.  By definition, $\cg{B}^{\vee}$ is a generalized eigenbasis for the action on $Hom(HF^*(H_0), \Lambda)$ dual to $\boldsymbol{c\circ\cg{S}}$.  A standard result in Floer theory shows that, under the identification $HF_*(H_0)\cong Hom(HF^*(H_0), \Lambda)$, the dual action is the map
\[
\boldsymbol{\cg{S}}^{-1}\circ \boldsymbol{c_*}:HF_*(H_0)\longrightarrow HF_*(H_0).
\]
Under the Poincar\'e duality isomorphism $HF_*(H_0)\cong HF^{2m-*}(-H_0)$, the continuation maps $\boldsymbol{c}$ and $\boldsymbol{c_*}$ are identified, and $\boldsymbol{\cg{S}}^{-1}\circ \boldsymbol{c_*}$ is identified with
\[
\boldsymbol{\cg{S}}\circ \boldsymbol{c}:HF^{2m-*}(-H_0)\longrightarrow HF^{2m-*}(-H_0).
\]
Thus, the dual basis $\cg{B}^{\vee}$ is a generalized eigenbasis for the action of $\boldsymbol{\cg{S}}\circ \boldsymbol{c}$ on $HF^{2m-*}(-H_0)$.  By definition, $HF^*_0(H_0)\subset\ker(w_i^{\lambda_j})$ precisely when $\lambda_j\neq 0$.  Thus,
\[
SH_{2m-*}(H_0)\cong \bigoplus_{\lambda\neq 0}HF^{2m-*}_{\lambda}(-H_0) = \bigoplus_{ev(\lambda) < \infty}HF^{2m-*}_{\lambda}(-H_0). 
\]

\end{proof}

\begin{lemma}
\label{lem:dualcomplexes}
Up to grading, $\widehat{SH_*}(H)$ is the dual homology theory to $\widehat{SH^*}(H)$.
\end{lemma}

\begin{proof}
Equip the uncompleted chain complex $SC^*(H)$ with the non-
Archimedean metric given by
\[
||X|| = e^{-\cg{A}(X)}.
\]
Completing $SC^*(H)$ with respect to $||\cdot||$ yields a complex whose elements are formal sums
\begin{equation} 
\label{eq:metricomp}
\left\{ \sum_{i=0}^{\infty}\kappa_ix_i + d_iy_i{\bf q} \hspace{.1cm}\bigg|\hspace{.1cm} x_i, y_i\in\cg{P}(H), \cg{A}(\kappa_ix_i) \rightarrow\infty, \cg{A}(d_iy_i)\rightarrow\infty\right\}.
\end{equation}
Because the connecting maps in the inverse limit $\lim\limits_{\substack{\leftarrow \\ a}}$ are surjections, $\widehat{SC^*(H)}$ is isomorphic to the complex (\ref{eq:metricomp}).
The dual of $\widehat{SC^*}(H)$ is bounded homomorphisms on $SC^*(H)$, that is,
\[
Hom_{\Lambda}(\widehat{SC^*}(H), \Lambda) = Hom^b_{\Lambda}(SC^*(H), \Lambda) := \left\{ \gamma\in Hom(SC^*(H), \Lambda)\hspace{.2cm}\bigg|\hspace{.2cm} \sup_{X\in SC^*(H)} \frac{||\gamma(X)||}{||X||}< \infty\right\},
\]
equipped with the usual differential
\[
\dd^*(\gamma) = \gamma\circ \dd.
\]
We want to show that $\widehat{SC_*}(H)$ is chain isomorphic to $Hom^b_{\Lambda}(SC^*(H), \Lambda)$.

Analogously to the statement (\ref{eq:metricomp}) for cohomology, $\widehat{SC_*}(H)$ is isomorphic to
\begin{equation} 
\label{eq:metricomphom}
\left\{ \sum_{i=0}^{\infty}\kappa_i\check{x_i} + d_i\check{y_i}{\bf q} \hspace{.1cm}\bigg|\hspace{.1cm} \check{x_i}, \check{y_i}\in\cg{P}(-H), \cg{A}(\kappa_i\check{x_i}), \cg{A}(d_i\check{y_i})\geq a\text{ for some fixed } a\in\RR\text{ independent of } i\right\}.
\end{equation}

Recall from the proof of Lemma \ref{lem:shisdual} that Poincar\'e duality for Hamiltonian Floer theory is a chain isomorphism
\begin{equation} 
\label{eq:poincare}
CF^*(-H_i) \cong CF_{2m-*}(H_i)
\end{equation}
induced by the set bijection $\cg{P}(-H_i) \cong \cg{P}(H_i)$ via $x(t) \mapsto x(-t)$.  For each $x\in\cg{P}(H_i)$, let $\check{x}$ be the corresponding element of $\cg{P}(-H_i)$ under this bijection.
Consider the map
\begin{equation} 
\label{eq:boundediso}
\widehat{SC_*}(H) \longrightarrow Hom^b_{\Lambda}(SC^{2m-*}(H), \Lambda)
\end{equation}
given on generators by
\[
\check{x}_i \mapsto \mathbf{1}_{x_i},
\]
where 
\[
\mathbf{1}_{x_i}(x) = \left\{\begin{array}{cc} 1 & x = x_i \\ 0 & x\neq x_i\end{array}\right\}
\]
To check that this map is well-defined, enumerate $\cg{P}(H) = \{x_0, x_1, x_2, ...\}$, and write
\[
\check{X} = \sum_{i=0}^{\infty}\kappa_i\check{x_i} + d_i\check{x_i}{\bf q}\in\widehat{SC_*}(H)
\]
so that there exists $a\in\RR$ with
\[
\cg{A}(\kappa_i\check{x_i}), \cg{A}(d_i\check{x_i}) \geq a
\] for all $i$.
Then for all $x\in\cg{P}(H)$,
\begin{align*}
ev(\check{X}(x)) &= ev\left(\sum_{i=0}^{\infty} \kappa_i\mathbf{1}_{x_i}(x)\right).
\end{align*}
If $x\neq x_i$ for any $i$, then $\check{X}(x) = 0$ and 
\begin{align*}
    \frac{||\check{X}(x)||}{||x||} = 0.
\end{align*}
So suppose there exists $i$ such that $x = x_i$.  Then
\begin{align*}
ev(\check{X}(x)) &= ev(\kappa_i) \\
&\geq a - \cg{A}(\check{x_i}) \\
&= a + \cg{A}(x).
\end{align*}
where the last equality follows because the capping disk of  $x_i(t)$ is the capping disk of $x_i(-t) = \check{x}_i(t)$ with reversed orientation.  Thus,
\begin{align*}
    \frac{||\check{X}(x)||}{||x||} &\leq \frac{e^{-a - \cg{A}(x)}}{e^{-\cg{A}(x)}} \\
    &= e^{-a}.
\end{align*}
Similarly,
\[
\cg{A}(\check{X}(x{\bf q})) \leq e^{-a}.
\]
Extending linearly, we find that $\check{X}$ is indeed a bounded operator.

To see that the map (\ref{eq:boundediso}) is surjective, suppose $\gamma \in Hom^b_{\Lambda}(SC^*(H), \Lambda)$, so that
\[
e^{-ev(\gamma(X)) + \cg{A}(X)} := \frac{||\gamma(X)||}{||X||} \leq M.
\]
for some fixed $M > 0$.  Then 
\[
ev(\gamma(X)) \geq \cg{A}(X) -\log(M).
\]
Let
\[
\check{X} = \sum_{x\in\cg{P}(H)} \gamma(x)\check{x} + \gamma(x{\bf q})\check{x}{\bf q}
\]
Then
\begin{align*}
\cg{A}(\gamma(x)\check{x}) &\geq \cg{A}(x) - \log(M) + \cg{A}(\check{x}) 
= \cg{A}(x) - \log(M) - \cg{A}(x) 
= -\log(M).
\end{align*}
Similarly,
\[
\cg{A}(\gamma(x{\bf q})\check{x}) \geq -\log(M),
\]
so $\check{X} \in \widehat{SC_*}(M).$

The bijection (\ref{eq:boundediso}) intertwines the differential by the Poincar\'e duality isomorphism (\ref{eq:poincare}) and a standard Morse theory argument.

Finally, Lemma \ref{lem:whentotakeH} completes the isomorphism:
\[
\widehat{SH_*}(H) \cong H\left(\widehat{SC_*}(H)\right)\cong H\left(Hom^b_{\Lambda}(SC^{2m-*}(H), \Lambda)\right)\cong H\left(Hom_{\Lambda}(\widehat{SC^{2m-*}}(H), \Lambda)\right).
\]
\end{proof}
\\
Analogously to the computation of $SH_*(H)$, we can compute action-completed symplectic homology $\widehat{SH_*}(H)$ as a subspace of $HF^*(-H_0)$.
\begin{proposition}
\label{prop:homcom}
There are isomorphisms
\[
\widehat{SH_*}(H) \cong \bigoplus_{ev(\lambda) \leq k\pi R^2}HF^*_{\lambda}(-H_0).
\]
and
\[
SH_*(H)\cong \bigoplus_{ev(\lambda) <\infty} HF^*_{\lambda}(-H_0)
\]
under which the map
\[
\widehat{SH_*}(H)\longrightarrow HF^*(-H_0)
\]
is the canonical inclusion.
\end{proposition}

\begin{proof}
We showed the second isomorphism in Lemma \ref{lem:shisdual}.  The first isomorphism is proved in an entirely analogous manner.  For completion, we sketch this argument again.

By Lemma \ref{lem:dualcomplexes}, there is an isomorphism
\[
\widehat{SH_*}(H) = H\left(Hom_{\Lambda}(\widehat{SC^{2m-*}}(H), \Lambda)\right).
\]
The Universal Coefficient Theorem gives an isomorphism
\[
H\left(Hom_{\Lambda}(\widehat{SC^{2m-*}}(H), \Lambda)\right) \cong Hom_{\Lambda}\left(\widehat{SH^{2m-*}}(H), \Lambda\right).
\]
The map
\[
\widehat{SH_*}(H)\longrightarrow SH_*(H)
\]
is identified with the map dual to 
\[
\pi:SH^*(H)\longrightarrow\widehat{SH^*}(H),
\]
denoted by
\[
\pi^{\vee}:Hom(SH^{2m-*}(H), \Lambda)\longrightarrow Hom(\widehat{SH^{2m-*}}(H), \Lambda).
\]
Under the identifications
\[
\widehat{SH^{2m-*}}(H)\cong\bigslant{HF^{2m-*}(H_0)}{\bigoplus\limits_{ev(\lambda) > k\pi R^2} HF^{2m-*}_{\lambda}(H_0)}\]
and
\[
SH^{2m-*}(H)\cong \bigslant{HF^{2m-*}(H_0)}{HF^{2m-*}_0(H_0)},
\]
$\pi^{\vee}$ is an inclusion 
\[
V\hookrightarrow \bigoplus_{ev(\lambda) < \infty}HF^*_{\lambda}(-H_0),
\]
where $V$ is the subspace of $HF^*(-H_0)$ on which $\bigoplus\limits_{ev(\lambda) > k\pi R^2}HF^{2m-*}_{\lambda}(H_0)$ vanishes.  As in the proof of Lemma \ref{lem:shisdual}, $V$ is precisely
\[
\bigoplus_{ev(\lambda)\leq k\pi R^2}HF^*_{\lambda}(-H_0).
\]

\end{proof}
\\
It remains to compute the connecting map
\[
\boldsymbol{ \widehat{\mathfrak{c}}}:\widehat{SH_*}(H)\longrightarrow\widehat{SH^*}(H).
\]
We have seen that $\boldsymbol{\widehat{\mathfrak{c}}}$ can be identified with the composition
\[
\bigoplus_{ev(\lambda)\leq k\pi (R')^2} HF^*_{\lambda}(-H_0)\hookrightarrow HF^*(-H_0) \xrightarrow{{\bf c}} HF^*(H_0) \twoheadrightarrow \bigslant{HF^*(H_0)}{\bigoplus\limits_{ev(\lambda) > k\pi R^2}HF^*_{\lambda}(H_0)}.
\]
Note, trivially, that the map $\boldsymbol{c}$ is equivalent to the map $\boldsymbol{c}\circ\boldsymbol{\cg{S}}\circ\boldsymbol{\cg{S}}^{-1}$.  The map $\boldsymbol{\cg{S}}^{-1}$ maps each $\lambda$-generalized eigenspace isomorphically onto a $\lambda$-generalized eigenspace.  The map $\boldsymbol{c}\circ\boldsymbol{\cg{S}}$ restricts to an isomorphism on $\bigoplus\limits_{ev(\lambda)\leq k\pi R^2}HF^*_{\lambda}(-H_0)$, as this subspace contains no $0$-generalized eigenvectors.  Thus, the map $\boldsymbol{c}$ restricts to an isomorphism 
\[
\bigoplus\limits_{ev(\lambda)\leq k\pi R^2}HF^*_{\lambda}(-H_0) \rightarrow\bigoplus\limits_{ev(\lambda)\leq k\pi R^2}HF^*_{\lambda}(H_0).
\]
This shows
\begin{lemma}
\label{lem:connecting}
The connecting map $\widehat{\boldsymbol{ \mathfrak{c}}}$ can be identified with the composition
\[
\bigoplus_{ev(\lambda)\leq k\pi (R')^2}HF^*_{\lambda}(-H_0)\xrightarrow{\simeq} \bigoplus_{ev(\lambda)\leq k\pi (R')^2}HF^*_{\lambda}(H_0)\hookrightarrow HF^*(H_0) \twoheadrightarrow \bigslant{HF^*(H_0)}{\bigoplus\limits_{ev(\lambda) > k\pi R^2}HF^*_{\lambda}(H_0)}.
\]
It therefore has image
\[
\im(\boldsymbol{\mathfrak{c}}) = \bigslant{\bigoplus\limits_{ev(\lambda)\leq k\pi (R')^2}HF^*_{\lambda}(H_0)}{\bigoplus\limits_{k\pi R^2 < ev(\lambda) \leq k\pi (R')^2}HF^*_{\lambda}(H_0)}.
\]
and kernel
\[
\ker(\boldsymbol{\mathfrak{c}}) = \bigoplus_{k\pi R^2 < ev(\lambda) \leq k\pi (R')^2}HF^*_{\lambda}(-H_0).
\]
\end{lemma}
We are now ready to prove Theorem \ref{thm:rfh}.

\begin{proofthm8}
If $R' \leq R$, the map $\widehat{\boldsymbol{\mathfrak{c}}}$ is injective by Lemma \ref{lem:connecting}.  The long exact sequence (\ref{eq:les}) and the isomorphism
\[
H(\widehat{SC^*}(H)) \cong \widehat{SH^*}(H)
\]
of Lemma \ref{lem:whentotakeH} show that
\[
\widehat{SH^*}(A_{R', R}) \cong \bigslant{\widehat{SH^*}(H)}{\im(\widehat{\boldsymbol{\mathfrak{c}}})}.
\]
By Lemma \ref{lem:connecting} the right-hand side is
\[
\bigslant{HF^*(H_0)}{\bigoplus\limits_{\substack{ev(\lambda) > k\pi R^2 \\ \text{or} \\ ev(\lambda)\leq k\pi (R')^2 }}HF^*_{\lambda}(H_0)}
\]
Finally, the isomorphism $HF^*(H_0)\cong QH^*(E)$ of Corollary \ref{cor:hftoqh} and the isomorphisms $HF^*_{\lambda}(H_0)\cong QH^*_{\lambda}(E)$ of Corollary \ref{cor:specsmatch} show the first statement in Theorem \ref{thm:rfh}.

If $R' > R$, the map $\widehat{\boldsymbol{\mathfrak{c}}}$ is surjective by Lemma \ref{lem:connecting}.  The long exact sequence (\ref{eq:les}) and the isomorphism
\[
H(\widehat{SC_*}(H)) \cong \widehat{SH_*}(H)
\]
of Lemma \ref{lem:whentotakeH} show that
\[
\widehat{SH^*}(A_{R', R}) \cong \ker(\boldsymbol{\widehat{\mathfrak{c}}}).
\]
By Lemma \ref{lem:connecting} the right-hand side is
\[
\bigoplus\limits_{k\pi R^2 < ev(\lambda) \leq k\pi(R')^2}HF^*_{\lambda}(-H_0).
\]
Dualizing the isomorphism
\[
QH^*(E)\xrightarrow{\simeq}HF^*(H_0)
\]
and applying Poincar\'e duality yields an isomorphism
\[
HF^{*}(-H_0)\xrightarrow{\simeq}QH_*(E).
\]
The first isomorphism intertwines $\rho^*c_1^E\cup_*-$ with the map $\mathbf{c}\circ\boldsymbol{\cg{S}}$.  The second isomorphism intertwines the dual maps: the quantum intersection product $PD(\rho^*c_1^E)\cap_*-$ and $\boldsymbol{\cg{S}}\circ c$.  The result follows.

\end{proofthm8}

Theorem \ref{thm:rfh} explicates a form of self-duality: 
\begin{corollary}
The $\Lambda$-modules $\widehat{SH^*}(A_{R', R})$ and $\widehat{SH^*}(A_{R, R'})$ are dual.
\end{corollary}

\begin{remark}
In \cite{albers-k} Albers-Kang studied the line bundle associated to the prequantization bundle over an integral monotone base.  They showed that the Rabinowitz Floer homology of a circle subbundle of radius $R$ vanishes whenever $R < \frac{1}{\sqrt{\pi\kappa}}$, where $\kappa$ is the monotonicity constant.  Their methods, in conjunction with the work of this section, show that, for $R' < R$ and $E$ a monotone line bundle of negativity constant $k$,
\[
\widehat{SH^*}(A_{R', R}) \simeq \left\{\begin{array}{cc}SH^*(H) & R' < \frac{1}{\sqrt{\pi k \kappa}} \leq R \\ 0 & \text{else}\end{array}\right..
\]
See also \cite{venkatesh-thesis} for an application of the methods in \cite{albers-k} to the case of a toric base.
\end{remark}

\subsection{Closed string mirror symmetry}
\label{hms}
The line bundle $M$ inherits the structure of a toric variety from the base $M$ and the $\CC^*$-action on the fibers.  Its moment polytope $\Delta_E$ can be described in terms of the moment polytope $\Delta_M$ of $M$ (see Subsection 7.6 in \cite{ritter-gromov} or Subsection 12.5 in \cite{ritter-s}).

\begin{example}
Let $E$ be the complex line bundle $\cg{O}(-k) \xlongrightarrow{\rho} \CC P ^m$.  $E$ is a toric variety whose image under the moment map is
\[
\Delta_E := \left\{(v_1, \dots, v_{m+1})\in\RR^{m+1}\hspace{.1cm}\bigg|\hspace{.1cm}v_i\geq 0\hspace{.05cm}\forall\hspace{.05cm} i\in\{1, \dots, m+1\}; \hspace{.1cm} -v_1 - \dots - v_m + kv_{m+1} \geq -1\right\}
\]
The facet of $\Delta_E$ lying in the $\RR^m\times\{0\}$ plane is precisely $\Delta_M$.
\end{example}

$E$ has a conjectural Landau-Ginzberg mirror $(E^{\vee}, W)$, where $E^{\vee} := ev^{-1}(Int(\Delta_E))$, and
\[
W:E^{\vee}\longrightarrow\Lambda
\]
is a superpotential that, to first order, is determined by the toric divisors.  Closed-string mirror symmetry predicts an isomorphism between the symplectic cohomology of $E$ and the Jacobian of $W$:
\begin{equation}
SH^*(E; \Lambda)\cong\bigslant{\Lambda[z_1^{\pm}, z_2^{\pm}, \dots, z_{m+1}^{\pm}]}{(\dd_{z_1}W, \dd_{z_2}W, \dots, \dd_{z_{m+1}}W)} =: Jac({W}).
\label{eq:jac}
\end{equation}

Suppose that $M$ is monotone, so that $c_1^{TM} = \kappa[\omega]$, and suppose $\kappa> k$.  Then $E$ is monotone as well, with monotonicity constant $\kappa - k$.  Suppose further that the superpotential $W$ is Morse, with distinct critical values.  In this case, computations in \cite{ritter-fano} and \cite{ritter-gromov} confirm Equation (\ref{eq:jac}).  

Open-string mirror symmetry was studied in \cite{ritter-s}.  They showed that the moment map $\mu: E\longrightarrow\Delta_E$ has a unique Lagrangian torus fiber $L$ that, when equipped with suitable choices of local systems, split-generates the wrapped Fukaya category.  $L$ sits inside the circle bundle of radius $\frac{1}{\sqrt{\pi k(\kappa-k)}}$.  The fiber that is mirror to $L$, defined by $ev^{-1}\circ\mu(L)$, contains all critical points of $W$.  Open-string mirror symmetry matches each choice of local system with a critical point of $W$.

\begin{remark}
\label{rmk:monotone}
With $E$ monotone, the Chern classes $\rho^*c_1^E = -k[\Omega]$ and $c_1^{TE} = (\kappa-k)[\Omega]$ are related by a constant.  The quantum cup products by $\rho^*c_1^E$ and by $c_1^{TE}$ therefore have the same eigenvalues, up to $\CC^*$-scalar.  Ritter showed in \cite{ritter-fano} that all eigenvalues of $c_1^{TE}\cup_*-$ have valuation in $\{0, \frac{1}{\kappa-k}\}$.  The Floer-essential Lagrangian $L$ therefore sits inside the circle bundle $\Sigma_R$ whose radius satisfies
\[
k\pi R^2 = ev(\lambda)
\]
for some non-zero eigenvalue $\lambda$ of $\rho^*c_1^E\cup_*-$.  This is precisely the critical radius where non-vanishing symplectic cohomology theories occur.  In particular,
\[
\widehat{SH^*}(A_{R_1, R_2}) \neq 0 \iff L\subset A_{R_1, R_2}.
\]
The non-vanishing statement can be seen directly as a consequence of a closed-open map, as expounded upon in \cite{venkatesh}.  The vanishing statement, although seemingly intimately related to the dearth of Floer-essential Lagrangians, does not follow directly.
\end{remark}

Closed-string mirror symmetry generalizes to domains of restricted size.  Let $A_{R_1, R_2}$ be the annulus bundle between radii $R_1$ and $R_2$ in $E$, with $R_1 < R_2$.  The mirror of $A_{R_1, R_2}$ is
\[
A_{R_1, R_2}^{\vee}:= \left\{(z_1, \dots, z_{m+1})\in E^{\vee}\hspace{.1cm}\bigg|\hspace{.1cm} k\pi R_1^2 \leq ev(z_{m+1}) \leq k\pi R_2^2\right\},
\] 
equipped with ${W}\big|_{A_{R_1, R_2}^{\vee}}$.

For $I = (i_1, \dots, i_{m+1})\in\RR^{m+1}$ and ${\bf z} = (z_1, \dots, z_{m+1})$, denote $(z_1^{i_1}, \dots, z_{m+1}^{i_{m+1}})$ by ${\bf z}^{I}$.  We denote the ring of functions on $A_{R_1, R_2}^{\vee}$ in the variable ${\bf z}$ by $\cg{O}(A_{R_1, R_2}^{\vee})_{{\bf z}}$, where
\[
\cg{O}(A_{R_1, R_2}^{\vee})_{{\bf z}} = \left\{\sum_{i=0}^{\infty} c_i{\bf z}^{I_i}\hspace{.1cm}\bigg|\hspace{.1cm} c_i\in\Lambda;\hspace{.1cm} I_i\in\RR^{m+1};\hspace{.1cm} \lim_{i\rightarrow\infty} ev(c_i{\bf z}^{I_i}) = \infty\hspace{.1cm}\forall\hspace{.1cm}{\bf z}\in A_{R_1, R_2}^{\vee}\right\}.
\]
We denote by $Z\left(\dd_{z_i}W\big|_{A_{(R_1, R_2)^{\vee}}}\right)$ the zeroes of the function $\dd_{z_i}W\big|_{A_{(R_1, R_2)^{\vee}}}$. 

The domain  $A_{R_1, R_2}^{\vee}$ is an example of a {\it Laurent domain}: if $\Delta_E$ is described by the functions
\[
\left\{ v\hspace{.2cm}\big|\hspace{.2cm} \la v, n_1^{(i)}\ra \geq \lambda_1, ..., \la v, n_s^{(i)}\ra\geq\lambda_s\right\} 
\]
then $A_{R_1, R_2}^{\vee}$ is cut out by the inequalities
\[
\left\{ ||T^{-\lambda_1}\prod_{i=1}^{m+1} z_i^{n_1^{(i)}}|| \geq 1, ...,  ||T^{-\lambda_s}\prod_{i=1}^{m+1} z_i^{n_s^{(i)}}|| \geq 1, ||T^{-k\pi R_1^2}z_{m+1}|| \geq 1, T^{-k\pi R_2^2}z_{m+1}|| \leq 1\right\},
\]
where
\[
||z|| = e^{-ev(z)}.
\]

Laurent domains are examples of {\it affinoid domains}, and therefore satisfy a Nullstellensatz \cite{tian}.  In particular, 
\[
Jac\left(W\big|_{A_{(R_1, R_2)^{\vee}}}\right) = 0\hspace{.5cm}\text{if and only if}\hspace{.5cm} \bigcap_i Z\left(\dd_{z_i}W\big|_{A_{(R_1, R_2)^{\vee}}}\right) = \emptyset.
\]
This occurs if and only if $Crit(W)\cap A_{R_1, R_2}^{\vee} = \emptyset$, that is, if and only if 
\[
R_2 < \frac{1}{k\pi R^2(k-\kappa)}\hspace{.5cm}\text{or}\hspace{.5cm}R_1 > \frac{1}{k\pi R^2(k-\kappa)}.
\]
Conversely, since $W$ is Morse by assumption, if $Crit(W)\subset A_{R_1, R_2}^{\vee}$ then
\[
Jac\left(W\big|_{A_{(R_1, R_2)^{\vee}}}\right) \simeq \cg{O}\left(Crit(W)\right) \simeq Jac(W).
\]
Altogether,
\[
Jac\left(W\big|_{A_{(R_1, R_2)^{\vee}}}\right)\simeq\left\{\begin{array}{cc} Jac(W) & k\pi R_1^2 < \frac{1}{\kappa - k} < k\pi R_2^2 \\
0 & \text{else}\end{array}\right.
\]
But, as discussed in Remark \ref{rmk:monotone}, $\frac{1}{\kappa - k}$ is exactly the valuation of the non-zero eigenvalues of $c_1^E\cup_*-$.  Thus, Theorem \ref{thm:rfh} says that
\[
\widehat{SH^*}(A_{R_1, R_2}) \simeq \left\{ \begin{array}{cc} SH^*(E) & k\pi R_1^2 < \frac{1}{\kappa - k} < k\pi R_2^2 \\
0 & \text{else}\end{array}\right.
\]
Via the isomorphism $SH^*(E)\cong Jac(W)$ of Equation (\ref{eq:jac}), we conclude a closed-string mirror symmetry statement for subdomains: 

\begin{closedstringms}
If $R_1 < R_2$, then
\[
\widehat{SH^*}(A_{R_1, R_2})\simeq Jac\left(W\big|_{A_{(R_1, R_2)^{\vee}}}\right).
\]
\end{closedstringms}

\begin{example}
Again let $E = \cg{O}(-k)\longrightarrow\CC P^m$.  The mirror of $E$ is
\[
E^{\vee} := \left\{(z_1, \dots, z_{m+1})\in(\Lambda^*)^{m+1}\hspace{.1cm}\big|\hspace{.1cm}\left(ev(z_1), \dots, ev(z_{m+1})\right)\in\Delta^{\mathrm{o}}\right\},
\]
equipped with superpotential
\begin{align}
W\colon E^{\vee}&\longrightarrow\Lambda \\
(z_1, z_2, \dots, z_{m+1}) &\mapsto z_1 + z_2 + \dots + z_m + z_{m+1} + Tz_1^{-1}z_2^{-1}\dots z_m^{-1}z_{m+1}^{k}.
\end{align}
(See Example 7.12 in \cite{ritter-gromov} or Proposition 4.2 in \cite{auroux}.)  Denote by $A$ the one-dimensional annulus defined by the inequalities $k\pi R_1^2 \leq ev(z_{m+1}) \leq k\pi R_2^2$.  A straight-forward computation shows that
\[
Jac(W\big|_{A^{\vee}}) \cong \bigslant{\cg{O}(A^{\vee})}{(1 - (-k)^kTz_{m+1}^{-1 - m + k})}.
\]
If $\pi R_2^2 < \frac{1}{1 + m - k}$, then $ev(Tz_{m+1}^{-1 - m + k}) > 0$ for all $z_{m+1}\in A$.  It follows that $1 - (-k)^kTz_{m+1}^{-1-m+k}$ is a unit in $\cg{O}(A)$, and so 
\[
Jac(W\big|_{A_{R_1, R_2}^{\vee}}
) = 0.  
\]
Similarly, if $\pi R_1^2 > \frac{1}{1 + m - k}$, then $ev(T^{-1}z_{m+1}^{1 + m - k}) > 0$ for all $z_{m+1}\in A$, and so again
\[
Jac(W\big|_{A_{R_1, R_2}^{\vee}}
) = 0.  
\]
If $\pi R_1^2 \leq \frac{1}{1 + m - k} \leq \pi R_2^2$ then 
\[
Jac(W\big|_{A_{R_1, R_2}^{\vee}})\cong\bigslant{\Lambda[z^{\pm}]}{(1 - (-k)^kTz^{-1-m+k}})\cong SH^*(E)\cong\widehat{SH^*}(A_{R_1, R_2}).
\]
\end{example}

\bibliography{mybib}{}

\begin{thebibliography}{10}

\bibitem{abouzaid}
{\sc Abouzaid, M.}
\newblock Symplectic cohomology and {V}iterbo's theorem.
\newblock In {\em Free loop spaces in geometry and topology}. Eur. Math. Soc.,
  2015, p.~271–485.

\bibitem{abouzaid-s}
{\sc Abouzaid, M., and Seidel, P.}
\newblock An open string analogue of {V}iterbo functoriality.
\newblock {\em Geom. Topol. 14}, 2 (2010), 627–718.

\bibitem{albers-k}
{\sc Albers, P., and Kang, J.}
\newblock Vanishing of {R}abinowitz {F}loer homology on negative line bundles.
\newblock {\em Mathematische Zeitschrift\/} (2017), 493--517.

\bibitem{audin-d}
{\sc Audin, M., and Damian, M.}
\newblock {\em Morse theory and {F}loer homology}.
\newblock Springer, London and EDP Sciences, Les Ulis, 2014.

\bibitem{auroux}
{\sc Auroux, D.}
\newblock Fukaya categories of symmetric products and bordered
  {H}eegaard-{F}loer homology.
\newblock {\em J. Gökova Geom. Topol. GGT 4\/} (2010), 1--54.

\bibitem{bourgeois-o}
{\sc Bourgeois, F., and Oancea, A.}
\newblock Symplectic homology, autonomous {H}amiltonians, and {M}orse-{B}ott
  moduli spaces.
\newblock {\em Duke Mathematics}, 1 (2009), 71–174.

\bibitem{c-f-o}
{\sc Cieliebak, K., Frauenfelder, U., and Oancea, A.}
\newblock Rabinowitz {F}loer homology and symplectic homology.
\newblock {\em Ann. Sci. \'Ec. Norm. Sup\'er (4)}, 43(6) (2010), 957–1015.

\bibitem{cieliebak-o}
{\sc Cieliebak, K., and Oancea, A.}
\newblock Symplectic homology and the {E}ilenberg-{S}teenrod axioms.
\newblock {\em Algebr. Geom. Topol. 18}, 4 (2018), 1953–2130.

\bibitem{dai}
{\sc Dai, X.}
\newblock An introduction to ${L}^2$ cohomology.
\newblock In {\em Topology of Stratified Spaces}, vol.~58. MSRI Publications,
  2010.

\bibitem{fooo}
{\sc Fukaya, K., Oh, Y.-G., Ohta, H., and Ono, K.}
\newblock Lagrangian {F}loer theory on compact toric manifolds {I}.
\newblock {\em Duke Math Journal 151}, 1 (2010), 23--174.

\bibitem{groman}
{\sc Groman, Y.}
\newblock Floer theory and reduced cohomology on open manifolds.
\newblock {\em arXiv preprint\/} (2017).
\newblock \href{https://arxiv.org/pdf/1510.04265.pdf}{arXiv.org/1510.04265}.

\bibitem{hofer-s}
{\sc Hofer, H., and Salamon, D.}
\newblock Floer homology and {N}ovikov rings.
\newblock {\em Progress in Mathematics}, The Floer memorial volume (1995),
  483–524.

\bibitem{mcduff-s}
{\sc McDuff, D., and Salamon, D.}
\newblock {\em J-holomorphic curves and symplectic topology, second edition}.
\newblock American Mathematical Society, Providence, RI, 2012.

\bibitem{mclean}
{\sc McLean, M.}
\newblock Birational {C}alabi-{Y}au manifolds have the same small quantum
  products.
\newblock {\em arXiv preprint\/} (2019).
\newblock \href{https://arxiv.org/pdf/1806.01752.pdf}{arXiv.org/1806.01752}.

\bibitem{nelson}
{\sc Nelson, J.}
\newblock Automatic transversality for contact homology ii: filtrations and
  computations.
\newblock {\em Proceedings of the London Mathematical Society 120\/} (2020),
  853--917.

\bibitem{oancea-leray}
{\sc Oancea, A.}
\newblock Fibered symplectic cohomology and {L}eray-{S}erre spectral sequence.
\newblock {\em Journal of symplectic geometry}, 3 (2008), 267--351.

\bibitem{ritter-gromov}
{\sc Ritter, A.}
\newblock Floer theory for negative line bundles via {G}romov-{W}itten
  invariants.
\newblock {\em Adv. Math. 262 (2014)\/} (2014), 1035–1106.

\bibitem{ritter-fano}
{\sc Ritter, A.}
\newblock Circle actions, quantum cohomology, and the {F}ukaya category of
  {F}ano toric varieties.
\newblock {\em Geom. Topol. 20}, 4 (2016), 1941--2052.

\bibitem{ritter-s}
{\sc Ritter, A., and Smith, I.}
\newblock The monotone wrapped {F}ukaya category and the open-closed string
  map.
\newblock {\em Selecta Math. (N.S.) 23}, 1 (2017), 533–642.

\bibitem{seidel-biased}
{\sc Seidel, P.}
\newblock A biased view of symplectic cohomology.
\newblock {\em Current developments in mathematics Volume 2006 (2008)\/}
  (2006), 211–253.

\bibitem{tian}
{\sc Tian, Y.}
\newblock Introduction to rigid geometry.
\newblock Lecture notes for the course Advanced Topics in Algebraic Geometry -
  Introduction to Rigid Geometry at Universit\"at Bonn, 2016.

\bibitem{varolgunes}
{\sc Varolgunes, U.}
\newblock Mayer-{V}ietoris property for relative symplectic cohomology.
\newblock {\em arXiv preprint\/} (2019).
\newblock \href{https://arxiv.org/pdf/1806.00684.pdf}{arXiv.org/1806.00684}.

\bibitem{venkatesh-thesis}
{\sc Venkatesh, S.}
\newblock {\em Completed symplectic cohomology and Liouville cobordisms}.
\newblock PhD thesis, Columbia University, 2018.

\bibitem{venkatesh}
{\sc Venkatesh, S.}
\newblock Rabinowitz {F}loer homology and mirror symmetry.
\newblock {\em Journal of Topology}, No. 1 (2018), 144–179.

\bibitem{viterbo}
{\sc Viterbo, C.}
\newblock Functors and computations in {F}loer homology with applications {I}.
\newblock {\em Geom. Funct. Anal. 9}, 5 (1999), 985--1033.

\end{thebibliography}
\bibliographystyle{plain}

\end{document}